\documentclass[a4paper,11pt]{article}
\usepackage{amsmath}
\usepackage{amsfonts}
\usepackage{titlesec}
\usepackage{amssymb}
\usepackage{amsthm}
\usepackage{enumitem}

\usepackage{geometry}
\usepackage{color}

\usepackage[utf8]{inputenc}

\usepackage{mathtools}

\DeclareMathOperator{\sech}{sech}

\usepackage{array}
\makeatletter
\newcommand{\thickhline}{%
	\noalign {\ifnum 0=`}\fi \hrule height 1pt
	\futurelet \reserved@a \@xhline
}
\newcolumntype{"}{@{\hskip\tabcolsep\vrule width 1pt\hskip\tabcolsep}}
\makeatother

\usepackage{listings}

\newcommand{\ci}{\mathrm{i}} 

\usepackage{framed}
\usepackage{graphicx}
\usepackage{caption, subcaption}
\usepackage{hyperref}
\usepackage{dsfont}
\usepackage{float}

\newtheorem{theorem}{Theorem}[section]
\newtheorem{corollary}{Corollary}[theorem]
\newtheorem{lemma}[theorem]{Lemma}
\newtheorem{remark}[theorem]{Remark}
\newtheorem{conclusion}[theorem]{Conclusion}

\newcommand{\quotes}[1]{``#1''}

\newcommand{\uLOD}{u_{\text{\tiny LOD}}}
\newcommand{\uOD}{u_{\text{\tiny LOD}}}

\newcommand{\VLOD}{V_{\text{\tiny LOD}}}
\newcommand{\VOD}{V_{\text{\tiny LOD}}}
\newcommand{\LOD}{_{\text{\tiny LOD}}}
\newcommand{\R}{\mathbb{R}}
\newcommand{\C}{\mathbb{C}}
\newcommand{\D}{\mathcal{D}}
\newcommand{\V}{V_2}
\newcommand{\Va}{V_1}
\newcommand{\Vb}{V_2}
\newcommand{\supp}{\text{supp}}
\newcommand{\CN}{\text{\tiny CN }}
\newcommand{\eps}{\varepsilon}

\definecolor{myBlue}{RGB}{113,104,238} 
\definecolor{myGreen}{RGB}{154,205,50} 
\definecolor{myGreen2}{RGB}{114,175,30} 
\definecolor{myRed}{RGB}{180,50,50}  
\definecolor{myOrange}{RGB}{225,92,22} 
\definecolor{lgray}{RGB}{200,200,200} 
\definecolor{llgray}{RGB}{155,155,155} 
\definecolor{lila}{rgb}{0.5,0.0,0.5}

\definecolor{mycolor1}{rgb}{0.00000,0.44700,0.74100}%

\usepackage{diagbox}

\renewcommand{\thefootnote}{\fnsymbol{footnote}}
\renewcommand{\thefootnote}{\arabic{footnote}}

\makeatletter
\newcommand\footnoteref[1]{\protected@xdef\@thefnmark{\ref{#1}}\@footnotemark}
\makeatother

\usepackage[]{algorithm2e}
\begin{document}

\begin{center}
{\LARGE Superconvergence of time invariants for the Gross--Pitaevskii equation
\renewcommand{\thefootnote}{\fnsymbol{footnote}}\setcounter{footnote}{0}
 \hspace{-3pt}\footnote{The authors acknowledge the support by the Swedish Research Council (grant 2016-03339) and the G\"oran Gustafsson foundation.}}\\[2em]
\end{center}

\begin{center}
{\large Patrick Henning\footnote[1]{Department of Mathematics, Ruhr-University Bochum, DE-44801 Bochum, Germany.}
\footnote[2]{\label{note2}Department of Mathematics, KTH Royal Institute of Technology, SE-100 44 Stockholm, Sweden.} and Johan W\"arneg{\aa}rd\textsuperscript{\ref{note2}}}\\[2em]
\end{center}

\begin{center}
{\large{\today}}
\end{center}

\begin{center}
\end{center}

\begin{abstract}
This paper considers the numerical treatment of the time-dependent Gross--Pitaevskii equation. In order to conserve the time invariants of the equation as accurately as possible, we propose a Crank--Nicolson-type time discretization that is combined with a suitable generalized finite element discretization in space. The space discretization is based on the technique of Localized Orthogonal Decompositions (LOD) and allows to capture the time invariants with 
an accuracy of order $\mathcal{O}(H^6)$ with respect to the chosen mesh size $H$. This accuracy is preserved due to the conservation properties of the time stepping method. Furthermore, we prove that the resulting scheme approximates the exact solution in the $L^{\infty}(L^2)$-norm with order $\mathcal{O}(\tau^2 + H^4)$, where $\tau$ denotes the step size. The computational efficiency of the method is demonstrated in numerical experiments for a benchmark problem with known exact solution.
\end{abstract}


\section{Introduction}
The so-called Gross--Pitaevskii equation (GPE) is an important model for many physical processes with applications in, for example, optics \cite{Optics,GravityWaves}, fluid dynamics \cite{FluidReview,DeepWater} and, foremost, quantum physics \cite{Gro61,LSY01,Pit61} where it describes the behavior of so-called Bose-Einstein condensates \cite{BaoNum,PiS03}. For a real-valued function $V(x)$ and a constant $\beta \in \R$, the Gross--Pitaevskii equation 
seeks a complex-valued wave function $u(x,t)$ such that 
$$
\ci \hspace{1pt} \partial_t u = - \triangle u + Vu+\beta |u|^2u  
$$
together with an initial condition $u(x,0)=u_0(x)$. In the context of Bose-Einstein condensates, $u$ describes the quantum state of the condensate, $|u|^2$ is its density, $V$ models a magnetic trapping potential and $\beta$ is a parameter that characterizes the strength and the direction of interactions between particles. 

The GPE is known to have physical time invariants where the mass (number of particles) and the energy are the most important ones. When solving the equation numerically it is desirable to conserve these quantities also in the discrete setting. In fact, the choice of conservative schemes over non-conservative schemes can have a tremendous advantage in terms of accuracy. This observation has been confirmed in various numerical experiments (cf. \cite{NLSComparison,Sanz-Serna}). Practically, the discrete conservation of mass and energy is subject to the choice of the time integrator. Among others, mass conservative time discretizations have been studied in \cite{Wang,zouraris_2001}, time integrators that are mass conservative and symplectic are investigated in  \cite{Akrivis1991,HeM17,Sanz-Serna1988,Tou91,Verwer1984},
energy conservative time discretizations in \cite{KaM99}
and time discretization that preserve mass and energy simultaneously are addressed in  \cite{Akrivis1991,BaC12,BaC13,Besse,BDD18,CCW20,NonlinearCN,H1Est,Sanz-SernaNLCN,Zou20}.
For further discretizations we refer to \cite{ANTOINE20132621,BaoNum,Spectral,Lub08,Tha12b} and the references therein. 

Beside the choice of the time integrator that guarantees the conservation of discrete quantities, the space discretization also plays an important role since it determines the accuracy with which invariants can be represented in the numerical method. For example, a low dimensional $P1$ finite element space typically only yields approximations of the energy of order $\mathcal{O}(H)$, where $H$ is the mesh size. Hence, even if the time integrator preserves the discrete energy exactly, there will always be an error of order $\mathcal{O}(H)$. We shall later present a numerical experiment where this plays a tremendous role.

In the light of this issue, we shall investigate the following question: can we find low dimensional spaces (to be used in the numerical scheme for solving the GPE) so that time invariants, such as mass and energy, can be approximated with very high accuracy in these spaces? It is natural that such spaces need to take the problem specific structure into account in order to ensure that they can capture the invariants as accurately as possible. One construction that allows to incorporate features of a differential operator directly into discrete spaces is known as Localized Orthogonal Decomposition (LOD) and was originally proposed by M{\aa}lqvist and Peterseim \cite{LocalElliptic} in the context of elliptic problems with highly oscillatory coefficients. 

The idea of the LOD is to construct a (localizable) orthogonal decomposition of a high dimensional (or infinite dimensional) solution space into a low dimensional space which contains important problem-specific information and a high-dimensional detail space that contains functions that only have a negligible contribution to the solution that shall be approximated. The orthogonality in the construction of the decomposition is with respect to an inner product that is selected based on the differential equation to be solved. After the LOD is constructed, the low dimensional part can be used as a solution space in a Galerkin method. The classical application of this technique are multiscale problems with low regularity, where it is possible to recover linear convergence rates without resolution conditions on the mesh size, i.e., without requiring that the mesh size is small enough to resolve the variations of the multiscale coefficient \cite{HeM14,HeP13,HeMaPe14b,LocalElliptic}. 

The LOD has been successfully applied to numerous differential equations where we exemplarily mention parabolic problems \cite{MaPer17,MaPer18}, hyperbolic problems \cite{AbH17,MaierPet19,PetSch17}, mixed problems \cite{HHM16},
linear elasticity \cite{HePer16}, linear and nonlinear eigenvalue problems \cite{HeMaPe14,MaPet15,MaP17} and Maxwell's and Helmholtz equations \cite{GaHeVe18,GaP15,HePer20,Pet17,Verfuerth2017,VerfOhl17}. The linear Schr\"odinger equation with multiscale potentials was recently addressed in \cite{WuZ21}. An introduction to the methodology is given in \cite{MaP21,Peterseim2016341} and implementation aspects are explained in \cite{ENGWER2019123}.

As opposed to many other multiscale methods, the LOD method greatly improves the accuracy order when applied to single-scale problems with high regularity (cf. \cite{Mai21,MaPet15}). The aim of this paper is to exploit this increase in accuracy to solve challenging and nonlinear time dependent partial differential equations, such as the Gross--Pitaevskii equation, on long time scales. Since the construction of the LOD space is time-independent and linear, its assembly is a one-time overhead that can be done efficiently by solving small linear elliptic problems in parallel.
Besides the construction of a modified Crank--Nicolson (CN) type time integrator that is combined with an LOD space discretization, the novel theoretical contributions of this paper are a proof of superconvergence (of order $6$ with respect to the mesh size of $H$) for time invariants of the GPE in the LOD space, and $L^\infty(L^2)$-convergence rates of order $\mathcal{O}(\tau^2+H^4)$ of the proposed scheme (where $\tau$ is the time step size).
To illustrate the strong performance of our method we present a numerical test case 
that is highly sensitive to energy perturbation and which is therefore very hard to solve on long time scales. Applying the proposed method we are able to easily solve the problem with a resolution on par with a classical $P1$ element space of $2^{21}$ degrees of freedom (i.e., the resolution on which the LOD basis functions are represented) and $2^{24}$ time steps ($\sim10^{13.5}$) on a regular computer. This resolution allows us to capture the correct solution well on long time scales. Solving the problem with standard $P1$ finite elements on meshes with a similar resolution would take months, whereas our computations ran within a few hours with the CN-LOD.

$\\$
{\bf Outline:} In Section \ref{section-LOD} we recall the basic concept of the LOD and we illustrate how superconvergence can be achieved under certain regularity assumptions. In Section \ref{section-GPE-invariants} we introduce the analytical setting of this paper and present important time invariants of the GPE. Superconvergence of the time invariants in the LOD space is afterwards studied in Section \ref{Functionals}. In Sections \ref{section-std-CN-LOD} and \ref{section-mod-CN-LOD} we formulate two versions of the CN-LOD and we present our analytical main results. Details on the implementation are given in Section \ref{section-implementation} and the numerical experiments are presented in Sections \ref{section-numerical-benchmark} and \ref{section-numerics-2D}. Finally, in Section \ref{section-proof-main-result} we prove our main results, which is the major part of this paper.

\section{Localized Orthogonal Decomposition}
\label{section-LOD}

The key to the superconvergence that we shall prove in this paper is due to the choice of a suitable generalized finite element space for discretizing the nonlinear Schr\"odinger equation. The spaces are known as Localized Orthogonal Decomposition (LOD) spaces. In this section we start with a brief introduction to the LOD in a general setting that serves our purposes. Here we recall important results that will be crucial for our error analysis. For further details on the proofs and for results in low-regularity regimes we refer to \cite{HeM14,HeP13,LocalElliptic,Peterseim2016341}.

Throughout this section, we assume that ${\mathcal{D}} \subset \R^d$ (for $d=1,2,3$) is a bounded convex domain with polyhedral boundary. On $\mathcal{D}$, the Sobolev space of complex-valued, weakly differentiable functions with zero trace on $\partial \mathcal{D}$ and $L^2$-integratable partial derivatives is as usual denoted by $H^1_0(\mathcal{D}):=H^1_0(\mathcal{D},\C)$. 
For brevity, we shall denote the $L^2$-norm of a function $v \in L^2(\mathcal{D}):=L^2(\mathcal{D},\C)$ by $\| v \|$. The $L^2$-inner product is denoted by $\langle v, w  \rangle=\int_{\mathcal{D}} v(x) \hspace{2pt} \overline{w(x)} \hspace{2pt} dx$. Here, $\overline{w}$ denotes the complex conjugate of $w$. 

\subsection{Ideal LOD space and approximation properties}
\label{subsection-Ideal_LOD}

Let $a(\cdot,\cdot )$ be an inner product on $H^1_0(\D)$ and let $f \in H^1_0(\D) \cap H^2(\D)$ be a given source term. We consider the problem of finding $u\in H^1_0(\D)$ that solves the variational equation:
\begin{align*}
a(u,v) = \langle f,v\rangle  \qquad  \mbox{for all }
 v\in H^1_0(\D).
\end{align*} 
The problem admits a unique solution by the Riesz representation theorem. The LOD aims at constructing a discrete (low dimensional) space that allows to approximate $u$ with high accuracy. 
For that, we start from a low dimensional (i.e., coarse) space $V_H\subset H^1_0(\D)$, which is given by a standard $P1$ Lagrange finite element space on a quasi-uniform simplicial mesh on $\D$. The mesh size is denoted by $H$ and $\mathcal{T}_H$ is the corresponding simplicial subdivision of $\mathcal{D}$, i.e., $\bigcup_{K\in \mathcal{T}_H} K = \overline{\D} $ (cf. \cite{BrS08}). It is well known that if $u\in H^2(\D)$, then the Galerkin approximation $u_H\in V_H$ of $u$ has an optimal order convergence with
$$
\| u - u_H \| + H \| u - u_H \|_{H^1(\D)} \le C H^2 \| u \|_{H^2(\D)},
$$
for some generic constant $C>0$ that only depends on the regularity of the mesh $\mathcal{T}_H$. It is natural to ask if there is a low dimensional subspace of $H^1_0(\D)$ that has the same dimension as $V_H$, but much better approximation properties. For that, we need to enrich $V_H$ with information from the differential operator. 

In the first step to construct such a space, we consider the $L^2$-projection $P_H : L^2(\D) \rightarrow V_H$, i.e., for $w\in L^2(\D)$ the projection $P_H(w) \in V_H$ fulfills
$$
\langle P_H(w) , v_H \rangle = \langle w , v_H \rangle \qquad \mbox{for all } v_H \in V_H.
$$
On quasi-uniform meshes it can be shown that this $L^2$-projection is actually $H^1$-stable (cf. \cite{BaY14}) and hence the kernel of the $L^2$-projection in $H^1_0(\D)$, i.e.,
$$
W := \ker(P_H) = \{ w \in H^1_0(\D) | \hspace{3pt} P_H(w) = 0 \} ,
$$
is a closed subspace of $H^1_0(\D)$ that we call the {\it detail space}. We immediately have the ideal $L^2$-orthogonal splitting $V_H\oplus W=H^1_0(\D)$. In the next step, we shall modify $V_H$ by enriching it with \quotes{details} (i.e., with functions from $W$).  More precisely, in order to account for problem specific structure while retaining the low dimensionality of the space $V_H$, we introduce the $a(\cdot,\cdot)$-orthogonal complement of the detail space,
\begin{align}
\label{def-OD-space}
\VOD = \{v\in H^1_0(\D)\hspace{2pt} | \hspace{4pt} a(v,w) = 0\mbox{ for all } w \in W \}.
\end{align}
By construction we have $\dim(\VOD) = \dim(V_H):=N_H$ as desired. We now have another ideal splitting of $H^1_0(\D)$ which is of the form $H^1_0(\D) = \VOD \oplus W $, where $ \VOD$ and $W$ are $a(\cdot,\cdot)$-orthogonal.
To quantify the approximation properties of $\VOD$, we denote by $\uOD$ the Ritz projection of $u$ onto $\VOD$, i.e., $\uOD \in \VOD$ is the unique solution to
  \begin{align}
  \label{LOD-var-problem}
a(\uOD,v) = a(u,v) \qquad \mbox{for all } v\in \VOD.
\end{align}
Consequently, $a(\uOD-u,v)=0$ for all $ v\in \VOD$, which allows us to conclude $\uOD-u\in W$ using the $a(\cdot,\cdot)$-orthogonality of $\VOD$ and $W$. The definition of $\VOD$ also entails a useful identity that we shall refer to as the LOD-orthogonality, namely that for any $w\in W$ we have
 \vspace{-\baselineskip}
\begin{align}
\label{LOD-orthogonality}
a(u-\uOD,w) = \langle f,w\rangle.
\end{align} 
A neat consequence of this is that if $f$ has enough regularity then $\|u-\uOD\| \le C \hspace{2pt} H^4$ for some constant $C>0$ that depends on $f$ and the coercivity constant of $a(\cdot,\cdot)$. To see this, recall that $u-\uOD\in W$, wherefore $P_H(u-\uOD) = 0$ by definition of $W$. From this it follows that 
\begin{align} \label{Proj}
\|u-\uOD\| = \|u-\uOD-P_H(u-\uOD)\| \leq C H\|u-\uOD\|_{H^1(\D)},
\end{align}
using the standard approximation properties of the $L^2$-projection $P_H$. If $\alpha>0$ denotes the coercivity constant of $a(\cdot,\cdot)$ then the variational equation \eqref{LOD-orthogonality} gives us
 \begin{align}
\label{LOD-estimate-p1}
\alpha \| u-\uOD\|^2_{H^1(\D)} \leq  a(u-\uOD,u-\uOD) &= \langle f,u-\uOD\rangle .
\end{align}
Using again $u-\uOD \in W$, allows us to play similar tricks on the above right-hand side,
\begin{align*}
\langle f,u-\uOD\rangle&=\langle f-P_H(f),u-\uOD\rangle \\
& = \langle f-P_H(f), u-\uLOD-P_H(u-\uOD)\rangle \\
& \leq C\hspace{2pt} H^2  \|f\|_{H^2(\D)} \hspace{2pt} H \hspace{2pt} \|u-\uOD\|_{H^1(\D)} .
\end{align*}
Note that we used the regularity of $f$, together with standard error estimates for the $L^2$-projection. 
In conclusion we have together with \eqref{LOD-estimate-p1} that 
\begin{align}
\label{H1-conv}
\| u-\uOD\|_{H^1(\D)} \leq C\hspace{2pt} H^3 \|f\|_{H^2(\D)}.
\end{align}
Combining this with \eqref{Proj} results in a $\mathcal{O}(H^4)$-convergence of the $L^2$-error,
\begin{align}\label{L2-conv}
\|u-\uOD\| \leq CH\|u-\uOD\|_{H^1} \leq CH^4.
\end{align}
For improved convergence orders by using higher order finite element spaces for $V_H$, we refer to \cite{Mai21}.

Finally, we note that by construction of $\uOD$ standard energy estimates yield the $H^1$-bound 
\vspace{-\baselineskip}
\begin{align*}
\| \uOD \|_{H^1(\D)} \le C \| f\|,
\end{align*}
for some constant $C>0$ that depends on $\D$ and on the coercivity constant of $a(\cdot,\cdot)$. 

\subsection{Localization of the orthogonal decomposition}
\label{subsection-localization-LOD}

Practically, it is not efficient to work with the full LOD space, $\VOD$, since it has basis functions with a global support. This makes the computation of the basis functions expensive and it leads to dense stiffness matrices in Galerkin discretizations. Fortunately, the basis functions are known to decay exponentially fast outside of small nodal environments, which is why they can be accurately approximated by local functions. In the following we sketch the localization strategy proposed and analyzed in \cite{HeM14,HeP13} in order to approximate the space $\VOD$ efficiently and accurately.

For that, let $\ell \in \mathbb{N}_{>0}$ denote the localization parameter that determines the support of the arising basis functions (which will be of order $\mathcal{O}(\ell H)$). First, we define for any simplex $K \in \mathcal{T}_H$ the corresponding $\ell$-layer patch around $K$ iteratively by 
\begin{align*}
S_\ell(K) & : = \bigcup \{T\in \mathcal{T}_H | \hspace{3pt}T\cap S_{\ell-1}(K) \not= \emptyset \} \qquad \mbox{and} \qquad S_0(K) := K.
\end{align*} 
This means that $S_\ell(K)$ consists of $K$ and $\ell$ layers of grid elements around it. The restriction of $W = \ker(P_H) $ on  $S_\ell(K)$ is given by
$$
W(\hspace{2pt}S_\ell(K)\hspace{2pt}) :=  \{ w \in H^1_0(\hspace{2pt}S_\ell(K)\hspace{2pt}) \hspace{2pt} | \hspace{3pt} P_H(w) = 0 \}  \subset W.
$$
For a given standard (coarse) finite element function $v_H \in V_H$ we can construct a correction so that the corrected function is almost in the $a(\cdot,\cdot)$-orthogonal complement of $W$. This is achieved in the following way. Given $v_H \in V_H$ and $K \in \mathcal{T}_H$ with $K \subset \mbox{\rm supp}(v_H)$ find $Q_{K,\ell} \in W(\hspace{2pt}S_\ell(K)\hspace{2pt})$ such that
\begin{align}
\label{local-LOD-problems}
a( Q_{K,\ell}(v_H) , w ) = - a_K(v_H , w ) \qquad \mbox{for all } w \in W(\hspace{2pt}S_\ell(K)\hspace{2pt}).
\end{align}
Here, $a_K( \cdot , \cdot  )$ is the restriction of $a(\cdot,\cdot)$ on the single element $K$. Since the problem only involves the patch $S_\ell(K)$ it is a local problem and hence cheap to solve. With this, the corrected function is defined by
$$
R_{\ell}(v_H) :=  v_H + \sum_{K \in \mathcal{T}_H}Q_{K,\ell}(v_H).
$$
Practically, $R_{\ell}(v_H)$ is computed for a set of nodal basis functions of $V_H$. We set the {\it localized} orthogonal decomposition space (as an approximation of the ideal space $\VOD$) to 
\begin{align}
\label{LOD-space}
V_{\ell,{\text{\tiny LOD}}} := \{ R_{\ell}(v_H) \hspace{2pt} | \hspace{3pt} v_H \in V_H \}.
\end{align}
Observe that if \quotes{$\ell=\infty$} is so large that $S_\ell(K) = \D$ then we have with \eqref{local-LOD-problems}
$$
a( R_{\infty}(v_H) , w ) = \sum_{K \in \mathcal{T}_H} \left( a_T(v_H , w )  + a( Q_{K,\infty}(v_H) , w )  \right)
= 0 \qquad \mbox{for all } w \in W.
$$
Hence, the functions $R_{\infty}(v_H)$ span indeed the $a(\cdot,\cdot)$-orthogonal complement of $W$, i.e., they span the ideal space $\VOD$. For small values of $\ell$ one might wonder about the approximation properties of $V_{\ell,{\text{\tiny LOD}}} $ compared to $V\LOD$. This question is answered by the following lemma which can be proved analogously to \cite[Conclusion 3.9]{HeM14} together with the ideal higher order estimates \eqref{H1-conv} and \eqref{L2-conv}.
\begin{lemma}\label{lemma-LOD-estimates}
Let the general assumptions of this section hold and assume that $f \in H^1_0(\D) \cap H^2(\D)$. Let the LOD space $V_{\ell,{\text{\tiny LOD}}}$ be given by \eqref{LOD-space} and let $u_{\ell,{\text{\tiny LOD}}} \in V_{\ell,{\text{\tiny LOD}}}$ denote the Galerkin approximation of $u$, i.e., the solution to
$$
a( u_{\ell,{\text{\tiny LOD}}} ,  v ) = \langle f , v \rangle \qquad \mbox{for all } v \in V_{\ell,{\text{\tiny LOD}}}. 
$$
There exits a generic constant $\rho>0$ (that depends on $a(\cdot,\cdot)$, but not on $\ell$ or $H$) such that
\begin{align}
\label{lemma-LOD-estimates-eqn}
\begin{split}
\| u - u_{\ell,{\text{\tiny LOD}}} \| &\le C (H^4 + \exp(-\rho \ell )) \| f\|_{H^2(\D)}
\qquad
\mbox{and}\\
\| u - u_{\ell,{\text{\tiny LOD}}} \|_{H^1(\D)} &\le C (H^3 + \exp(-\rho \ell )) \| f\|_{H^2(\D)}.
\end{split}
\end{align}
Here, the constant $C>0$ can depend on the coercivity and continuity constants of $a(\cdot,\cdot)$ and it can depend on $\D$, but it does not depend on $\ell$, $H$ or $u$ itself.
\end{lemma}
Selecting $\ell \ge 4 |\log(H)| / \rho$ ensures that the optimal convergence rates (of order $\mathcal{O}(H^4)$ for the $L^2$-error and order $\mathcal{O}(H^3)$ for the $H^1$-error) are preserved. Practically $\rho$ is unknown, but it is a common observation in the literature that small values of $\ell$ suffice to obtain an optimal order of accuracy w.r.t. the mesh size $H$ (cf. \cite{HeM14,HeP13}). The same observation is made in our numerical experiments in Section \ref{ConvergenceEnergySection}.

In the following, our error analysis will be carried out in the ideal LOD setting of Section \ref{subsection-Ideal_LOD}, which means that we will not study the influence of the truncation and hence disregard the exponentially decaying error term.

\begin{remark}
The estimates in Lemma \ref{lemma-LOD-estimates} can be refined. For example, the exponentially decaying term will typically only scale with the $L^2$-norm of $f$ and not with the full $H^2$-norm. Furthermore, the decay rate for the $L^2$-error is faster than for $H^1$-error. Since this is not important for our analysis and the application of the results to the Gross--Pitaevskii equation, we decided to only present the more compact estimates \eqref{lemma-LOD-estimates-eqn}.
\end{remark}

For details on the practical implementation of the LOD, we refer to \cite{ENGWER2019123}.

\section{Gross--Pitaevskii equation and time invariants}
\label{section-GPE-invariants}
In this section we present the precise analytical setting of this paper, by introducing the equation and by describing some of its most important features that will come into play in the numerical example in Section \ref{section-numerical-benchmark}.

In the following
\begin{enumerate}[label={(A\arabic*)}]
\item\label{A1} $\D \subset \R^d$, with $d=1,2,3$, denotes a convex polygon which describes the physical domain. 
\item\label{A2} The trapping potential $V \in L^{\infty}(\mathcal{D};\R)$ is real and nonnegative and
\item\label{A3} $\beta\ge 0$ denotes a repulsion parameter that characterizes particle interactions.
\end{enumerate}
Given a final time $T>0$ and an initial value $u^0 \in H^1_0(\D)$, we consider the defocussing Gross--Pitaevskii equation (GPE), which seeks 
$$
u \in L^{\infty}([0,T],H^1_0({\mathcal{D}})) 
\qquad
\mbox{and}
\qquad
\partial_t  u \in L^{\infty}([0,T],H^{-1}({\mathcal{D}}))
$$
such that $u(\cdot,0)=u^0 $ and
\begin{eqnarray}
\label{model-problem}\ci \partial_t u = - \triangle u 
+  V \hspace{1pt} u + \beta |u|^2 \hspace{1pt} u
\end{eqnarray}
in the sense of distributions. The problem is locally well-posed, i.e., for any initial value $u^0 \in H^1_0(\D)$ there exists a time $T>0$ (that can depend on $\| u^0 \|_{H^1(\D)}$) so that the GPE \eqref{model-problem} admits at least one solution. This solution is unique in $1D$ and $2D$. For corresponding proofs we refer to the textbook by Cazenave \cite[Chapter 3]{Cazenave}. To the best of our knowledge, uniqueness in $3D$ is still open in the literature. In 1D, the solution is also global for any initial value (cf.  \cite[Remark 3.5.4]{Cazenave}). In 2D and 3D, the solution can be global for sufficiently small initial values (cf.  \cite[Corollary 3.6.2]{Cazenave} for a corresponding 2D result), however, in the focussing regime, i.e., for $\beta < 0$, or for negative or sign-changing potentials, the solutions are typically no longer global.

For optimal convergence rates in our error analysis we require some additional regularity assumptions. In the following we shall assume that the potential $V$ and the initial value $u^0$ are sufficiently smooth, that is 
\begin{enumerate}[resume,label={(A\arabic*)}]
\item\label{A4} $\quad$ $V \in H^{2}(\mathcal{D};\R)$ and
\item\label{A5} $\quad$ $u^0 \in H^1_0(\D) \cap H^4(\D) \qquad \mbox{with } \triangle u^0 \in H^1_0(\D).$
\end{enumerate}
Observe that the assumption \ref{A5} makes a natural consistency statement that can be either mathematically justified with the structure of the equation, i.e., we have $\triangle u(t)  =  V \hspace{1pt} u(t) + \beta |u|^2 \hspace{1pt} u(t) - \ci \partial_t u(t) \in H^1_0(\D)$ for any sufficiently smooth solution $u$, or it can be physically justified by the typical exponential confinement of trapped Bose-Einstein condensates.

Finally, we also require some regularity for $u$, where we assume that
\begin{enumerate}[resume,label={(A\arabic*)}]
\item\label{A6}  $\quad$ $\partial_t^{(k)} u \in L^2(0,T;H^4(\mathcal{D}) \cap H^1_0(\D)) \quad \mbox{for } 0\le k \le 3$.
\end{enumerate}
In \cite[Lemma 3.1]{NonlinearCN} it was pointed out that any solution that fulfills the above regularity requirements must be unique, which is relevant for the $3D$-case where uniqueness is still open in general. 

The GPE possesses several time invariants of which arguably the two most important ones are the mass (or number of particles) $M$ and the energy $E$, defined by
\begin{align}
&M[u] := \int_\D |u(x,t)|^2 \hspace{2pt} dx,  \label{Mass} \\
& E[u] := \int_\D |\nabla u(x,t)|^2 +V(x)|u(x,t)|^2 + \frac{\beta}{2}|u(x,t)|^4 \hspace{2pt}dx. \label{Energy}
\end{align}
Both quantities are constant in $t$, i.e., they are preserved for all times and in particular we have $M[u^0]=M[u]$ and $E[u^0]=E[u]$. The mass conservation is easily verified by  testing with $u$ in the variational formulation of \eqref{model-problem} and taking the imaginary part afterwards. The energy conservation is seen by testing with $\partial_t u$ instead and then taking the real part afterwards. Formally the latter argument requires $\partial_t u(t) \in H^1_0(\mathcal{D})$ to be rigorous, however, the property still holds without this regularity assumption and can be obtained as a by-product of the existence proof (cf. \cite[Chapter 3]{Cazenave}). 
The momentum, $P$, of $u$ is defined by
\begin{align}\label{Momentum}
P[u] := \int_{\D} 2\Im\big(\overline{u(x,t)} \nabla u(x,t) \big) \hspace{2pt}dx. \hspace{99pt}
\end{align}
Note that the momentum is a vector-valued quantity and that $\Im$ denotes the imaginary part of the expression.  The center of mass $ X_c[u]$ evolves with a velocity that is determined by the momentum, more precisely we have,
\begin{align}\label{Center}
X_c[u] := \int_{\D} x|u(x,t)|^2 dx \quad \text{and}\quad \partial_t X_c[u] =  \hspace{2pt} P[u(t)].  \hspace{25pt}
\end{align}
In particular, if the momentum is vanishing, then the center of mass is conserved. We can test in the variational formulation of \eqref{model-problem} with $\nabla u$ and take the real part to find, provided $u\in H^2_0(\D)$, that over time the momentum changes as:
\begin{align*}
\partial_t P[u](t) = -2\int_{\D} |u(x,t)|^2 \hspace{2pt}\nabla V(x) \hspace{2pt} dx.\hspace{110pt}
\end{align*}
Thus, in the absence of a potential, i.e., for $V(x)=0$, we also have conservation of momentum if $u$ decays sufficiently rapidly near the boundary.

\section{Super-approximation of energy, mass and momentum} 
\label{Functionals}

In the last section we saw that Gross--Pitaevskii equations have important time invariants; it is therefore natural to seek a time discretization that conserves these invariants. However, also the spatial discretization plays a crucial  role here. In fact, in the first step, the given physical initial value has to be projected/interpolated into a finite dimensional (discrete) space. This introduces an error that affects the actual values for the energy, mass and other invariants. Hence, even if a perfectly conservative time stepping method is chosen (up to machine precision), it will also conserve the size of initial discretization errors. Consequently, this limits the accuracy with which the time invariants can be conserved.

In this section we will study this initial discretization error that appears when projecting $u^0$ onto the LOD space introduced in Section \ref{section-LOD}. We will show that the order of accuracy with which the correct values for energy, mass and momentum are conserved, is even higher than what we would expect from the superconvergence results in Lemma \ref{lemma-LOD-estimates}. To be precise, we make the important observation that for the projected initial value in the LOD-space, $u^0\LOD$, functional outputs converge with $6$th order in the mesh size $H$. 
This is a rather surprising upshot as it holds for general classes of nonlinear functionals, and in particular
for all of the above mentioned time invariants. The conservation of (discrete) time invariants itself is then subject to a suitable time integrator, which is the topic of the subsequent section.

In order to be able to apply the abstract results presented in Section \ref{section-LOD}, we first need to decide how to select the inner product $a(\cdot,\cdot)$ in the LOD. For that we split the potential $V$ into two contributions $\Va$ and $\Vb$, so that 
\begin{enumerate}[resume,label={(A\arabic*)}]
\item\label{splitting-potential} $\quad$ $V = \Va+\Vb, \qquad 
\mbox{where } \hspace{5pt} \Va\ge 0; \hspace{10pt} \mbox{and}  \hspace{10pt} \Va, \Vb \in H^2(\D).$
\end{enumerate}
Practically, the splitting is chosen in such a way that $\Vb$ is sufficiently smooth and such that 
\begin{align}
\label{def-a-innerproduct}
a(v,w) := \int_{\D} \nabla v \cdot \overline{\nabla w} + \Va \hspace{2pt} v  \hspace{2pt}  \overline{w} \hspace{2pt} dx
\end{align}
defines an inner product, which hence can be used to construct a corresponding LOD-space.

\begin{remark}[Motivation for $\Va$ and $\Vb$]
\label{remark-structure-aVaVb}
From a computational point of view it makes sense to chose $\Va$ such that the LOD basis functions become (almost) independent of $x$. Looking at the structure of the local problems \eqref{local-LOD-problems} we can see that if $a(\cdot,\cdot)$ has a certain uniform or periodic structure, then it is enough to solve for just a few representative LOD basis functions whereas the remaining basis functions are simply translation of the computed ones. Practically, this avoids a lot of unnecessary computations and hence reduces the CPU time significantly. In terms of physical applications, we make two relevant examples:
\begin{itemize}
\item If $V$ is a harmonic trapping potential of the form $V(x)=\frac{1}{2}\sum_{j=1}^d \gamma_j^2 x_j^2$, with real trapping frequencies $\gamma_j \in \R_{>0}$, a reasonable choice is to select $\Va = 0$ and $\Vb=V$.
\item  Let $V$ be a periodic optical lattice (Kronig-Penney-type potential) of the form 
$$
V(x) = \sum_{j=1}^d \alpha_j \sin\left(\frac{2\pi x_j}{\lambda}\right)^2,
$$
where $\lambda$ is the wavelength of the laser that generates the lattice and where $\alpha_j$ is the amplitude of the potential in direction $x_j$. In this setting we would align the coarse mesh $\mathcal{T}_H$ with an integer multiple of the lattice period $\lambda/2$ and select $\Va=V$ and consequently $\Vb=0$. Typically it is very valuable to incorporate information about the optical lattice directly onto the LOD space $\VLOD$.
\end{itemize}
\end{remark}
With this, we consider the given initial value $u^0 \in H^1_0(\D) \cap H^{2}(\D)$ with $\triangle u^0 \in H^1_0(\D) \cap H^2(\D)$.
Consequently we observe
\begin{align}
\label{def-f0}
f^0 := - \triangle u^0 + \Va \hspace{2pt}u^0 \in H^2(\D) \cap H^1_0(\D).
\end{align}
Hence, we can characterize $u^0 \in H^1_0(\D)$ as the solution to 
$$
a(u^0 , v ) = \langle f^0 , v \rangle \qquad \mbox{for all } v \in H^1_0(\D)
$$
and apply the general results of Section \ref{section-LOD}. In particular, if we define the (ideal) LOD space $\VLOD$ according to \eqref{def-OD-space} and let $u^0\LOD \in \VLOD$ denote the $a(\cdot,\cdot)$-orthogonal projection of $u^0$ into $\VLOD$, i.e.,
\begin{align}
\label{LOD-approx-initial-value}
a( u^0\LOD  , v) = a( u^0 , v) \qquad \mbox{for all } v \in \VLOD,
\end{align}
then the estimates \eqref{H1-conv} and \eqref{L2-conv} apply and we obtain that the initial discretization error in the $L^2$- and $H^1$-norm is
\begin{align*}
\| u^0 - u^0\LOD \| + H \| u^0 - u^0\LOD \|_{H^1(\D)}  &\le C H^4  \|  \triangle u^0 - \Va \hspace{2pt}u^0 \|_{H^2(\D)}.
\end{align*}
In the following we will use the notation $A \lesssim B$, to abbreviate $A \le C B$, where $C$ is a constant that can depend on $u^0$, $u$, $t$, $d$, $\mathcal{D}$, $\Va$, $\Vb$ and $\beta$, but not on the mesh size $H$ or the time step size $\tau$. With this, the estimate can be compactly written as
\begin{align}
\label{LOD-estimate-notation}
\| u^0 - u^0\LOD \| + H \| u^0 - u^0\LOD \|_{H^1(\D)}  &\lesssim H^4.
\end{align}
In the following we shall see that the mass and energy, as well as momentum and center of mass (for $V=0$) are even approximated with $6$th order accuracy with respect to the mesh size $H$. Before we can prove our first main result, we need one lemma.
\begin{lemma}\label{lemma-estimate-particle-interactions}
Assume \ref{A1}-\ref{A5} and \ref{splitting-potential}. Then 
$$
\left| \int_{\D}  |u^0|^4-|u^0\LOD|^4 \hspace{2pt} dx \right| \lesssim H^6.
$$
\end{lemma}
\begin{proof}
We split the error in the following way 
\begin{eqnarray*}
\lefteqn{\int_{\D}  |u^0|^4-|u^0\LOD|^4 \hspace{2pt} dx = \Re \langle (|u^0|^2+|u^0\LOD|^2)(u^0+u^0\LOD),u^0-u^0\LOD\rangle  } \\
&=&\Re \langle (|u^0|^2+|u^0\LOD|^2)(2 u^0+u^0\LOD-u^0),u^0-u^0\LOD\rangle \\
&=& \underbrace{\Re \langle (|u^0|^2+|u^0\LOD|^2)2u^0,u^0-u^0\LOD\rangle}_{\mbox{I}} -\underbrace{ \langle |u^0|^2+|u^0\LOD|^2,|u^0-u^0\LOD|^2\rangle}_{\mbox{II}}.
\end{eqnarray*}
We proceed to bound term II, where we have with the Cauchy-Schwarz inequality
\begin{align*}
\left| \langle |u^0|^2+|u^0\LOD|^2,|u^0-u^0\LOD|^2\rangle \right| \leq  \left( \| u^0 \|_{L^4(\D)}^2+\| u^0\LOD \|_{L^4(\D)}^2 \right) \| u^0- u^0\LOD \|_{L^4(\D)}^2.
\end{align*}
With the Sobolev embedding $H^1(\mathcal{D}) \hookrightarrow L^4(\mathcal{D})$ we conclude that
\begin{align*}
\left| \langle |u^0|^2+|u^0\LOD|^2,|u^0-u^0\LOD|^2\rangle \right| \leq C \left( \| u^0 \|_{L^4(\D)}^2+\| u^0\LOD \|_{H^1(\D)}^2 \right) \| u^0- u^0\LOD \|_{H^1(\D)}^2,
\end{align*}
where $\| u^0\LOD \|_{H^1(\D)} \lesssim \| u^0 \|_{H^1(\D)}$ by stability of the Ritz-projection and where we have $\| u^0- u^0\LOD \|_{H^1(\D)} \lesssim H^3$ by \eqref{LOD-estimate-notation}. In conclusion we have $|\mbox{II}| \lesssim H^6$.\\[0.2em]
Next, we consider $\mbox{I}$, where we split 
\begin{eqnarray*}
\lefteqn{\Re \langle (|u^0|^2+|u^0\LOD|^2)u^0,u^0-u^0\LOD\rangle} \\
&=& \underbrace{2 \Re \langle |u^0|^2 u^0,u^0-u^0\LOD\rangle}_{=:\mbox{II}_1}
+ \underbrace{\Re \langle (|u^0\LOD|^2 - |u^0|^2 )u^0,u^0-u^0\LOD\rangle}_{=:\mbox{II}_2}.
\end{eqnarray*}
For $\mbox{II}_1$ we observe with $|u^0|^2 u^0 \in H^1_0(\D) \cap H^2(\D)$ and the properties of the $L^2$-projection $P_H$ that
\begin{align*}
|\mbox{II}_1| &= 2 \hspace{2pt} |\Re \langle |u^0|^2 u^0 - P_H(|u^0|^2 u^0),u^0-u^0\LOD\rangle| \\
&\le 2 \hspace{3pt} \|  |u^0|^2 u^0 - P_H(|u^0|^2 u^0) \| \hspace{4pt} \| u^0-u^0\LOD \| \\
&\le C \hspace{3pt} H^2 \|  |u^0|^2 u^0 \|_{H^2(\D)}  \hspace{4pt} \| u^0-u^0\LOD \|
\lesssim H^6.
\end{align*}
For $\mbox{II}_2$ we have similarly as for $\mbox{I}$
\begin{align*}
|\mbox{II}_2| &=  |\Re \langle (|u^0\LOD|^2 - |u^0|^2 )u^0,u^0-u^0\LOD\rangle | \\
&\le \| u^0 \|_{L^{\infty}(\D)} \left( \| u^0\LOD \|+ \| u^0 \| \right) \hspace{1pt} \| u^0\LOD - u^0 \|_{L^4(\D)}^2  \\
&\lesssim \| u^0\LOD - u^0 \|_{H^1(\D)}^2 \lesssim H^6.
\end{align*}
Collecting the estimates for $\mbox{I}$, $\mbox{II}_1$ and $\mbox{II}_2$, the result follows.
\end{proof}
With this, we are ready to prove the super-approximation properties for the time invariants in the LOD-space.
\begin{theorem}(6th order convergence of time invariants)
\label{theorem-superapproximation}
Assume \ref{A1}-\ref{A5} and \ref{splitting-potential} and let the LOD-approximation $u^0\LOD \in \VLOD$ of the initial value $u^0$ be given by \eqref{LOD-approx-initial-value}. Then the error in mass can be bounded as
$$
\left| M[u^0] - M[u^0\LOD]  \hspace{2pt}\right| \lesssim H^6
$$
and the initial energy error as
$$
\left| E[u^0] - E[u^0\LOD]  \hspace{2pt}\right| \lesssim H^6.
$$
In the absence of a potential term, i.e., $V=0$, we recall the momentum as another time invariant. We can approximate it with the same order of accuracy as mass and energy, that is
$$
\left| P[u^0] - P[u^0\LOD]  \hspace{2pt}\right| \lesssim H^6.
$$
The same holds for the center of mass in this case, where we have
$$
\left| X_c[u^0] - X_c[u^0\LOD]  \hspace{2pt}\right| \lesssim H^6.
$$
\end{theorem}
\begin{proof}
We start with the convergence for the mass, then we investigate the energy and finally the momentum and the center of mass.\\[0.4em]
{\it Step 1: 6th order convergence of mass.}\\
With the definition of $M$ we have
	\begin{eqnarray*}
		M[u^0] - M[u^0\LOD] &=& \int_{\D} |u^0|^2-|u^0\LOD|^2 dx \\
		&=&\Re\langle u^0+u^0\LOD,u^0-u^0\LOD\rangle \\
		&=&  2 \hspace{2pt} \Re \langle u^0,\underbrace{u^0-u^0\LOD }_{\in W}\rangle - \langle u^0-u^0\LOD,u^0-u^0\LOD \rangle \\
	       &=&  2 \hspace{2pt}  \Re \langle u^0 - P_H(u^0),u^0-u^0\LOD\rangle -\|u^0-\uLOD^0\|^2,
	\end{eqnarray*}
where we recall $P_H : H^1_0(\D) \rightarrow V_H$ as the $L^2$-projection onto the standard FE space, which implies $L^2$-orthogonality of $P_H(u^0)$ and $(u^0-u^0\LOD)$.
From this we gather:
\begin{eqnarray*}
\left| M[u^0] - M[u^0\LOD]  \hspace{2pt}\right| \leq  C (H^6+H^8),
\end{eqnarray*}
for some constant $C$ that depends on the $H^4$-norm of $u^0$ and the $H^2$-norm of $V_1$. This proves the superconvergence for the mass.\\[0.4em]
{\it Step 2: 6th order convergence of energy.}\\
The energy error can be decomposed into
\begin{eqnarray*}
\lefteqn{E[u^0] - E[u^0\LOD]}\\  
&=& \underbrace{a(u^0,u^0) - a(u^0\LOD,u^0\LOD)}_{=: \mbox{I}} + \underbrace{\int_{\D} V_2 \left(  |u^0|^2-|u^0\LOD|^2 \right) dx}_{=: \mbox{II}}
+ \frac{\beta}{2} \underbrace{\int_{\D}  |u^0|^4-|u^0\LOD|^4 \hspace{2pt} dx}_{=: \mbox{III}}.
\end{eqnarray*}
For the first term we have with the definition of $f^0$ in \eqref{def-f0} that 
\begin{align*}
a(u^0,u^0)-a(\uLOD^0,\uLOD^0) & = \langle f^0,u^0\rangle - \langle f^0, \uLOD^0 \rangle  = \langle f^0, \underbrace{u^0-\uLOD^0}_{\in W}\rangle  \\
& = \langle f^0-P_H(f^0), u^0-\uLOD^0 \rangle.
\end{align*}
Since $f^0  \in H^2(\D) \cap H^1_0(\D)$ we conclude with the approximation properties of the $L^2$-projection $P_H$ and together with \eqref{LOD-estimate-notation} that 
$$
|\mbox{I}| \le \| f^0-P_H(f^0) \| \hspace{3pt} \| u^0-\uLOD^0 \| \le C H^2 \| f^0 \|_{H^2(\D)} H^4 \| f^0 \|_{H^2(\D)} \lesssim H^6.
$$
For the second term we observe analogously to the estimate for the mass that
	\begin{eqnarray*}
		 \int_{\D} V_2 (|u^0|^2-|u^0\LOD|^2 ) dx 
	       &=&  2 \hspace{2pt}  \Re \langle V_2 \hspace{2pt}u^0 - P_H(V_2 \hspace{2pt}u^0),u^0-u^0\LOD\rangle -\|\sqrt{V_2} \hspace{2pt} (u^0-\uLOD^0)\|^2.
	\end{eqnarray*}
Since $V_2 \hspace{2pt}u^0\in H^1_0(\D) \cap H^2(\D)$ we have as before
$$
|\mbox{II}| = \left| \int_{\D} V_2 (|u^0|^2-|u^0\LOD|^2 ) \hspace{2pt}dx \right| \lesssim H^6.
$$
For the third term, we can directly apply Lemma \ref{lemma-estimate-particle-interactions} to see $|\mbox{III}|\lesssim H^6$. Combining the estimates for $|\mbox{I}|$, $|\mbox{II}|$ and $|\mbox{III}|$ yields the desired estimate for the energy.\\[0.4em]
{\it Step 3: 6th order convergence of momentum.}\\
We recall the (vector-valued) momentum with $P[v] = 2\int_{\D} \Im\big(\overline{v} \nabla v \big) \hspace{2pt}dx$. Hence it is sufficient to study $\Im\langle u^0,\partial_{x_i} u^0\rangle - \Im\langle u\LOD^0,\partial_{x_i}u\LOD^0 \rangle$ for $1\le i \le d$. We obtain (using Gauss's theorem)
\begin{eqnarray*}
\lefteqn{\Im\langle u^0,\partial_{x_i} u^0\rangle - \Im\langle u\LOD^0,\partial_{x_i} u^0\LOD\rangle = \Im\langle u^0-u\LOD^0,\partial_{x_i}u^0 \rangle + \Im\langle  u\LOD^0,\partial_{x_i}(u^0-u\LOD^0)\rangle} \\
		&=& 2 \hspace{2pt} \Im \langle u^0-u\LOD^0,\partial_{x_i} u^0 \rangle - \Im \langle u^0-u\LOD^0,\partial_{x_i}(u^0-u\LOD^0)\rangle \\
		&=& 2 \hspace{2pt} \Im \langle u^0-u\LOD^0, \partial_{x_i} u^0-P_H(\partial_{x_i} u^0)\rangle - \Im \langle u^0-u\LOD^0,\partial_{x_i}(u^0-u\LOD^0)\rangle.
\end{eqnarray*}
We conclude that
		\begin{eqnarray*}
		\lefteqn{ \left| P[u^0]-P[u\LOD^0] \right|} \\
		&\leq& 4\|u^0-u\LOD^0\|\|\partial_x u^0-P_H(\partial_xu^0)\| +2 \|u^0-u\LOD^0\|\|\nabla (u^0-u\LOD^0)\| \lesssim H^6+H^7.
		\end{eqnarray*}
{\it Step 4: 6th order convergence of center of mass.}\\
Since $X_c[v] = \int_{\D} x\hspace{2pt}|v(x)|^2 \hspace{2pt}dx$, the proof is fully analogous to the estimate of the term $\mbox{II}$ in {\it Step 2}.
\end{proof}

\section{Standard Crank--Nicolson discretization in the LOD space}
\label{section-std-CN-LOD}
We now turn to the fully discrete problem where, as pointed out, conservation of a time invariant is subject to a suitable time discretization. For that we apply a Crank--Nicolson time integrator \cite{Akrivis1991,BaC12,BaC13,NonlinearCN,Sanz-SernaNLCN} that is known to conserve both the discrete mass and the discrete energy exactly for general classes of nonlinear Schr\"odinger equations.

We start with discretizing the considered time interval $[0,T]$ with $N$ time steps. Consequently the time step size is given by $\tau:= T/N$ and we shall denote the discrete time levels by $t_n := n\tau$, where $n = 0,\dots,N$. With this, the classical energy-conservative Crank--Nicolson method applied to the LOD space reads as follows.

Given $u^{\CN 0}\LOD:=u\LOD^0\in \VLOD$ according to \eqref{LOD-approx-initial-value}, find $u^{\CN n+1}\LOD \in \VLOD, \ n=0,\dots,N-1$, such that
	\begin{equation}\label{Crank}
		\ci \big\langle D_\tau u^{\CN n}\LOD,v\big\rangle = 
		\big\langle\nabla u^{\CN n+1/2}\LOD,\nabla v \big\rangle  +\langle Vu^{\CN n+1/2}\LOD,v\rangle + \beta \big\langle \frac{|u^{\CN n+1}\LOD|^2+|u^{\CN n}\LOD|^2}{2}u^{\CN n+1/2}\LOD,v\big\rangle
	\end{equation}
for all $v\in \VLOD$. Here we use the short hand notation
\begin{align*}
D_\tau u^{\CN n}\LOD := \frac{ u^{\CN n+1}\LOD - u^{\CN n}\LOD }{ \tau} \qquad \mbox{and}
\qquad u^{\CN n+1/2}\LOD := \frac{ u^{\CN n+1}\LOD + u^{\CN n}\LOD }{ 2}.
\end{align*}
It is easily seen, by testing with $v=u^{\CN n+1/2}\LOD$ in \eqref{Crank} and taking the imaginary part that the discrete mass is conserved exactly, i.e.,
$$
M[u^{\CN n}\LOD] = M[u^{0}\LOD] \qquad \mbox{for all } n \ge 0.
$$
Together with the super-approximation properties in Theorem \ref{theorem-superapproximation} we hence conclude that $M[u^{\CN n}\LOD]$ will stay close to the exact mass for all times, i.e.,
$$
\left| \hspace{2pt} M[u^{\CN n}\LOD]  - M[u^0] \hspace{2pt} \right| = \mbox{const}\hspace{1pt}\lesssim H^6.
$$
for all $n\ge 0$. Similarly, by testing with $v=D_\tau u^{\CN n}\LOD $ in \eqref{Crank} and taking the real part we see that also the discrete energy is conserved exactly and we have
$$
E[u^{\CN n}\LOD] = E[u^{0}\LOD] \qquad \mbox{for all } n \ge 0.
$$
Theorem \ref{theorem-superapproximation} implies again
$$
\left| \hspace{2pt} E[u^{\CN n}\LOD]  - E[u^0] \hspace{2pt} \right| = \mbox{const}\hspace{1pt}\lesssim H^6.
$$
Due to the nonlinearity in \eqref{Crank} it is not obvious that the scheme is well-posed and always admits a solution. However, we have the following existence result that we shall prove in the appendix for the sake of completeness. 
\begin{lemma}[existence of solutions to the classical Crank--Nicolson method]
\label{lemma-existence-solutions-classiccal-CN}
Assume \ref{A1}-\ref{A3}, then for any $n\ge 1$ there exists at least one solution $u^{\CN n}\LOD \in \VLOD$ to the Crank--Nicolson scheme \eqref{Crank}. 
\end{lemma}
Even though the Crank--Nicolson method \eqref{Crank} is well-posed, conserves the mass and energy and exhibits super-approximation properties it has a severe disadvantage from the computational point of view that is that the repeated assembly of the nonlinear term 
$$
\big\langle \frac{|u^{\CN n+1}\LOD|^2+|u^{\CN n}\LOD|^2}{2}u^{\CN n+1/2}\LOD,v\big\rangle$$ 
in each iteration is extremely costly in the LOD space. We will elaborate more on this drawback in the next section, where we will also propose a modified Crank--Nicolson discretization that overcomes this issue and which can be implemented in an efficient way. 

\section{A modified Crank--Nicolson discretization in the LOD space} 
\label{section-mod-CN-LOD}

In this section we present a modified energy conservative Crank--Nicolson scheme tailored for the LOD-space in terms of computational efficiency. To facilitate reading we again define $u^{n+1/2}\LOD := (u^{n+1}\LOD+u^n\LOD)/2 $ and $D_\tau u^{n}\LOD := (u^{n+1}\LOD - u^{n}\LOD) / \tau$. Furthermore, we let $P\LOD : H^1_0(\D) \rightarrow \VLOD$ denote the $L^2$-projection onto the LOD-space, i.e., for $v\in H^1_0(\D)$ we have that $P\LOD(v) \in \VLOD$ is given by
$$
\langle P\LOD(v) , v\LOD \rangle = \langle v , v\LOD \rangle \qquad \mbox{for all } v\LOD \in \VLOD.
$$
With this, we propose the following variation of the CN-method which allows for a significant speed-up in the LOD-setting while respecting both energy and mass conservation and without affecting convergence rates. The modified method reads:

Given $u\LOD^0\in \VLOD$ according to \eqref{LOD-approx-initial-value}, find $u^{n+1}\LOD \in \VLOD, \ n=0,\dots,N-1$, such that
		\begin{align}\label{FullyDiscrete}
	\ci \big\langle D_\tau u^n\LOD,v\big\rangle = 
	\big\langle\nabla \uLOD^{n+1/2},\nabla v \big\rangle + \big\langle V\uLOD^{n+1/2},v\big\rangle + \beta\langle\frac{P\LOD\big(|u^{n+1}\LOD|^2+ |u^n\LOD|^2\big)}{2}  u^{n+1/2}\LOD,v\rangle
	\end{align}
for all $v\in V\LOD$. Before we start presenting our analytical main results concerning well-posedness of the method, conservation properties and convergence rates, we shall briefly discuss the significant computational difference between \eqref{FullyDiscrete} and the classical formulation \eqref{Crank}.

For that, let $\{\varphi_i\}_{i=1}^{N_H}$ denote the computed basis of $V\LOD$. We compare the algebraic characterizations of the nonlinear terms in \eqref{Crank} and \eqref{FullyDiscrete}, respectively. The speed-up in CPU time is motivated by the large difference in computational work required to assemble the vectors:
	\begin{align*}
	&\mbox{i)} \quad \langle |u\LOD|^2u\LOD,\varphi_l\rangle = \langle \sum_{i,j,k=1}^{N_H} \hspace{2pt} \mathbf{U}_i\overline{\mathbf{U}}_j \mathbf{U}_k \hspace{2pt} \varphi_i \varphi_j \varphi_k ,\varphi_l\rangle \\
	&\mbox{ii)} \quad \langle P\LOD(|u\LOD|^2)u\LOD,\varphi_l\rangle   = \langle \sum_{i,j=1}^{N_H} \boldsymbol{\varrho}_i \mathbf{U}_j \hspace{2pt} \varphi_i \varphi_j,\varphi_l\rangle
	\end{align*}
	where  $\mathbf{U} \in \mathbb{C}^{N_H}$ denotes the vector of nodal values representing the function $\uLOD \in V_{\text{\tiny LOD}}$, i.e., $\uLOD = \sum_{i=1}^{N_H} \mathbf{U}_i \varphi_i$. Likewise, $\boldsymbol{\varrho}\in \mathbb{C}^{N_H}$ represents those of $P\LOD(|u\LOD|^2)$, i.e., $P\LOD(|u\LOD|^2) = \sum_{i=1}^{N_H} \boldsymbol{\varrho}_i \varphi_i$. 
	As an example, consider the 1D case and assume that the support of a basis function $\varphi_i$ is $2(\ell+1)$ coarse simplices, where $\ell \in \mathbb{N}$ is the truncation parameter introduced in Section \ref{subsection-localization-LOD}. Consequently, vector expression i) requires of $\mathcal{O}(\ell^4)$ operations, whereas vector expression ii) requires $\mathcal{O}(\ell^3)$ operations. Moreover, computing the specific projection $P\LOD(|u\LOD|^2)$ can be done efficiently with precomputations that can be reused to compute vector expression ii). Details on the latter aspect are given in the section on implementation, i.e., Section \ref{section-implementation}, where we elaborate more on the efficient realization of the assembly process. A comparison between the classical CN \eqref{Crank} and the modified CN \eqref{FullyDiscrete} in terms of CPU times is later presented in the numerical experiments, where we measured speed-ups by a factor of up to 1200 (cf. Table \ref{CPUs}).\\

The following main results now summarizes the properties of the modified Crank--Nicolson scheme. As we will see, it is well-posed, conserves the mass and a modified energy and we have superconvergence for the $L^{\infty}(L^2)$-error.

\begin{theorem}\label{main-theorem-modified-CN}
Assume \ref{A1}-\ref{splitting-potential} and let $\tau \le \tau_0$ for a sufficiently small parameter $\tau_0>0$ that depends on $u$ and the data functions. Then for every $n\ge 1$ there exists a solution $u^{n}\LOD \in \VLOD$ to the Crank--Nicolson method \eqref{FullyDiscrete} with the following properties: The sequence of solutions is mass-conservative, i.e., for all $n\ge 0$
$$
M[u^{n}\LOD] = M[u^{0}\LOD] \qquad \mbox{where } \quad \left| \hspace{2pt} M[u^{n}\LOD]  - M[u^0] \hspace{2pt} \right| \lesssim H^6.
$$
Furthermore, we have conservation of a modified energy, i.e., for all $n\ge0$
\begin{align*}	
E\LOD[u^{n}\LOD] = E\LOD[u^{0}\LOD] \qquad \mbox{where } \quad 
E\LOD[v]:=\int_{\D} |\nabla v|^2+V|v|^2+ \frac{\beta}{2}|P\LOD(|v|^2)|^2\hspace{2pt}dx. 
\end{align*}
The exact energy is approximated with a $6$th order accuracy, i.e.,
\begin{align*}	
\left| \hspace{2pt}  E[u^{n}\LOD] - E[u^{0}] \hspace{2pt} \right| \lesssim H^6.
\end{align*}
Finally, we also have the following superconvergence result for the $L^2$-error between the exact solution $u$ at time $t_n$ and the CN-LOD approximation $u^{n}\LOD$:
$$
\max_{0 \le n \le N} \| u(\cdot , t_n) -  u^{n}\LOD \| \lesssim \tau^2 + H^4.
$$
\end{theorem}

\begin{remark}
Provided the existence of an analytical solution $u$, discrete solutions $u^{n}\LOD$ and semi-discrete solutions in the sense of Lemma \ref{lemma-semidiscrete-CN} below, the estimates of Theorem \ref{main-theorem-modified-CN} remain valid in the regime $\beta<0$, i.e., when assumption \ref{A3} is dropped.
\end{remark}

Since the proof of Theorem \ref{main-theorem-modified-CN} is extensive and requires several auxiliary results we present it in a separate section. Before that, we discuss some practical aspects of the method, such as its implementation, and we demonstrate its performance for a test problem with known exact solution. The proof of Theorem \ref{main-theorem-modified-CN} follows in Section \ref{section-proof-main-result}.

\section{Implementation}
\label{section-implementation}
In this section we present some implementation details on how to assemble and solve the nonlinear system in an efficient way. Recalling $\varphi_i$ as the LOD basis functions that span the $N_H$-dimensional space $\VLOD$, we introduce short hand notation for the following matrices $\mathbf{M}, \mathbf{A}, \mathbf{M_V} \in \R^{N_H \times N_H}$ and vector $\boldsymbol{U_{\Gamma}} \in \C^{N_H}$:

\begin{align*}
&(\mathbf{M})_{ij} := \langle \varphi_j , \varphi_i \rangle, \qquad  \mathbf{A}_{ij} := \langle \nabla \varphi_j , \nabla \varphi_i \rangle, \qquad (\mathbf{M_V})_{ij} = \langle V \varphi_j,\varphi_i\rangle, \\
& (\boldsymbol{U_{\Gamma}}(\mathbf{U},\mathbf{V}) )_i :=  \beta \Big\langle P\LOD\big(|\sum_{k=1}^{N_H} \varphi_k \mathbf{U}_k|^2+|\sum_{k=1}^{N_H} \varphi_k \mathbf{V}_k|^2\big)(\mathbf{U}+\mathbf{V}),\varphi_i\Big\rangle.
\end{align*} 
Equation \eqref{FullyDiscrete} in matrix-vector form becomes:
\begin{align}\label{NonlinearMatrixEq}
\ci \mathbf{M}\frac{\mathbf{U}^{n+1}-\mathbf{U}^n}{\tau} = \mathbf{A}\frac{\mathbf{U}^{n+1}+ \mathbf{U}^n}{2} + \mathbf{M_V}\frac{\mathbf{U}^{n+1}+\mathbf{U}^n}{2} + \frac{\boldsymbol{U_\Gamma}( \mathbf{U}^n, \mathbf{U}^{n+1})}{4},
\end{align}
where $\mathbf{U}^n \in \C^{N_H}$ is the solution vector in the LOD space, i.e., $u\LOD^n = \sum_{k=1}^{N_H} \mathbf{U}_k^n \varphi_k$.
To solve the nonlinear vector equation \eqref{NonlinearMatrixEq} we propose a fixed point iteration. Let 
$$
\mathbf{L} :=  \mathbf{M}+\tfrac{\ci\tau}{2}( \mathbf{A}+ \mathbf{M_V})$$ 
and $ \mathbf{L}^*$ be its Hermitian adjoint. Our fixed point iteration takes the form:
\begin{equation}\label{FPI}
\mathbf{U}^{n+1}_{m+1} =  \mathbf{L}^{-1} \mathbf{L}^*\mathbf{U}^n - \tfrac{\ci \tau}{4} \mathbf{L}^{-1} \boldsymbol{U_\Gamma} (\mathbf{U}^{n+1}_m, \mathbf{U}^n) \qquad \text{for } m=0,1,2,\dots
\end{equation} 
and $\mathbf{U}^{n+1}_0 = \mathbf{U}^n$.
Here we note that matrix $ \mathbf{L}$ does not change with time. Hence, the above iteration can be done efficiently by precomputing the LU-factorization of $ \mathbf{L}$, which is of size $N_H\times N_H$. However, in each iteration the vector $\boldsymbol{U_\Gamma}$ must be assembled. As a first step we consider the problem of computing $\rho^n=P\LOD(|u^n\LOD|^2)$. By definition we have that $\langle \rho^n,\varphi_i\rangle = \langle |u^n\LOD|^2,\varphi_i\rangle$ for all $\varphi_i\in V\LOD$.
The vector $\langle |u^n\LOD|^2,\varphi_i\rangle$, requires computing the expression
\begin{align*}
\langle |u^n\LOD|^2,\varphi_i\rangle & = \langle \sum_{k=1}^{N_H}\sum_{j=1}^{N_H} \mathbf{U}^n_k\overline{\mathbf{U}}^n_j \varphi_k \varphi_j,\varphi_i\rangle 
= \underset{k\leq j}{\sum_{k,j=1 }^{N_H}}
\Re(\mathbf{U}^n_k \overline{\mathbf{U}}^n_j) \hspace{2pt} (2-\delta_{kj}) \hspace{2pt}
\langle  \varphi_k\varphi_j,\varphi_i\rangle \\
& = \underset{k\leq j}{\sum_{k,j=1 }^{N_H}} \Re(\mathbf{U}^n_k\overline{\mathbf{U}}^n_j) \hspace{2pt} (2-\delta_{kj}) \hspace{2pt} \boldsymbol{\omega}_{kji}, \quad\mbox{where } \boldsymbol{\omega}_{kji}:= \langle  \varphi_k\varphi_j,\varphi_i\rangle.
\end{align*}
The tensor $\boldsymbol{\omega}_{kji}$ is very sparse as it is zero whenever  $\supp(\varphi_i)\cap\supp(\varphi_j)\cap \supp(\varphi_k)\hspace{-2pt} = \emptyset$. More importantly it can be completely precomputed and will, due to the exponential decay of the basis functions, have many approximately zero values since the LOD-basis functions decay exponentially. Therefore setting a tolerance on the entries of $\boldsymbol{\omega}$ can significantly lower the computational cost without loss of accuracy. 
Due to the typically local structure of the basis functions (in the sense of Remark \ref{remark-structure-aVaVb}), computing $\boldsymbol{\omega}_{kji}$ needs only be done for a handful of entries which can be done in parallel. Once $\boldsymbol{\omega}$ is computed it can be reused in the computation of $\boldsymbol{U_\Gamma}$. For example, with $\rho^{n+1/2} := P\LOD(|u^n\LOD|^2+|u^{n+1}\LOD|^2)$ and the representation $\rho^{n+1/2} = \sum_{k=1}^{N_H}  \boldsymbol{\varrho}_k^{n+1/2} \varphi_k$ we have
\begin{align*}
(\boldsymbol{U_\Gamma})_i &= \langle \rho^{n+1/2}u\LOD^{n+1/2}, \varphi_i\rangle  = \sum_{k,j=1}^{N_H} \langle \boldsymbol{\varrho}_k^{n+1/2} \varphi_k  \hspace{4pt} \boldsymbol{U}^{n+1/2}_j\varphi_j,\varphi_i\rangle \\
&= \sum_{k,j=1}^{N_H} \varrho_k^{n+1/2} \boldsymbol{U}^{n+1/2}_j \boldsymbol{\omega}_{kji}.
\end{align*}

\section{Numerical experiments in 1D - a benchmark problem}
\label{section-numerical-benchmark}
In the following we consider a challenging and illustrative numerical experiment that shows the capabilities of our new approach and which can be used as a benchmark problem for future discretizations of the time-dependent GPE. Even though the experiment is only in $1D$ with a known analytical solution, it is extremely hard to solve it numerically. We believe that the formal simplicity of the problem (in terms of its description) makes it very well suited for benchmarking.

The experiment considers the case of two stationary solitons that are interacting with each other and it was first described in \cite{Exact2010} and numerically studied in \cite{NLSComparison}. The combined behavior of the two solitons is characterized as the solution $u$ to the following focusing Gross--Pitaevskii equation with cubic nonlinearity,
  	\begin{align*}
	\ci \partial_t u = - \partial_{xx} u  - 2|u|^2u \qquad \mbox{in } \mathbb{R}\times (0,T]
	\end{align*}
	and with initial value 
	\begin{align*}
	u(x,0)  =  \frac{8(9e^{-4x}+16e^{4x})-32(4e^{-2x}+9e^{2x})}{-128 + 4e^{-6x}+16e^{6x}+81e^{-2x}+64e^{2x} } & .
	\end{align*}
	As derived in \cite{Exact2010}, the exact solution is given by
\begin{align} \label{StationarySoliton}
u(x,t)& =\frac{8e^{4\ci t}(9e^{-4x}+16e^{4x})-32e^{16\ci t}(4e^{-2x}+9e^{2x})}{-128\cos(12t)+4e^{-6x}+16e^{6x}+81e^{-2x}+64e^{2x}}.
\end{align}
The present problem has interesting dynamics, in particular it is very sensitive to energy perturbations. As we will see below, small errors in the energy will be converted into artificial velocities that make the solitons drift apart. 

The exact solution $u$ is depicted in Fig. \ref{fig:Solution} for $0 \le t \le2$ and is best described as two solitons balanced so that neither wanders off. As is readily seen in \eqref{StationarySoliton} the resulting interaction is periodic in time with period $\pi/2$, but the density $|u|^2$ is periodic with period $\pi/6$. 
As for the previous mentioned time invariants we have conservation of all four: mass $M[u]=12$, energy $E[u] = -48$, momentum $P[u] = 0,$ and center of mass $X_c[u] \approx -1.3863$. It is worth mentioning here that despite being analytic, the $L^2$-norm of its spatial derivative of order $n$ grows geometrically with $n$; already for the 9th derivative the size of the $L^2$-norm is of order $10^{11}$. The growth is even more pronounced for its time derivatives as $\|\partial_t^{(6)}u(x,0)\|\approx\mathcal{O}(10^{11})$.
In \cite{NLSComparison} it was noted that for coarse time steps and non energy conservative schemes the numerical solution had a tendency to split into two separate traveling solitons. An example of this is shown in Fig. \ref{Split}, where the converged state w.r.t. $\tau$ at $T=200$, using the standard Crank--Nicolson method on a mesh of size $h=40/16384$, is two separate solitons. Moreover, the popular Strang splitting spectral method of order 2 (SP2 in \cite{Spectral}), failed on long time scales ($T\geq 200$) due to severe blow-up in energy. In fact, in order to solve the equation on long time scales  extreme resolution in space is required, which is why it makes for an excellent test case. We stress again the issue here: even if the chosen time-discretization is perfectly conservative, it will only preserve the discrete quantities. This means any initial error in mass and energy will be preserved for all times and will severely affect the numerical approximation of $u$.

We now turn to the problem of understanding the observed split and quantifying it in terms of the offset in the discrete energy. To this end we make use of the time invariants to determine which configuration of two solitons is consistent with the original problem.  It is well known that the soliton:
\begin{equation} \label{SingleSoliton}
\psi(x,t) = \sqrt{\alpha}e^{\ci(\frac{1}{2}cx-(\frac{1}{4}c^2-\alpha)t)}\sech(\sqrt{\alpha}(x-ct))
\end{equation}
solves $\ci \partial_t \psi = -\partial_{xx}\psi - 2|\psi|^2\psi$, cf. \cite{SingleSoliton}.
Consider the two solitons, call them $\psi_1$ and $\psi_2$, at a time $T$ long after the split. Due to the exponential decay of each soliton we may, to a good approximation, consider them as separate, i.e., $\psi \approx \psi_1+\psi_2$, where each soliton is described according to \eqref{SingleSoliton}. Referring to \eqref{SingleSoliton}, there are 2 degrees of freedom for each soliton namely $\alpha_1,c_1$ and $\alpha_2,c_2$, where $\alpha_i$ is a shape parameter that determines the amplitude $\sqrt{\alpha_i}$ of the soliton and $c_i$ is the velocity with which the soliton moves. Drawing on inspiration from the exponents in \eqref{StationarySoliton}, we conclude that the shape parameters of the separated solitons would be given by $\alpha_1 = 4$ and $\alpha_2 = 16$. 
Consequently we have $\|\psi_1\|^2 = 4$ and $\|\psi_2\|^2 = 8$, which is consistent with the total mass being $\|u\|^2=12$. Since the momentum is conserved it follows from \eqref{Momentum} that if the momentum is non-zero then the center of mass, $X_c[u]$, evolves linearly.  However due to the periodicity of the solution, $X_c[u]$ cannot evolve linearly, we conclude that the momentum $P[u]$ must be 0. Therefore we must also have $c_1 = -2c_2$.  Lastly we determine the velocities from the energy. The energy being translation invariant we may chose a convenient coordinate system to calculate it; let $y$ and $\tilde{y}$ be translations of $x$ such that the solitons are described by
\begin{align*}
\ &\psi_1(y(x),T) = 2e^{\ci\frac{1}{2}c_1y}\sech(2y) =2e^{-\ci c_2y}\sech(2y) \quad \mbox{and}  \\
 &\psi_2(\tilde{y}(x),T) = 4e^{\ci\frac{1}{2}c_2\tilde{y}}\sech(4\tilde{y}).
\end{align*}
 The energy of each soliton is now calculated. We have
\begin{eqnarray*}
&\partial_x \psi_1 = \partial_y \psi_1 = -\ci c_2\psi_1 -  2\tanh(2y)\psi_1, \qquad  & \partial_x \psi_2 = \partial_{\tilde{y}} \psi_2 =  \ci \frac{c_2}{2}\psi_2 -  4\tanh(4\tilde{y})\psi_2,\\ 
&|\partial_x \psi_1|^2 = c_2^2|\psi_1|^2 + 4 |\psi_1|^2\tanh^2(2y),
 \qquad &	|\partial_x \psi_2|^2 = \frac{c_2^2}{4}|\psi_2|^2 + 16|\psi_2|^2\tanh^2(4\tilde{y}).
\end{eqnarray*}
Thus, 
\begin{align*}
E[\psi_1] &=\int_{\D} |\partial_x\psi_1|^2 - |\psi_1|^4 dx = \frac{16}{3} + c_2^2\|\psi_1\|^2 - \frac{32}{3} = c_2^2\|\psi_1\|^2 -5-1/3,\\
E[\psi_2] &= \int_{\D} |\partial_x\psi_2|^2 - |\psi_2|^4 dx=\frac{128}{3} + \frac{c_2^2}{4}\|\psi_2\|^2 - \frac{256}{3} = \frac{c_2^2}{4}\|\psi_2\|^2-42-2/3.
\end{align*}
Again owing to the separation and the exponential decay it holds approximately $E[\psi_1+\psi_2] = E[\psi_1]+E[\psi_2] =-48+c_2^2\|\psi_1\|^2+c_2^2\|\psi_2\|^2/4$. For complete consistency with the original problem we must have $c_2=0$. However the energy of the discretized problem will not be exactly -48, in fact in turns out that it will be slightly higher. We are thus lead to ponder, what happens if all this extra energy contributes to velocities of the solitons? Denote the error in energy by $\epsilon_h$, i.e., $E[u^0_h]+48 = \epsilon_h$. 
Suppose all of this extra energy is contributing to the velocities, then $\epsilon_h = 4c_2^2  + 2c_2^2 = 6c_2^2$ and we conclude 
\begin{align*}
 |c_2| & = \sqrt{\frac{\epsilon_h}{6}}  \qquad \mbox{and with $c_1 = -2 c_2$ that} \qquad
 |c_1| = \sqrt{\frac{2\epsilon_h}{3}}.
\end{align*}
 If the quantity $T\sqrt{\epsilon_h}$ is not small the error will be of $\mathcal{O}(1)$ as the converged result w.r.t. $\tau$ will be two separate solitons with velocity $\propto \sqrt{\epsilon_h}$. Note however, that this analysis does not say when the split occurs. 

Due to the exponential decay, we restrict our computations to a finite computational domain of size $[-20,20]\times(0,T]$ and prescribe homogenous Dirichlet boundary conditions on both ends of the spatial interval. The results are divided into 4 parts: first we confirm the 6th order convergence rates of the energy of the initial value derived in Section \ref{Functionals}, next we confirm the optimal convergence rates on a short time scale, in Section \ref{T200} we present plots for $T=200$ confirming the analysis of the split completed with convergence rates.

\begin{figure}[H]
    \centering
    \begin{subfigure}[b]{0.35\textwidth}
        \includegraphics[width=\textwidth]{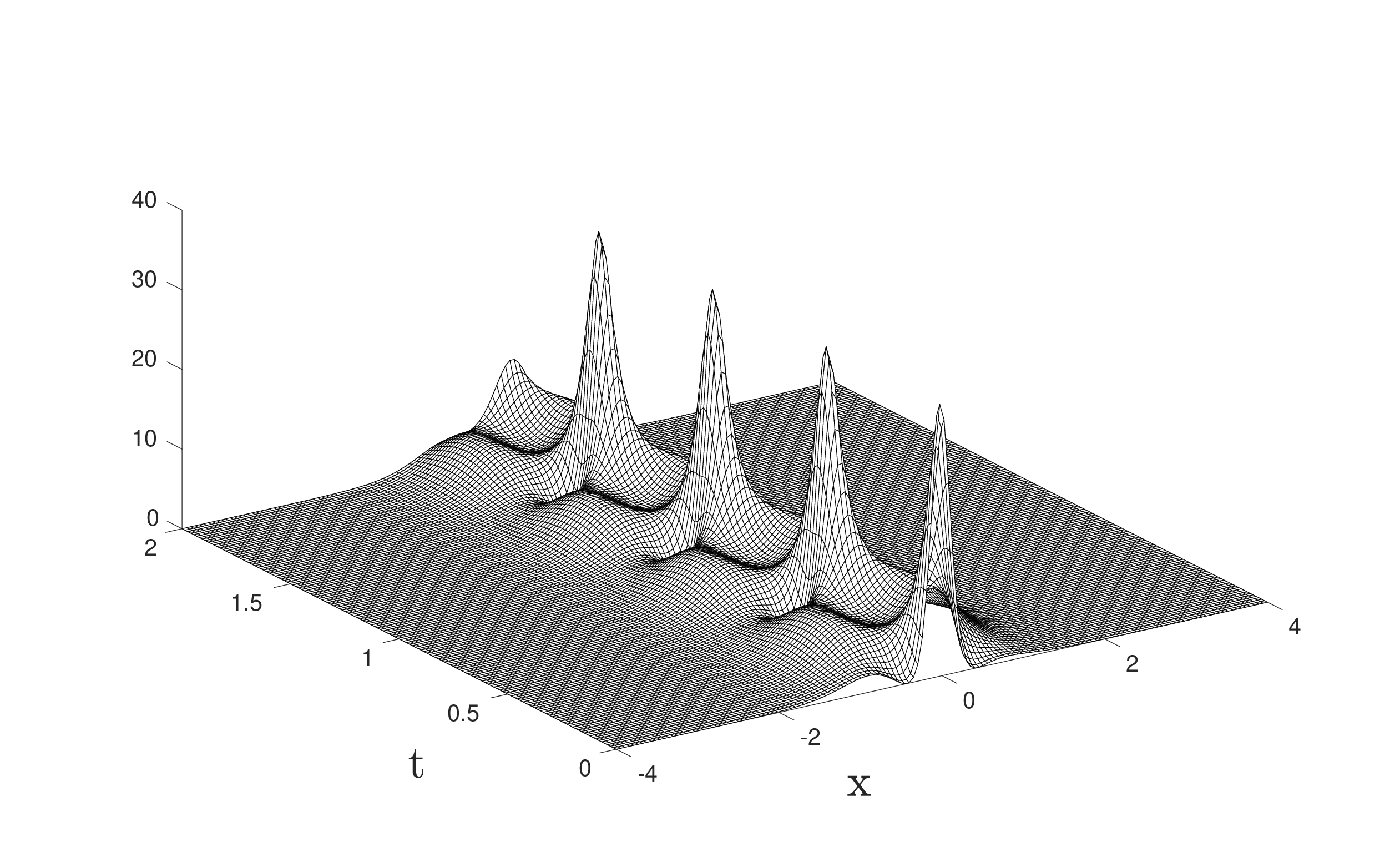}
        \caption{$|u|^2$}
    \end{subfigure}\hspace{-25pt}
    ~ 
    \begin{subfigure}[b]{.35\textwidth}
        \includegraphics[width=\textwidth]{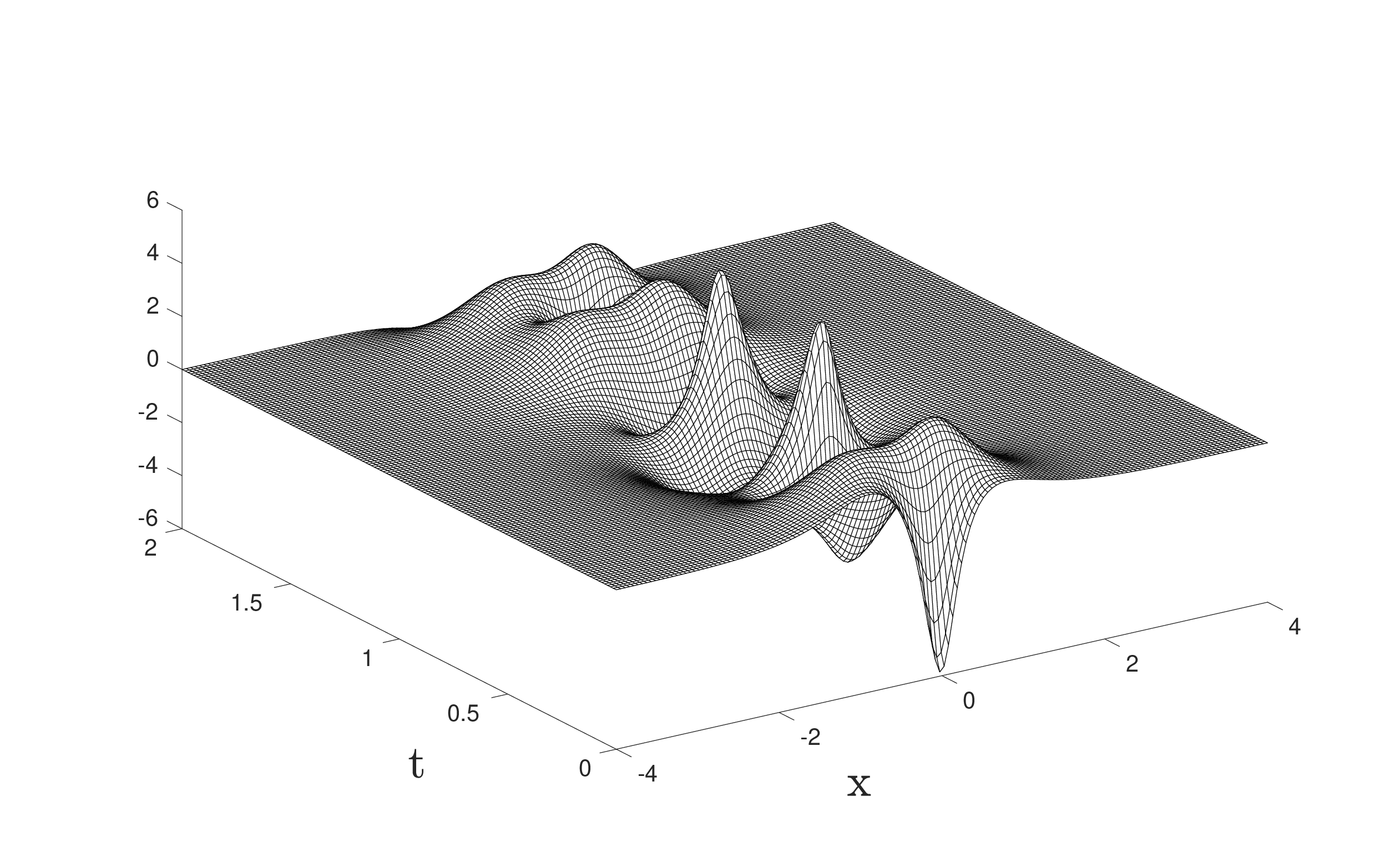}
        \caption{$\Re u$}
    \end{subfigure}\hspace{-25pt}
    ~ 
    \begin{subfigure}[b]{0.35\textwidth}
        \includegraphics[width=\textwidth]{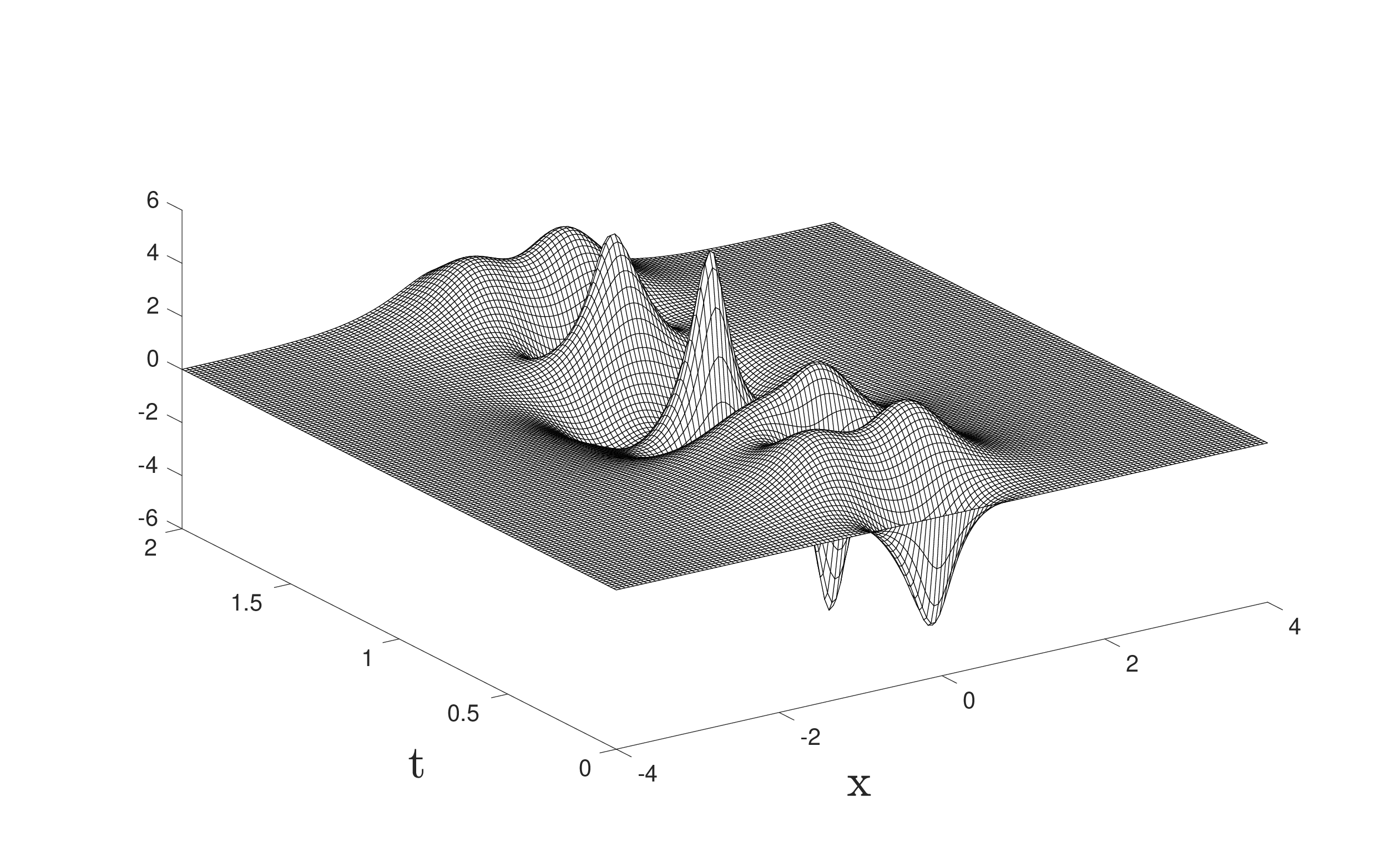}
        \caption{$\Im u$}
    \end{subfigure}
    \caption{Figures (a)-(c) show the time periodic solution to the two stationary soliton test case.}
\label{fig:Solution}
\end{figure}

\subsection{Localization of basis functions}
\label{subsection-numexp-loc-basis}
For computational purposes it is important to localize the basis functions of the LOD space according to the descriptions in Section \ref{subsection-localization-LOD}, where $\ell$ denotes the truncation parameter that characterizes the diameter of the support of the basis functions (which is of order $\mathcal{O}(\ell H)$). The local linear elliptic problems \eqref{local-LOD-problems} that need to be solved to construct the basis functions are discretized with standard $P1$-FEM on a fixed fine mesh of size $h=40/2^{21}$ in most of our experiments (except for the comparison experiments in Section \ref{ConvergenceEnergySection}, where we investigate the influence of $h$ and the first set of experiments in Table \ref{CPUs}). Note that there are only $\mathcal{O}(\ell)$ local problems that have to be solved for and the remaining basis functions are obtained through translations and reflections. The total CPU times stated in this paper for LOD-based methods include the time for computing the corresponding basis functions.

For a better distinction, we shall in the following refer to $H$ as the coarse mesh size (as it determines the dimension of the LOD space) and $h$ as the fine mesh size which  limits the numerical resolution with which the LOD-basis functions are represented.

\subsection{Convergence of energy}\label{ConvergenceEnergySection}

In this experiment, the energy is calculated for different coarse meshes of sizes $H$ and the number of coarse layer patches is fixed to  $\ell = 12$ corresponding to a sufficiently accurate approximation of the ideal global basis functions. For comparison, we also show the influence of the fine mesh size $h$ on which we represent the LOD basis functions. The 6th order convergence of the energy predicted by Theorem \ref{theorem-superapproximation} is confirmed in Fig. \ref{Energy_Conv}. In Fig. \ref{Decay} we show the influence of the truncation parameter $\ell$ on the energy for different discretizations.  We observe that only a small number of layers is needed to capture the full potential of the LOD-basis functions, e.g., $\ell = 5 \approx 2 |\log(H)|$ suffices for $H=40/2^9$ (i.e., $N_H = 512$). The figure also clearly shows the logarithmic relationship between the mesh size and optimal values for $\ell$.

	\begin{figure}[H]
\begin{subfigure}[t]{0.45\textwidth}
\includegraphics[scale=0.5]{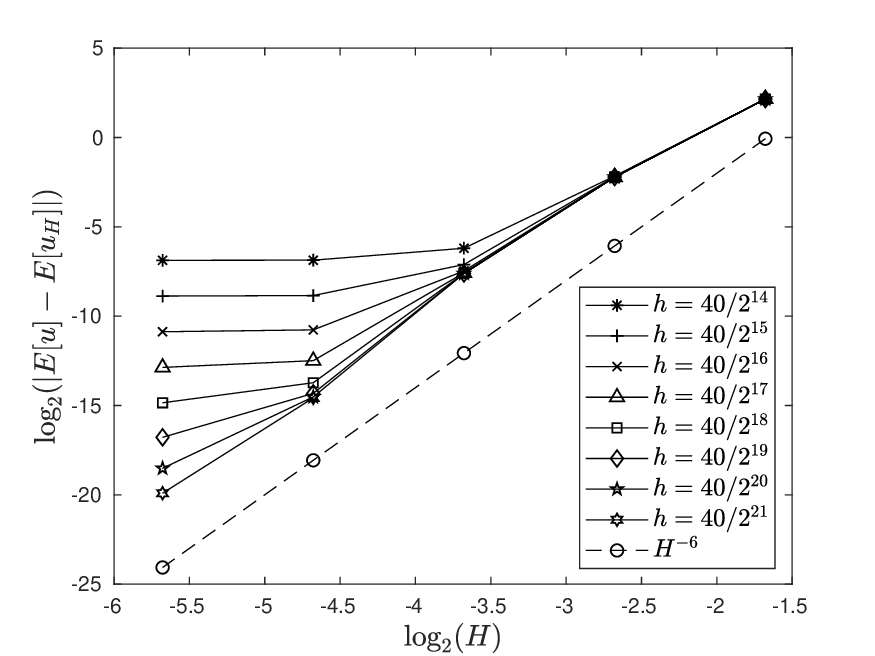} 
\caption{ 6th order convergence of the energy, $\ell = 12$ (approximately global basis functions).} \label{Energy_Conv}
		\end{subfigure}
		~
		\begin{subfigure}[t]{0.45\textwidth}
		\centering
        \includegraphics[scale=0.5]{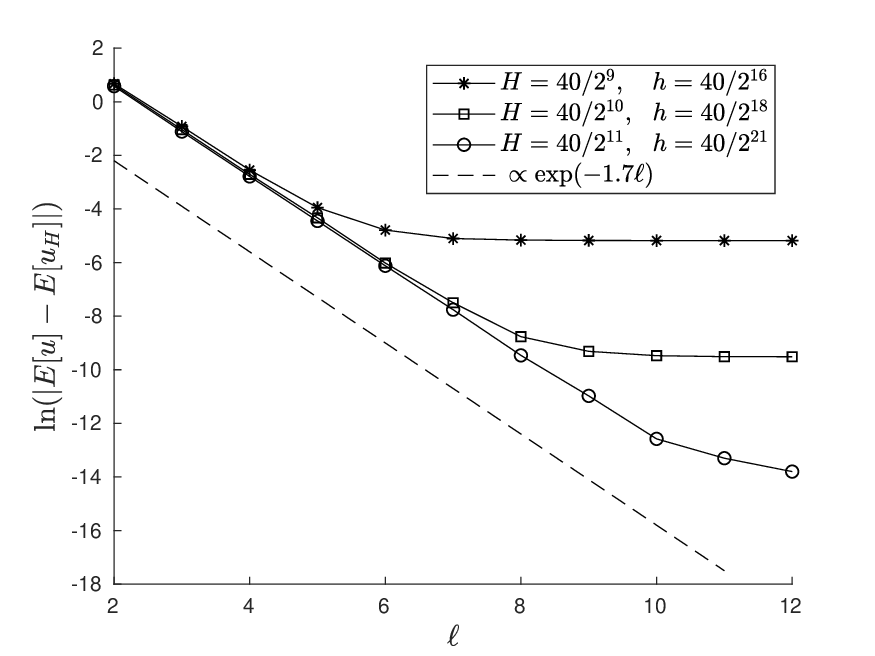}
        \caption{Error in energy versus localization of basis functions, $\ell$.  }
			\label{Decay}
		\end{subfigure}%
		\caption{Influence of the fine mesh size $h$ and localization parameter $\ell$.}
\end{figure}

\subsection{Short time $L^{\infty}(L^2)$- and $L^{\infty}(H^1)$-convergence rates for $T=2$}
Again we study the convergence rates in $H$, 
where the LOD space is computed as described in Section \ref{subsection-numexp-loc-basis} with fixed truncation parameter $\ell=12$.
The final time is set to  $T= 2$ and the number of time steps is set to $N = 2^{18}$ in order to isolate the influence of $H$.
\begin{table}[H]
\centering
 \begin{tabular}{ c c c c c}
$H$ & $\|u-u\LOD\|/\|u\|$ & $\frac{\|u-u_{\text{\tiny LOD},H}\|}{\|u-u_{\text{\tiny LOD},H/2}\|}$   & $\log_2\big(\frac{\|u-u_{\text{\tiny LOD},H}\|}{\|u-u_{\text{\tiny LOD},H/2}\|}\big)$  & CPU [h]\\ \hline
 $40/2^8$ &  1.853126	 & 199 & 7.6 & 0.4 \\
$40/2^9$ &   0.009330 &  225 & 7.8 & 0.6 \\
$40/2^{10}$& 0.000042&  42 & 5.4 & 1.0 \\
$40/2^{11}$& 0.000001&  & & 2.0 \\ 
& & & & \\
$H$ & $\|\nabla(u-u\LOD)\|/\|\nabla u\|$ & $\frac{\|\nabla(u-u_{\text{\tiny LOD},H})\|}{\|\nabla(u-u_{\text{\tiny LOD},H/2})\|}$   & $\log_2\big(\frac{\|\nabla(u-u_{\text{\tiny LOD},H})\|}{\|\nabla(u-u_{\text{\tiny LOD},H/2})\|}\big)$ & CPU [h] \\ \hline
$40/2^8$   &1.734525 & 116 & 6.9 & 0.4 \\ 
$40/2^9$   &0.014931 & 182 & 7.5 & 0.6\\ 
$40/2^{10}$&0.000082 & 7.5 & 2.9 & 1.0 \\ 
$40/2^{11}$&0.000011 & & & 2.0
\end{tabular}
\caption{Error table over varying $H$ for final time $T=2$, truncation parameter is $\ell = 12$ and the number of time steps is $N = 2^{18}$.}\label{Table_ShortTime}
\end{table}
In Table \ref{Table_ShortTime} we observe that the rate of convergence in the $L^{\infty}(L^2)$-norm is initially higher than predicted but seems to flatten out to the expected $\mathcal{O}(H^4)$. A similar observation is made for the error in $L^{\infty}(H^1)$-norm, where we observe asymptotically a convergence rate of order $\mathcal{O}(H^3)$.

\subsection{CPU times}

In Table \ref{CPUs} we make a comparison between different implementations of the Crank--Nicolson method in terms of CPU time per time step. The computations were performed on an Intel Core i7-6700 CPU with 3.40GHz$\times$8 processor. The CN-FEM refers to the solution to \eqref{Crank} in a standard $P1$ Lagrange finite element space on a quasi-uniform mesh with fine mesh size $h$ (which is the same mesh size on which the LOD basis functions are computed). Hence, the methods in the comparison have the same numerical resolution. The nonlinear equation that has to be solved in each time step was either solved by Newton's method or by the fixed point iteration of the form \eqref{FPI}. The respective schemes are accordingly indicated by CN-FEM Newton and CN-FEM FPI in Table \ref{CPUs}. To make the comparison fair we discretize the CN-FEM schemes using the mesh on which the LOD-basis is represented and choose $N_H$ and $\ell$ so large that the energy is represented with equal precision by the methods. We stress that we did not observe any dependency of the number of fixed point iterations on the mesh size. The number may, however, increase with larger $\beta$ and decrease with smaller time step sizes, $\tau$. For this example the stopping criterion was set to $\| u^{n+1}_{\text{\tiny{LOD}},i+1}- u^{n+1}_{\text{\tiny{LOD}},i}\| \leq 10^{-10}$. The speed-up of CN-LOD compared to CN-FEM ranges from 500 to 1200. Some of the computations in the next subsection required a day or two thereby putting them completely out of reach of the Crank--Nicolson method with classical $P1$ finite element spaces. 
\begin{table}[H] 
\centering 
\begin{tabular}{c|c|c|c|c|}
\multicolumn{5}{c}{One time step with $N_H=1024,\ h = 40/2^{18}$ and step size $\tau = 200/2^{21}$}\\
 & CN-FEM Newton & CN-FEM FPI  &    CN-FEM LOD $\ell = 7$ & CN-FEM LOD $\ell = 10$ \\ \hline
CPU [s] & 4.5 & 2 & 0.0095 & 0.014 \\ \hline
$E-E_h$ & 3.33e-5 & 3.33e-5 & 5.5e-4 &  7.7e-5 \\ \hline 
N\textsuperscript{\underline{o}} it. & 3 & 5 &5 &5 \\ \hline
\end{tabular} 
\\
\begin{tabular}{c|c|c|c|c|}
\multicolumn{5}{c}{One time step with  $N_H=2048,\ h=40/2^{21}$, and step size $\tau = 200/2^{21}$}\\
 & CN-FEM Newton & CN-FEM FPI &     CN-FEM LOD $\ell= 10$ & CN-FEM LOD $\ell = 12$ \\ \hline
CPU [s] & 36 & 15.9 &     0.029& 0.032 \\ \hline
$E-E_h$ & 5.2e-7 & 5.2e-7 & 3.3e-6 & 9.7e-7 \\ \hline 
N\textsuperscript{\underline{o}} it. & 3 & 5 &5 &5 \\ \hline
\end{tabular}
\caption{CPU times in seconds for some different approaches to solving the nonlinear system of equations arising from the Crank--Nicolson discretization of the stationary soliton problem. The CN-FEM refers to the classical Crank--Nicolson finite element method \eqref{Crank} on the fine grid $h$. The stopping criterion was set to $\| u^{n+1}_{\text{\tiny{LOD}},i+1}- u^{n+1}_{\text{\tiny{LOD}},i}\| \leq 10^{-10}$.}\label{CPUs}
\end{table}
The precomputations for this example are completely negligible as only $\mathcal{O}(\ell)$ local problems need to be solved for all interior basis functions. For example, consider the finest discretization in this paper for which $2^{11}$ LOD-basis functions are represented on a fine grid of dimension $2^{21}$, for this discretization solving the  linear system of equations that gives the interior basis functions by means of a direct solver such as LAPACK requires only 0.04 seconds. Computing the tensor $\boldsymbol{\omega}$, described in Section \ref{section-implementation}, requires for the very same discretization around one minute. As the space is low dimensional the LU-factorization of the matrix $\mathbf{L}$ requires only a few seconds even for the finest discretization with the LOD space of size $N_H = 2^{11}$.

\subsection{Long time $L^{\infty}(L^2)$- and $L^{\infty}(H^1)$-convergence rates  for $T=200$} \label{T200}
As previously described in this section, an error in the energy produces, for large final computational times, a highly noticeable drift that can only be remedied by increasing spatial resolution. In Figures \ref{Split} through \ref{2048} we illustrate how the split into two separate solitons diminishes as the spatial resolution is increased for final time $T=200$. Fig. \ref{Split} shows the converged solution w.r.t. $\tau$ of the classical Crank--Nicolson method (i.e., even smaller time steps will not improve the approximation). We observe that the solution is fully off in this case for a classical finite element space of dimension 16\hspace{2pt}384. In Figures \ref{1024} and \ref{2048} we can see the numerical approximation in LOD spaces of dimension $N_H=1024$ and $N_H=2048$. We observe that $u\LOD$ captures the correct long time behavior, where for $N_H=2048$ it is no longer distinguishable from the analytical reference solution.

\begin{figure}[H]
    \centering
    \begin{subfigure}[b]{0.33\textwidth}
        \includegraphics[width=\textwidth]{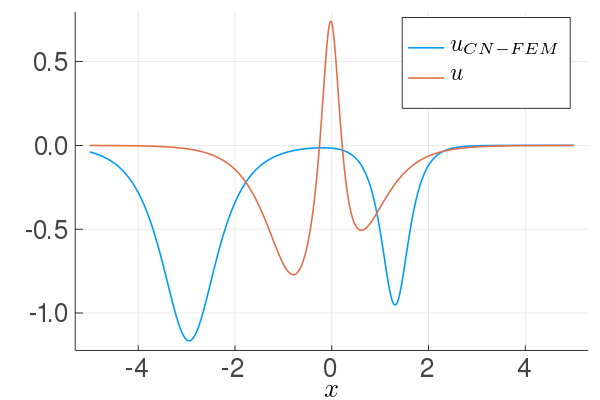}
        \caption{$\Re u$}
    \end{subfigure}\hspace{-10pt}
    \begin{subfigure}[b]{0.33\textwidth}
        \includegraphics[width=\textwidth]{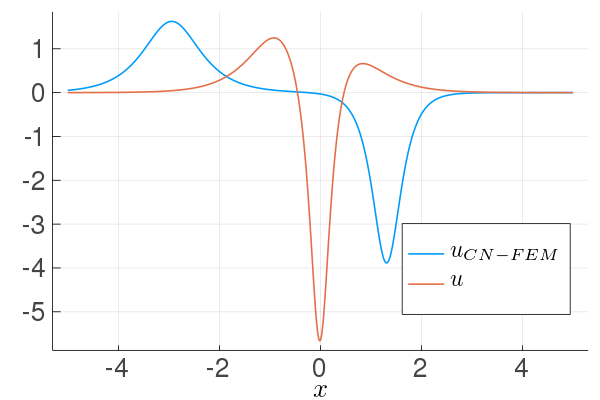}
        \caption{$\Im u$}
    \end{subfigure}\hspace{-10pt}
    \begin{subfigure}[b]{0.33\textwidth}
        \includegraphics[width=\textwidth]{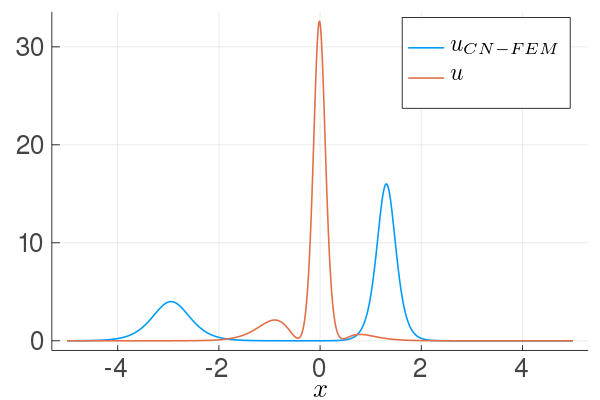}
        \caption{$|u|^2$}
    \end{subfigure}
\caption{Converged solution $u_h$ w.r.t. $\tau$ at $T=200$ of CN-FEM FPI using $h=40/2^{14}$ and $N=2^{21}$ time steps. We have $E[u_h]=-47.9914743 $, $c_1 =0.075$ (drift velocity of left going soliton as estimated in Section \ref{section-numerical-benchmark}). The relative $L^2$ and $H^1$ errors are   $\|u_h-u\|/\|u\|=1.447$ and $\|\nabla(u-u_h)\|/\|\nabla u\| $ = 1.148. The  required CPU time was 28h.}\label{Split} 
\end{figure}

\begin{figure}[H]
    \centering
    \begin{subfigure}[b]{0.33\textwidth}
        \includegraphics[width=\textwidth]{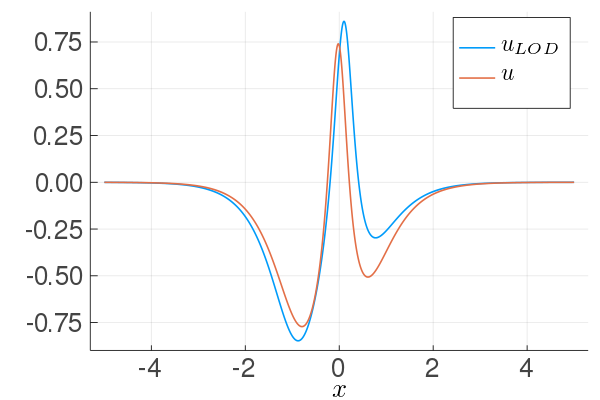}
        \caption{$\Re u $}
    \end{subfigure}\hspace{-10pt}
    \begin{subfigure}[b]{0.33\textwidth}
        \includegraphics[width=\textwidth]{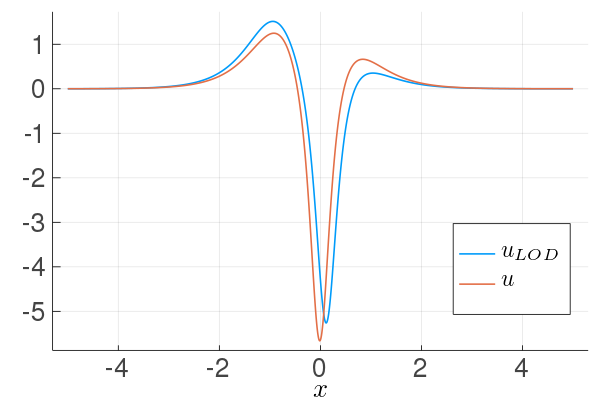}
        \caption{$\Im u$}
    \end{subfigure}\hspace{-10pt}
    \begin{subfigure}[b]{0.33\textwidth}
        \includegraphics[width=\textwidth]{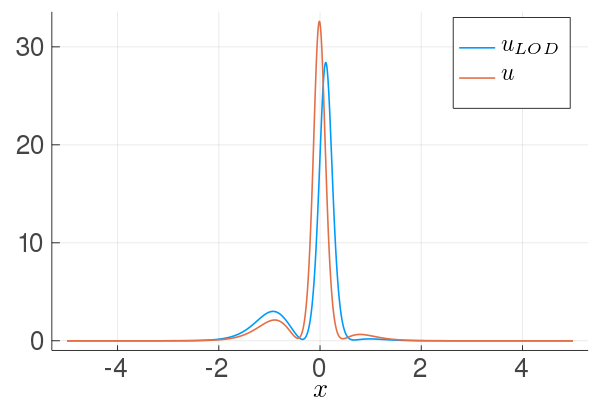}
        \caption{$|u|^2$}
    \end{subfigure}
\caption{Converged solution $u\LOD$ w.r.t. $\tau$ at $T=200$ of CN-FEM LOD using $ H = 40/2^{10} ,\ \ell = 10$ and $N=2^{23}$ time steps. We have $E[u\LOD]=-47.99992458 $, $c_1 =0.0088 $.  The relative $L^2$ and $H^1$ errors are   $\|u\LOD-u\|/\|u\|=0.663$ and $\|\nabla(u-u\LOD)\|/\|\nabla u\LOD\|= 0.718$. The required CPU time was 29h.
} \label{1024}
\end{figure}

\begin{figure}[H]
    \centering
    \begin{subfigure}[b]{0.33\textwidth}
        \includegraphics[width=\textwidth]{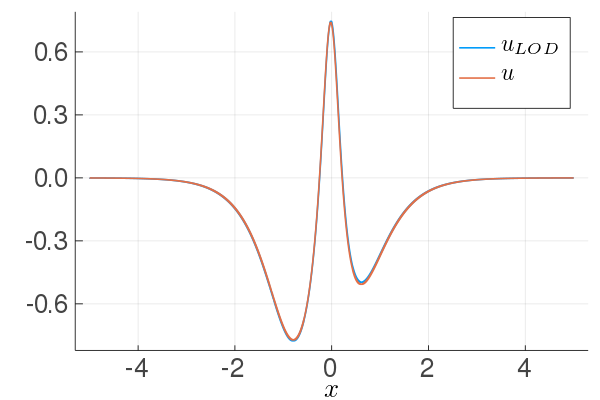}
        \caption{$\Re u$}
    \end{subfigure}\hspace{-10pt}
    \begin{subfigure}[b]{0.33\textwidth}
        \includegraphics[width=\textwidth]{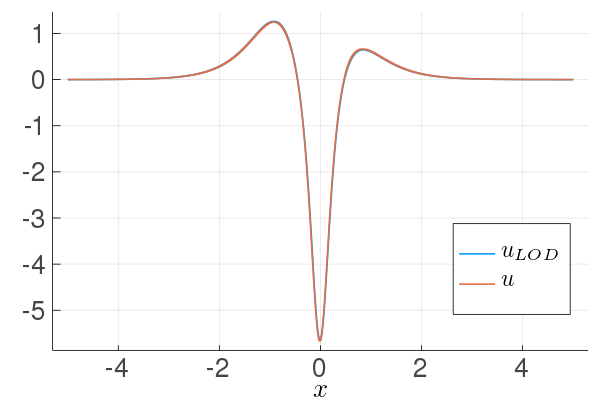}
        \caption{$\Im u$}
    \end{subfigure}\hspace{-10pt}
    \begin{subfigure}[b]{0.33\textwidth}
        \includegraphics[width=\textwidth]{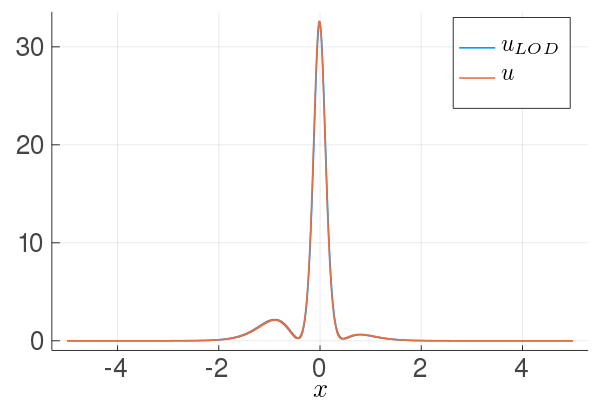}
        \caption{$|u|^2$}
    \end{subfigure}
\caption{Converged solution $u\LOD$ w.r.t. $\tau$ at $T=200$ of CN-FEM LOD using $ H = 40/2^{11} ,\ \ell = 12$ and $N=2^{24}$ time steps. We have $E[u\LOD]=-47.99999898 $, $c_1 =0.0010 $.  The relative $L^2$ and $H^1$ errors are $\|u\LOD-u\|/\|u\|=0.032$ and $\|\nabla(u-u\LOD)\|/\|\nabla u\LOD\| = 0.037$. The required CPU time was 100h.} \label{2048} 
\end{figure}

\section{Numerical experiments in 2D}
\label{section-numerics-2D}
In the previous example, the translational invariance of the mesh was used to reduce the number of local problems to a handful. To further illustrate the competitiveness of the proposed method we consider a two dimensional problem where the local problems are solved in parallel. Given sufficiently many parallel processes there is no need to split the potential as proposed in eq.  \eqref{def-a-innerproduct}. In this problem we seek $u(x,t)$ with
\begin{gather}  \label{Harmonic2D}
\begin{cases}
\ci \partial_t u  &= -\frac{1}{2} \Delta u + V u  + 5\pi|u|^2u   \qquad \mbox{in } \Omega \times (0,T], \\
u(\cdot,t)  &=  0     \qquad\hspace{102pt} \mbox{on } \partial \Omega \times (0,T],  \\
u(\cdot,0)  &= \sqrt{\frac{2}{\pi}}e^{-(x^2+y^2)} \qquad\hspace{47pt} \mbox{in }  \Omega.
\end{cases}
\end{gather}
Here, $\Omega = (-6,6)^2$ is the computational domain and we have an {\it anisotropic} harmonic trapping potential $V(x,y)=  \frac{1}{2}(x^2+(2y)^2)$. The energy is, up to machine precision, $E[u^0]=33/8$, likewise the mass is $M[u^0] = 1$. For the maximum time we selected $T = 2$.
The inner product $a(\cdot,\cdot)$, in the LOD is choosen as, 
\begin{align*}
	a(v,w) = \int_{\D}\frac{1}{2}\nabla v\cdot\overline{\nabla w}+Vv\overline{w}\ dx .
\end{align*}
\subsection{Convergence of $\uLOD^0$}
In Table \ref{2D_Error_Table} are tabulated the initial errors of $\uLOD^0$ for different values of $H$ and $\ell$. These values are subsequently plotted in Fig. \ref{2D_Error_Graph} versus $H$. From Fig. \ref{Energy_Mass_2D}, the 6th order convergence of mass and energy becomes apparent. The 4th order convergence in the $L^2$-norm and the 3rd order convergence in the $H^1$-seminorm are illustrated in Fig. \ref{L2_H1_2D}. In passing we note that the small kink in both convergence plots at $H = 0.21875$ is due to insufficient $\ell$. The characteristic length of the fine mesh is $h=1/128 = 0.0078125$, corresponding to roughly 2.7 million degrees of freedom. Remarkably the LOD-space reaches the accuracy of the fine grid already for $H=0.1875$, which corresponds to a mere $4565$ degrees of freedom. 
\begin{table}[H] 
	\begin{center}
		\begin{tabular}{l r|c| c|c | c|}
			\multicolumn{6}{c}{ }\\
		&	& $|E\LOD[\uLOD^0]-E[u^0]|$  & $|M[\uLOD^0]-M[u^0]|$ & $\|u^0-\uLOD^0\|$  & $\|\nabla( u^0-\uLOD^0)\| $  \\ \hline
			$H=0.5$& $\ell=2$ &0.044919 &0.006855 & 0.008919  & 0.109768     \\ \hline 
			$H=0.25$& $\ell=4$ & 0.000923 &  0.000132  & 0.000509  &  0.014307 \\ \hline
			$H = 0.21875$ & $\ell = 4$ & 0.000704& 0.000099 & 0.000364 & 0.012054 \\ \hline 
			$H=0.1875$ & $\ell=5$ &0.000188 &0.000023 & 0.000102  & 0.003711   \\ \hline 
			$H = 0.125$ & $\ell =5 $ & 0.000172& 0.000022 & 0.000053 & 0.003183 \\ \hline
		\end{tabular}
		\caption{Initial errors of the $a$-orthogonal projection of $u^0$ onto the LOD-space for the 2D model problem.}\label{2D_Error_Table}
	\end{center}	 
\end{table}

\begin{figure}[H]
	\begin{subfigure}[t]{0.45\textwidth}
		\includegraphics[scale=0.5]{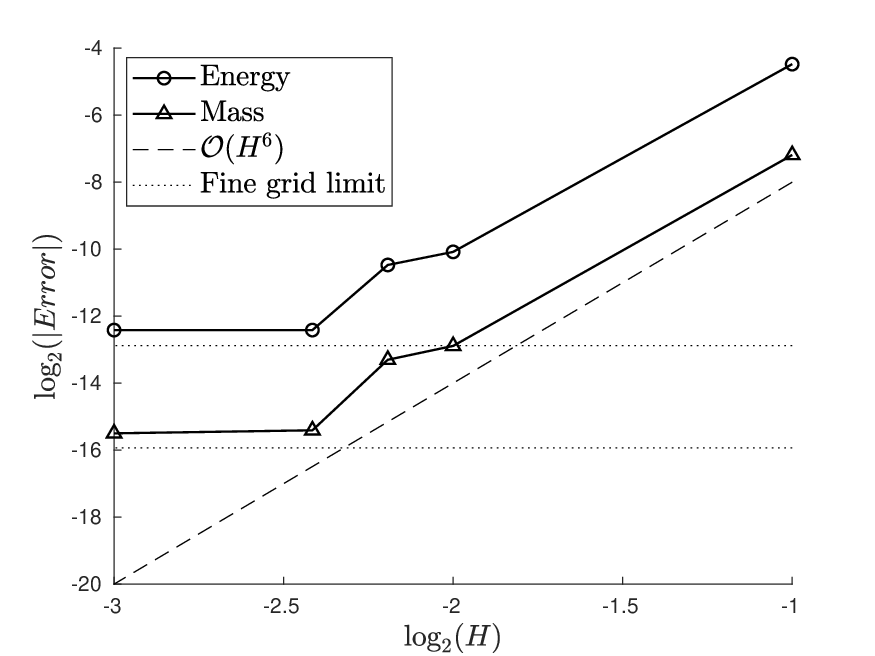} 
		\caption{6th order convergence of mass and of energy.}\label{Energy_Mass_2D}
	\end{subfigure}
	~
	\begin{subfigure}[t]{0.45\textwidth}
		\centering
		\includegraphics[scale=0.5]{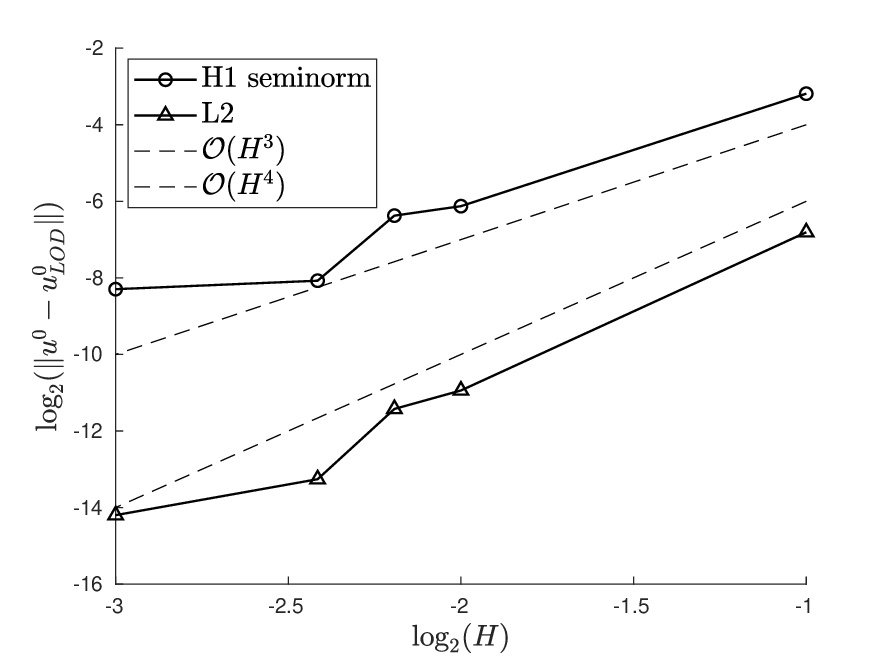}
		\caption{4th order convergence in $L^2$-norm and 3rd order convergence in $H^1$-seminorm.  }\label{L2_H1_2D}
	\end{subfigure}%
	\caption{Graph of values in Table \ref{2D_Error_Table} illustrating the superconvergence for the 2D model problem.}\label{2D_Error_Graph}
\end{figure}

\subsection{$L^{\infty}(L^2)$- and $L^{\infty}(H^1)$-convergence rates for $T=2$}
The final time is set to  $T= 2$ and the number of time steps is set to $N = 2^{13}$ to isolate the influence of $H$. As no analytic solution is known and since computing the solution on the fine mesh is infeasible we take as reference solution the LOD-solution with parameters $H=0.125,\ \ell = 5$. Surprisingly the order of convergence is one order higher than predicted, namely, the $\uLOD$ solution converges with 5th order in the $L^2$-norm and with 4th order in the $H^1$-seminorm. This is shown in Fig. \ref{L2_H1_T2}. The reason for these high rates could be related to the fact that the reference solution 
was an LOD-solution, however, this requires further investigation in the future. The density of the reference solution at $T=2$ is shown in Fig. \ref{Density_T2}.
\begin{figure}[H]
	\begin{subfigure}[t]{0.45\textwidth}
		\includegraphics[scale=0.5]{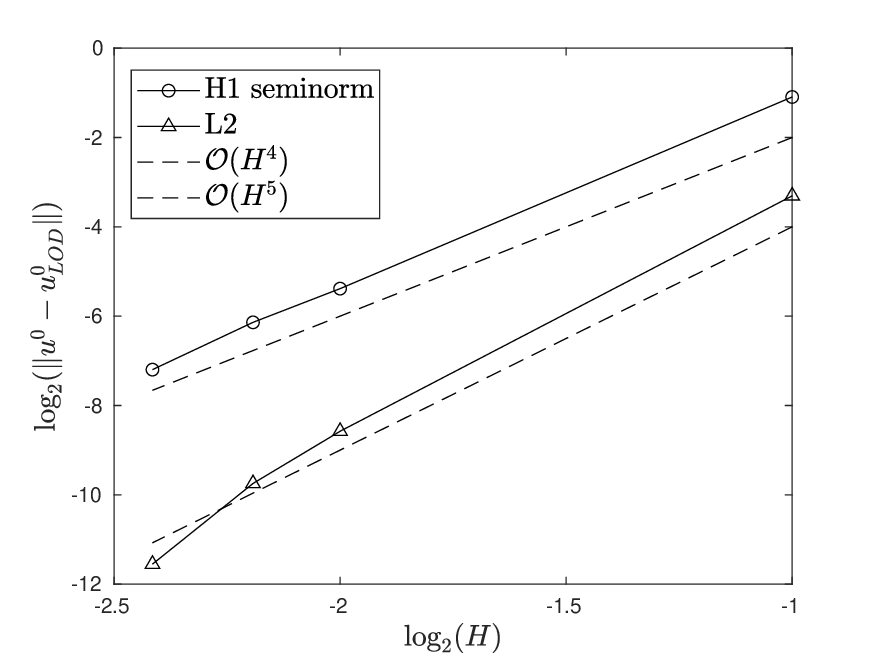} 
		\caption{Convergence rates of the solution at time $T=2$ versus $H$. The parameter $\ell$ is according to Table \ref{2D_Error_Table}.}\label{L2_H1_T2}
	\end{subfigure}
	~
	\begin{subfigure}[t]{0.45\textwidth}
		\centering
		\includegraphics[scale=0.5]{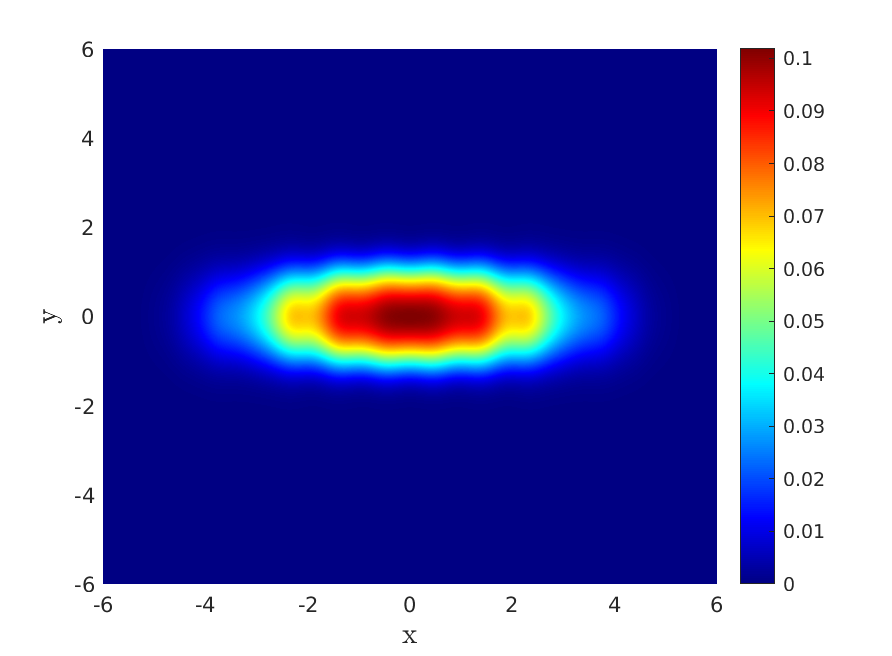}
		\caption{Density plot of reference solution $|\uLOD|^2$ at time $T=2$. }\label{Density_T2}
	\end{subfigure}%
	\caption{Convergence rates and density plot of the solution to initial value problem \eqref{Harmonic2D} in the LOD-space using the modified Crank-Nicolson method.  }\label{2D_T2}
\end{figure}
\subsection{CPU times}
One time step in the LOD-space with $H=0.1875$ and $ \ell = 5$, using 4 fixed point iterations requires 5 seconds on a single processor  on an Intel(R) Xeon(R) CPU E5-2637 v3 @ 3.50GHz unit. However as the assembly of the nonlinear term is embarrassingly parallel we find that this can be reduced to 1.2 seconds using the full 16 parallel processes of the very same computer. Consequently the solution at $T=2$ with discretization parameters $H=0.1875$ and $\ell=5$ was computed in about 3 hours. In comparison, one single time step using the fine mesh discretization and the very same fixed point iteration with similar tolerance, required 100 seconds. As in the previous example, the stopping criterion in the fixed point iteration was set to $\| u^{n+1}_{\text{\tiny{LOD}},i+1}- u^{n+1}_{\text{\tiny{LOD}},i}\| \leq 10^{-10}$. The precomputation of the LOD-space required roughly 13h on a 4 x 12 cores Intel E7-8857v2 Ivy Bridge unit. The tensor $\omega_{ijk}$ added another 6 hours to the precomputation. We note here for future improvement that the local problems should be amenable to being solved on a GPU.

\section{Proofs - Analysis of the modified Crank--Nicolson scheme}
\label{section-proof-main-result}

In this section we prove the main result stated in Theorem \ref{main-theorem-modified-CN}. We split the proof into several lemmas and start with the well-posedness. 

\begin{lemma}[existence of solutions to the modified Crank--Nicolson method]
\label{existence-modified-CN-LOD}
Assume \ref{A1}-\ref{A3}. Then for any $n\ge 1$ there exists at least one solution $u^{n}\LOD \in \VLOD$ to the modified Crank--Nicolson scheme \eqref{FullyDiscrete}. 
\end{lemma}

\begin{proof}
In the following we let $N_H$ denote the dimension of $V\LOD$ and a corresponding basis of the $V\LOD$ space shall be given by the set $\{ \phi_{\ell} \hspace{2pt} | 1 \le \ell \le N_H \}$. By $\cdot$ we denote the Euclidean inner product on $\mathbb{C}^{N_H}$. We note that the following proof does not exploit the structure of $V\LOD$ and works for any finite dimensional space.

We seek $u^{n+1}\LOD \in V\LOD$ given by \eqref{FullyDiscrete}. By multiplying the defining equation with the complex number $\ci$ we have
\begin{eqnarray}
\label{CN-LOD-2-algebraic}\nonumber\lefteqn{ 0 = \tau^{-1} \langle u^{n+1}\LOD  ,  \phi_{\ell} \rangle 
- \tau^{-1} \langle u^n\LOD  ,  \phi_{\ell} \rangle
 \hspace{2pt}
+ \ci \langle \nabla u^{n+\tfrac{1}{2}}\LOD , \nabla  \phi_{\ell} \rangle
+ \hspace{2pt} \ci \langle V u^{n+\tfrac{1}{2}}\LOD ,   \phi_{\ell} \rangle
 }\\
&\enspace& \qquad+ \hspace{2pt} \ci \beta \left\langle \frac{P\LOD(|u^{n+1}\LOD|^2+|u^{n}\LOD|^2)}{2}u^{n+1/2}\LOD ,  \phi_{\ell}\right\rangle
\hspace{140pt}
\end{eqnarray}
for all $ \phi_{\ell}$. Since we cannot guarantee that $P\LOD(|u^{n+1}\LOD|^2+|u^{n}\LOD|^2) \ge 0$, we consider a truncated auxiliary problem (note here the difference to the existence proof given in the appendix Lemma \ref{lemma-existence-solutions-classiccal-CN}). For the auxiliary problem let $M\in \mathbb{N}$ denote a truncation parameter and let $\chi_M : \R \rightarrow [-M,M]$ denote the continuous truncation function
$\chi_M(t):=\min\{ \tfrac{M}{|t|} , 1 \} \hspace{2pt}t$.
With this, we seek $u^{n,(M)}\LOD \in V\LOD$ as the solution to the truncated equation
\begin{eqnarray}
\label{CN-LOD-2-algebraic-truncated}\nonumber\lefteqn{ 0 = \tfrac{1}{\tau} \langle u^{n,(M)}\LOD  ,  \phi_{\ell} \rangle 
- \tfrac{1}{\tau} \langle u^n\LOD  ,  \phi_{\ell} \rangle
 \hspace{2pt}
+ \tfrac{\ci}{2} \langle \nabla u^{n,(M)}\LOD + \nabla u^{n}\LOD , \nabla  \phi_{\ell} \rangle
+ \hspace{2pt} \tfrac{\ci}{2} \langle V (u^{n,(M)}\LOD +  u^{n}\LOD) ,   \phi_{\ell} \rangle
 }\\
&\enspace& \qquad+ \hspace{2pt} \ci \tfrac{\beta}{4} \left\langle \chi_M( P\LOD(|u^{n,(M)}\LOD|^2+|u^{n}\LOD|^2)) (u^{n,(M)}\LOD +  u^{n}\LOD),  \phi_{\ell}\right\rangle.
\hspace{100pt}
\end{eqnarray}
for all $ \phi_{\ell}$. We start with proving the existence of $u^{n,(M)}\LOD \in V\LOD$, where we assume inductively that $u^{n}\LOD$ exists. The goal is to show the existence of $u^{n,(M)}\LOD \in V\LOD$ by using a variation of the  Browder fixed-point theorem, which says  that if $g : \C^{N_H} \rightarrow \C^{N_H}$ is a continuous function and if there exists a $K>0$ such that $\Re(g(\boldsymbol{\alpha}) \cdot \boldsymbol{\alpha}) >0 $ for all $\boldsymbol{\alpha}$ with $|\boldsymbol{\alpha}|=K$, then there exists a zero $\boldsymbol{\alpha}_0$ of $g$ with $|\boldsymbol{\alpha}_0| < K$ (cf. \cite[Lemma 4]{Browder1965}). 

To apply this result, we define the function $g^{(M)} : \C^{N_H} \rightarrow \C^{N_H}$ for $\boldsymbol{\alpha}\in \C^{N_H}$ through
\begin{eqnarray*}
\lefteqn{g_\ell^{(M)}(\boldsymbol{\alpha}) := \frac{1}{\tau} \sum_{m=1}^{N_H} \boldsymbol{\alpha}_m \langle  \phi_m , \phi_\ell \rangle 
 +
\frac{\ci}{2} \sum_{m=1}^{N_H} \boldsymbol{\alpha}_m \hspace{2pt} 
\langle \nabla \phi_m , \nabla \phi_\ell \rangle
 +
\frac{\ci}{2} \sum_{m=1}^{N_H} \boldsymbol{\alpha}_m \hspace{2pt} 
\langle V \phi_m  , \phi_\ell \rangle
}\\
&\enspace&
+ \frac{\beta \ci}{4} \langle \hspace{2pt}\chi_M \circ P\LOD\left(  \left| \sum_{m=1}^{N_H} \boldsymbol{\alpha}_m \phi_m  \right|^2 + |u^{n}\LOD|^2 \right)
\left( \sum_{m=1}^{N_H} \boldsymbol{\alpha}_m \phi_m \hspace{2pt} + u^{n}\LOD
\right)
 ,\phi_\ell \rangle
 +F_\ell,
\end{eqnarray*}
where $F\in \C^{N_H}$ is defined by
\begin{align*}
F_\ell := \frac{\ci}{2} \langle \nabla u^{n}\LOD ,  \nabla \phi_\ell \rangle
+ \frac{\ci}{2} \langle V u^{n}\LOD ,  \phi_\ell \rangle
 - \frac{1}{\tau} \langle u^{n}\LOD  , \phi_\ell \rangle.
\end{align*}
To show existence of some $\boldsymbol{\alpha}_0$ with $g^{(M)}(\boldsymbol{\alpha}_0)=0$ we need to show there is a sufficiently large $K \in \R_{>0}$ such that $\Re( g^{(M)}(\boldsymbol{\alpha}) \cdot \boldsymbol{\alpha} ) > 0$ for all $\boldsymbol{\alpha} \in \C^{N_H}$ with $|\boldsymbol{\alpha}|= K$.  For brevity, we denote $z_\alpha:=\sum_{m=1}^{N_H} \boldsymbol{\alpha}_m \phi_m$ and obtain
\begin{eqnarray*}
\lefteqn{\Re( g^{(M)}(\boldsymbol{\alpha}) \cdot \boldsymbol{\alpha} ) }\\
&=& \tfrac{1}{\tau} \|  z_\alpha \|^2  +
\Re\left( 
\frac{\beta \ci}{4} \langle \hspace{2pt}\chi_M \circ P\LOD\left(  |  z_\alpha |^2 + |u^{n}\LOD|^2 \right)
u\LOD^{n} , z_\alpha \rangle \right) \\
&\enspace& \qquad +  \Re\left( 
\frac{\ci}{2} \langle \nabla u^{n}\LOD ,  \nabla z_\alpha \rangle
+ \frac{\ci}{2} \langle V u^{n}\LOD ,  z_\alpha \rangle
 - \frac{1}{\tau} \langle u^{n}\LOD  , z_\alpha\rangle
\right).
\end{eqnarray*}
With the boundedness for $\chi_M$ and the Young inequality we have for the second term
\begin{align*}
\left| \Re\left( 
\frac{\beta \ci}{4} \langle \hspace{2pt}\chi_M \circ P\LOD\left(  |  z_\alpha |^2 + |u^{n}\LOD|^2 \right)
u\LOD^{n} , z_\alpha \rangle \right) \right| &\le \frac{\beta M}{4} \| z_\alpha \| \hspace{3pt} \| u^{n}\LOD \|\\
 &\le 
\frac{1}{8 \tau} \| z_\alpha \|^2 \hspace{3pt} + \tau \frac{\beta^2 M^2}{8} \| u^{n}\LOD \|^2.
\end{align*}
Similarly, we have
 \begin{align*}
 \left|  \frac{\ci}{2} \langle V u^{n}\LOD ,  z_\alpha \rangle \right| 
&\le \frac{1}{8\tau} \| z_\alpha \|^2 + \frac{\tau}{2} \hspace{2pt}\| V \|_{L^{\infty}(\D)}^2 \|  u^{n}\LOD \|^2;\\
 \left| \frac{1}{\tau} \langle u^{n}\LOD  , z_\alpha\rangle \right| &\le \frac{1}{8\tau} \| z_\alpha \|^2 + 
 \frac{2}{\tau} \|  u^{n}\LOD \|^2 \quad \mbox{and}\\
 \left| \frac{\ci}{2} \langle \nabla u^{n}\LOD ,  \nabla z_\alpha \rangle \right| &\le 
  \frac{1}{8\tau} \| z_\alpha \|^2 + 
\frac{\tau}{2} C\LOD^2 \| \nabla u^{n}\LOD \|^2,
 \end{align*}
where $C\LOD$ is the norm equivalence constant in the (finite-dimensional) LOD space, i.e., $C\LOD>0$ is the optimal constant such that $\| \nabla v \| \le C\LOD \| v \|$ for all $v \in V\LOD$. Combining the previous estimates, we have
\begin{align*}
\Re( g^{(M)}(\boldsymbol{\alpha}) \cdot \boldsymbol{\alpha} )  \ge   \frac{1}{2\tau} \| z_\alpha \|^2 - \tilde{C}
\end{align*}
where $\tilde{C}= \tau \frac{\beta^2 M^2}{8} \| u^{n}\LOD \|^2 +  \frac{\tau}{2} \hspace{2pt}\| V \|_{L^{\infty}(\D)}^2 \|  u^{n}\LOD \|^2 + \frac{2}{\tau} \|  u^{n}\LOD \|^2 + \frac{\tau}{2} C\LOD^2 \| \nabla u^{n}\LOD \|^2$. Hence, for every sufficiently large $\boldsymbol{\alpha}$ with $\| z_\alpha \|^2 > 2 \tau \tilde{C}$ we have positivity of $\Re( g^{(M)}(\boldsymbol{\alpha}) \cdot \boldsymbol{\alpha} )$ and consequently the existence of a point $\boldsymbol{\alpha}_0$ with $g^{(M)}(\boldsymbol{\alpha}_0)=0$, which in turn implies the existence of $u^{n,(M)} \in V\LOD$.

Now that we have verified the existence of truncated solutions we easily observe by testing in \eqref{CN-LOD-2-algebraic-truncated} with $u^{n,(M)} + u^{n}\LOD$ and taking the real part that 
$$
\| u^{n,(M)}\LOD \| = \| u^{n}\LOD \| \qquad \mbox{for all } M\ge 0.
$$
Note that we used here that $P\LOD(v)$ is real if $v$ is real, which is essential for this argument. Since $V\LOD$ is a finite-dimensional space and all norms are equivalent, this means that $\{ u^{n,(M)} \}_{M \in \mathbb{N}} \subset V\LOD$ is a bounded sequence. Consequently, we can extract a subsequence (for simplicity still denoted by $\{ u^{n,(M)}\LOD \}_{M \in \mathbb{N}}$) that converges strongly to some limit $u^{n,(\infty)}\LOD \in V\LOD$. Note that this also implies that the subsequence is uniformly bounded in $L^{\infty}(\D)$. Hence by passing to the limit $M\rightarrow \infty $ in \eqref{CN-LOD-2-algebraic-truncated} we have
\begin{eqnarray*}
\nonumber\lefteqn{ 0 = \tfrac{1}{\tau} \langle u^{n,(\infty)}\LOD  ,  \phi_{\ell} \rangle 
- \tfrac{1}{\tau} \langle u^n\LOD  ,  \phi_{\ell} \rangle
 \hspace{2pt}
+ \tfrac{\ci}{2} \langle \nabla u^{n,(\infty)}\LOD + \nabla u^{n}\LOD , \nabla  \phi_{\ell} \rangle
+ \hspace{2pt} \tfrac{\ci}{2} \langle V (u^{n,(\infty)}\LOD +  u^{n}\LOD) ,   \phi_{\ell} \rangle
 }\\
&\enspace& \qquad+ \hspace{2pt} \ci \tfrac{\beta}{4} \left\langle P\LOD(|u^{n,(\infty)}\LOD|^2+|u^{n}\LOD|^2) (u^{n,(\infty)}\LOD +  u^{n}\LOD),  \phi_{\ell}\right\rangle,
\hspace{240pt}
\end{eqnarray*}
where we can set $u^{n+1}\LOD=u^{n,(\infty)}\LOD$, which finishes the existence proof.
\end{proof}

Next, we prove the conservation of the mass and the modified energy.

\begin{lemma}[conservation properties]
Assume \ref{A1}-\ref{A3}. Then we have $M[u^{n}\LOD] = M[u^{0}\LOD]$ and
$E\LOD[u^{n}\LOD] = E\LOD[u^{0}\LOD]$.
\end{lemma}
\begin{proof}
Since $P\LOD(v)$ is real for any real function $v\in H^1_0(\D)$, the mass conservation follows readily from testing with $v=u^{n+1/2}\LOD$ in  \eqref{FullyDiscrete} and taking the imaginary part.

To verify conservation of the modified energy, we take the test function $v=u^{n+1}\LOD-u^n\LOD$ and consider the real part:
	\begin{align*}
	0 = &  \int_{\D} |\nabla u^{n+1}\LOD|^2-|\nabla u^n\LOD|^2 + V(|u^{n+1}\LOD|^2-|u^n\LOD|^2)
	\\ 
	&\quad +  \frac{\beta}{2}  P\LOD(|u^{n+1}\LOD|^2 + |u^{n}\LOD|^2) \left( |u^{n+1}\LOD|^2 - |u^{n}\LOD|^2\right) \hspace{2pt} dx 
	\end{align*}
	By definition of $P\LOD$ we have 
	\begin{eqnarray*}
	\lefteqn{\int_{\D} P\LOD(|u^{n+1}\LOD|^2 + |u^{n}\LOD|^2) \left( |u^{n+1}\LOD|^2 - |u^{n}\LOD|^2\right) \hspace{2pt} dx}\\ 
	&=& \int_{\D} P\LOD(|u^{n+1}\LOD|^2 + |u^{n}\LOD|^2) \hspace{3pt} P\LOD( |u^{n+1}\LOD|^2 - |u^{n}\LOD|^2) \hspace{2pt} dx
	\end{eqnarray*}
	and consequently by linearity of $P\LOD$
	 \begin{align*}
	0 = \int_{\D} |\nabla u^{n+1}\LOD|^2-|\nabla u^n\LOD|^2 + V(|u^{n+1}|^2-|u^n|^2)+ \frac{\beta}{2}\bigg( P\LOD(|u^{n+1}\LOD|^2)^2 - P\LOD(|u^n\LOD|^2)^2  \bigg)dx.
	\end{align*}
\end{proof}
Before we can prove the error estimate for the difference between the exact energies, i.e., $E[u^{n}\LOD]$ and $E[u(t_n)]$, we first require an $L^2$-error estimate for the error $u^n\LOD-u(t_n)$. This is done in several steps. Our approach is to show a $\tau$-independent convergence result for $u^n\LOD-u^n$, where $u^n$ denotes the solution of the semi-discrete Crank--Nicolson scheme in $H^1_0(\D)$, i.e., we split $u^n\LOD-u(t_n)=(u^n\LOD-u^n)+(u^n-u(t_n))$. Crucial for the proof is thus the following semi-discrete auxiliary problem whose properties have been studied in \cite{NonlinearCN} and \cite{H1Est}.
\begin{lemma}[semi-discrete Crank--Nicolson scheme]
\label{lemma-semidiscrete-CN}
Assume \ref{A1}-\ref{A6} and let $u^0$ denote the usual initial value. If $\tau$ is sufficiently small (bounded by a small constant that depends on $u$, $u_0$, $T$, $V$ and $\beta$), then for every $n\ge 0 $ there exists a solution $u^{n+1}\in H^1_0(\D)$ to the semi-discrete Crank--Nicolson equation
\begin{eqnarray}
\label{semi-disc-cnd-problem}
 \ci \langle \frac{u^{n+1} - u^n}{\tau} ,v\rangle
 = \langle \nabla u^{n+1/2},\nabla v\rangle
+ \langle V \hspace{1pt} u^{n+1/2},v\rangle + \beta \langle
\frac{|u^{n+1}|^2+ |u^n|^2}{2} u^{n+1/2},v\rangle
\end{eqnarray}
for all $v\in H^1_0(\D)$. 

Furthermore, we have $u^n \in H^2(\D)$ and there is unique family of solutions $u^n$ (family w.r.t. to $\tau$) so that it holds the a priori error estimate
\begin{align}
\label{error-est-un-H1}
\sup_{0\le n \le N} \left( \|u(\cdot , t_n ) - u^n\|_{H^1(D)}  +  \tau \| u(\cdot , t_n ) - u^n \|_{H^2(\D)}  \right)
\lesssim \tau^2,
\end{align}
where the hidden constant depends on the exact solution $u$ to problem \eqref{model-problem} and the maximum time $T$, but not on $\tau$. In the following we use the silent convention that $u^n$ always refers to the uniquely characterized solution that fulfills \eqref{error-est-un-H1}.
\end{lemma}
A proof of the $L^\infty(H^2)$ and $L^\infty(L^2)$ estimates is given in \cite{NonlinearCN}, a proof of the $L^\infty(H^1)$ estimate is given in \cite{H1Est}. As we will see later, the $L^\infty(H^2)$-estimate is not optimal and can be improved by one order. This improvement is one of the pillars of our error analysis in the LOD space. In fact, the $L^\infty(H^2)$-rates provided in Lemma \ref{lemma-semidiscrete-CN} are not sufficient to prove super convergence of $\mathcal{O}(H^4)$ for the final method.

Before we can derive the improved $L^\infty(H^2)$-estimates, we first need to investigate the regularity of $u^{n}$ in more detail and derive uniform and $\tau$-independent bounds for $\| \triangle u^n \|_{H^2}$. Note that with the availability of such bounds, we may apply the general theory of Section \ref{section-LOD} to conclude that $u^n$ is well-approximated in the LOD space, i.e., $\|u^n-A\LOD(u^n)\|\leq CH^4\|-\triangle u^n + V_1 u^n\|_{H^2}$, where $A\LOD$ is the Galerkin-projection on $V\LOD$. 

The next lemma takes the first step into that direction by showing that $u^n$ inherits regularity from the initial value and that $\|-\triangle u^n + V_1 u^n\|_{H^2}$ is bounded independent of the step size $\tau$.
\begin{lemma}\label{timedsc-regularity}
Assume \ref{A1}-\ref{splitting-potential}
and recall that \ref{A5} guarantees $u^0\in H^1_0(\D)\cap H^4(\D)$ and $\triangle u^0 \in H^1_0(\D) \cap H^2(\D)$. Furthermore, $u^n$ denotes the solution to the semi-discrete method \eqref{semi-disc-cnd-problem}. Then $\triangle u^n \in H^1_0(\D) \cap H^2(\D)$ and there exists a $\tau$-independent constant $C$ so that 
$$
\| D_{\tau} u^{n}  \|_{H^2(\D)} + \| \triangle u^{n} \|_{H^2(\D)} \le C
$$
for all $n\ge 0$.
\end{lemma}
\begin{proof}
The proof is established in several steps. For brevity, we denote in the following $\mathcal{H} u:= - \triangle u + V u$.

$\\$
{\it Step 1: We show that $D_{\tau} u^{n} \in H^1_0(\D) \cap H^2(\D)$ and $\| D_{\tau} u^{n}  \|_{H^2(\D)} \lesssim 1$}.

We already know that $u^n,u^{n+1}\in H^1_0(\D) \cap H^2(\D)$. It is hence obvious that $D_\tau u^n \in H^1_0(\D) \cap H^2(\D)$. With Lemma \ref{lemma-semidiscrete-CN} we have
\begin{align*}
\|  D_\tau u^n \|_{H^2(\D)} &= \tau^{-1} \| (u^{n+1} - u(t^{n+1})) + (u(t^{n}) - u^{n}) + (u(t^{n+1}) - u(t^{n})) \|_{H^2(\D)} \\
&\lesssim 1 + \tau^{-1}  \| u(t^{n+1}) - u(t^{n}) \|_{H^2(\D)} \le 1 + \| \partial_t u \|_{L^{\infty}(0,T;H^2(\D))}.
\end{align*}
{\it Step 2: We show that $\triangle u^{n+1/2} \in H^1_0(\D) \cap H^2(\D)$ and $\| \triangle u^{n+1/2} \|_{H^2(\D)} \lesssim 1$}.

We start from \eqref{semi-disc-cnd-problem} and observe that $u^{n+1/2} \in H^1_0(\D)$ can be characterized as the solution to
\begin{align}
\label{def-eqn-u-n-half}
 \langle \mathcal{H} u^{n+1/2}, v\rangle  =  \langle f^{n+1/2}, v\rangle \qquad \mbox{for all } v \in H^1_0(\D)
\end{align}
and where
$$
f^{n+1/2} := -\beta \frac{|u^{n+1}|^2+ |u^n|^2}{2} u^{n+1/2} + \ci D_\tau u^n.
$$
From {\it Step 1}, we already know that $ D_\tau u^n$ has the desired regularity and uniform bounds. It remains to check the nonlinear term, where a quick calculation shows that the second derivative of $\frac{|u^{n+1}|^2+ |u^n|^2}{2} u^{n+1/2} $ can be bounded by the $H^2$-norm of $u^n$ and $u^{n+1}$, which itself is bounded independent of $\tau$ according to Lemma \ref{lemma-semidiscrete-CN}. For example, we have
$$
\| |u^n|^2 u^{n}  \|_{H^2(\D)} \lesssim \| u^n \|_{H^2(\D)} \| u^n \|_{L^4(\D)}^2 + \| u^n \|_{L^{\infty}(\D)} \| u^n \|_{W^{1,4}(\D)}^2 
\lesssim \| u^n \|_{H^2(\D)}^3 \lesssim 1.
$$
Collecting the estimates hence guarantees $f^{n+1/2} \in H^1_0(\D) \cap H^2(\D)$ with $\| f^{n+1/2}\|_{H^2(\D)} \lesssim 1$. We conclude
$$
\| \triangle  u^{n+1/2} \|_{H^2(\D)} \le 
\| V  u^{n+1/2} \|_{H^2(\D)}  + \| f^{n+1/2} \|_{H^2(\D)} \lesssim 1,
$$
where we used assumption \ref{A4} and the Sobolev embedding $H^1(\mathcal{D}) \hookrightarrow L^6(\mathcal{D})$ for bounded Lipschitz domains to bound $V \hspace{1pt} u^{n+1/2} \in H^1_0(\D) \cap H^2(\D)$ uniformly and independent of $\tau$.

$\\$
{\it Step 3: We show that $\triangle u^{n} \in H^1_0(\D) \cap H^2(\D)$ and $\| \triangle u^{n} \|_{H^1(\D)} \lesssim C$}.

In the previous step we saw that $\mathcal{H} u^{n+1/2}  \in H^1_0(\D) \cap H^2(\D)$. Recursively we conclude with the assumptions on the initial value that $\mathcal{H} u^{n+1} \in H^1_0(\D) \cap H^2(\D)$ and in particular $\triangle u^{n+1} \in H^1_0(\D) \cap H^2(\D)$. We can hence apply $\mathcal{H}$ to \eqref{def-eqn-u-n-half} to obtain
\begin{align*}
\mathcal{H}^2 u^{n+1/2} =  - \beta \mathcal{H} \left( \frac{|u^{n+1}|^2+ |u^n|^2}{2} u^{n+1/2} \right) + \ci \mathcal{H}(D_\tau  u^n).
\end{align*}
By exploiting that $\triangle u^{n} \in H^1_0(\D) \cap H^2(\D)$ and Sobolev embeddings we easily observe that $\mathcal{H}^2 u^{n+1/2}  \in H^1_0(\D) \cap H^2(\D)$. Iteratively we can conclude that $\mathcal{H}^2 u^{n} \in H^1_0(\D) \cap H^2(\D)$ (and $\mathcal{H}^3 u^{n} \in L^2(\D)$) for all $n\ge 0$. This implies
\begin{eqnarray}
\label{identity-H-Dtau-u-n}
\lefteqn{\ci \langle \nabla \mathcal{H}(D_\tau  u^n) , \nabla \mathcal{H} u^{n+1/2} \rangle}\\
\nonumber&=& \langle\nabla \mathcal{H}^2 u^{n+1/2} ,  \nabla \mathcal{H}  u^{n+1/2} \rangle +
\beta \langle  \nabla \mathcal{H} \left( \frac{|u^{n+1}|^2+ |u^n|^2}{2} u^{n+1/2} \right) ,  \nabla \mathcal{H}  u^{n+1/2} \rangle.
\end{eqnarray}
We have a closer look at the first term and observe
\begin{eqnarray*}
\lefteqn{ \langle\nabla \mathcal{H}^2 u^{n+1/2} ,  \nabla \mathcal{H}  u^{n+1/2} \rangle}\\
&=& - \langle\nabla \triangle \mathcal{H} u^{n+1/2} ,  \nabla \mathcal{H}  u^{n+1/2} \rangle
 + \langle  \mathcal{H} u^{n+1/2} \hspace{2pt} \nabla V ,  \nabla \mathcal{H}  u^{n+1/2} \rangle
+ \langle V \nabla \mathcal{H} u^{n+1/2}  ,  \nabla \mathcal{H}  u^{n+1/2} \rangle \\
&=& \langle \triangle \mathcal{H} u^{n+1/2} ,  \triangle \mathcal{H}  u^{n+1/2} \rangle
 + \langle  \mathcal{H} u^{n+1/2} \hspace{2pt} \nabla V ,  \nabla \mathcal{H}  u^{n+1/2} \rangle
+ \langle V \nabla \mathcal{H} u^{n+1/2}  ,  \nabla \mathcal{H}  u^{n+1/2} \rangle,
\end{eqnarray*}
where the last step exploited that $\triangle \mathcal{H} u^{n+1/2} \in H^1_0(\D)$.
Hence, by taking the imaginary part in \eqref{identity-H-Dtau-u-n} we obtain
\begin{eqnarray*}
\lefteqn{ \frac{ \| \nabla \mathcal{H} u^{n+1}\|^2 - \| \nabla \mathcal{H} u^{n}\|^2  }{2 \tau} }\\
&=& \Im  \langle  \mathcal{H} u^{n+1/2} \hspace{2pt} \nabla V ,  \nabla \mathcal{H}  u^{n+1/2} \rangle + 
\beta \Im \langle  \nabla \mathcal{H} \left( \frac{|u^{n+1}|^2+ |u^n|^2}{2} u^{n+1/2} \right) ,  \nabla \mathcal{H}  u^{n+1/2} \rangle \\
&\lesssim& \| u^{n}\|^2_{H^2(\D)} + \| u^{n+1}\|^2_{H^2(\D)} +   \| \nabla \mathcal{H} u^{n+1/2}\|^2 + \left|  \langle  \nabla \mathcal{H} \left( \frac{|u^{n+1}|^2+ |u^n|^2}{2} u^{n+1/2} \right) ,  \nabla \mathcal{H}  u^{n+1/2} \rangle \right| \\
&=& \| u^{n}\|^2_{H^2(\D)} + \| u^{n+1}\|^2_{H^2(\D)} +   \| \nabla \mathcal{H} u^{n+1/2}\|^2 + \left|  \langle  \mathcal{H} \left( \frac{|u^{n+1}|^2+ |u^n|^2}{2} u^{n+1/2} \right) ,  \triangle \mathcal{H}  u^{n+1/2} \rangle \right|.
\end{eqnarray*}
Since {\it Step 2} proved $\| \mathcal{H} u^{n+1/2} \|_{H^2(\D)} \lesssim 1$ and $\|  \tfrac{|u^{n+1}|^2+ |u^n|^2}{2} u^{n+1/2} \|_{H^2(\D)} \lesssim 1$ we conclude
\begin{align*}
 \| \nabla \mathcal{H} u^{n}\|^2  \le  \| \nabla \mathcal{H} u^{n-1}\|^2 + \tau 
 \le  \| \nabla \mathcal{H} u^{0}\|^2 +n \tau  \lesssim 1,
\end{align*}
which in turn implies $\|  \triangle u^{n}\|_{H^1(\D)} \lesssim 1$.

$\\$
{\it Step 4: We show that $\| \triangle u^{n} \|_{H^2(\D)} \lesssim C$}.

We apply $\mathcal{H}^2$ to \eqref{def-eqn-u-n-half} and multiply the equation with $\mathcal{H}^2 u^{n+1/2} \in H^1_0(\D) \cap H^2(\D)$ (cf. {\it Step 3}) to obtain
\begin{eqnarray*}
\lefteqn{\ci \langle  \mathcal{H}^2(D_\tau  u^n) , \mathcal{H}^2 u^{n+1/2} \rangle}\\
&=& \langle \mathcal{H}^3 u^{n+1/2} ,  \mathcal{H}^2  u^{n+1/2} \rangle +
\beta \langle  \mathcal{H}^2 \left( \tfrac{|u^{n+1}|^2+ |u^n|^2}{2} u^{n+1/2} \right) ,  \mathcal{H}^2  u^{n+1/2} \rangle \\
&=& \langle \nabla \mathcal{H}^2 u^{n+1/2} ,  \nabla \mathcal{H}^2  u^{n+1/2} \rangle 
+  \langle V \mathcal{H}^2 u^{n+1/2} ,  \mathcal{H}^2  u^{n+1/2} \rangle 
+ \beta \langle  \mathcal{H}^2 \left( \tfrac{|u^{n+1}|^2+ |u^n|^2}{2} u^{n+1/2} \right) ,  \mathcal{H}^2  u^{n+1/2} \rangle.
\end{eqnarray*}
Taking the imaginary part yields
\begin{eqnarray}
\label{Hsquare-un-id}
\lefteqn{ \frac{ \| \mathcal{H}^2 u^{n+1}\|^2 - \|  \mathcal{H}^2 u^{n}\|^2  }{2 \tau} }\\
\nonumber&=& - \beta \Im \langle \triangle \mathcal{H} \left( \tfrac{|u^{n+1}|^2+ |u^n|^2}{2} u^{n+1/2} \right) ,  \mathcal{H}^2  u^{n+1/2} \rangle +
\beta \Im \langle V \mathcal{H} \left( \tfrac{|u^{n+1}|^2+ |u^n|^2}{2} u^{n+1/2} \right) ,  \mathcal{H}^2  u^{n+1/2} \rangle
\end{eqnarray}
The second term can be bounded in the usual manner by
\begin{align}
\label{Hsquare-un-est-1}
\left|  \langle V \mathcal{H} \left( \tfrac{|u^{n+1}|^2+ |u^n|^2}{2} u^{n+1/2} \right) ,  \mathcal{H}^2  u^{n+1/2} \rangle \right| \lesssim  1 + \| \mathcal{H}^2 u^{n+1}\|^2 + \|  \mathcal{H}^2 u^{n}\|^2.
\end{align}
The first term needs a more careful investigation where we need to find a bound for the expression $\langle \triangle^2 \left( \tfrac{|u^{n+1}|^2+ |u^n|^2}{2} u^{n+1/2} \right) ,  \mathcal{H}^2  u^{n+1/2} \rangle $. For simplicity, letting $g^n:=\frac{|u^{n+1}|^2+ |u^n|^2}{2}$ we have
\begin{eqnarray*}
\lefteqn{\triangle^2 ( g^n u^{n+1/2} )}\\
&=& \triangle \left( g^n \triangle u^{n+1/2} + 2 \nabla u^{n+1/2} \cdot \nabla g^n + u^{n+1/2} \triangle g^n  \right) \\
 &=& 6 \triangle u^{n+1/2} \hspace{2pt} \triangle g^n
 + 4 \nabla \triangle u^{n+1/2} \cdot \nabla g^n  
 +   \triangle^2 u^{n+1/2} \hspace{2pt} g^n
 + 4 \nabla u^{n+1/2} \cdot \nabla \triangle g^n
 + u^{n+1/2} \triangle^2 g^n
\end{eqnarray*}
and the derivatives of $g^n$ can be computed with
\begin{align*}
\nabla |u^n|^2 &=2 \Re \left( u^n \overline{\nabla u^n} \right);
\hspace{50pt}
\triangle |u^n|^2 = 2 |\nabla u^n|^2 + 2 \Re \left( u^n \overline{\triangle u^n} \right);\\
\nabla \triangle |u^n|^2 &= 6 \Re\left( \nabla u^n \hspace{2pt} \overline{\triangle u^n} \right) + 2 \Re \left( u^n \overline{\nabla \triangle u^n} \right) \qquad \hspace{40pt}\mbox{and} \\
 \triangle^2 |u^n|^2 &=
  6 |\triangle u^n|^2 +
  8 \Re\left( \nabla u^n \hspace{2pt} \overline{\nabla \triangle u^n} \right) 
+ 2 \Re \left( u^n \overline{\triangle^2 u^n} \right).
\end{align*}
Consequently, we estimate the various terms with
\begin{align*}
\left| \langle \triangle u^{n+1/2} \hspace{2pt} \triangle g^n,  \mathcal{H}^2  u^{n+1/2}  \rangle \right|  &\le \|  \triangle u^{n+1/2}  \|_{L^{\infty}(\D)} \| \triangle g^n \|_{L^2(\D)} \| \mathcal{H}^2  u^{n+1/2} \|_{L^2(\D)} \\
&\lesssim  \|  \triangle u^{n+1/2}  \|_{H^2(\D)}  \left( \| u^n \|_{H^2(\D)}^2 
+  \| u^{n+1} \|_{H^2(\D)}^2
\right) \| \mathcal{H}^2  u^{n+1/2} \|_{L^2(\D)}  \\
&\lesssim \| \mathcal{H}^2  u^{n+1/2} \|_{L^2(\D)},
\end{align*} 
where we used the result of {\it Step 2} to bound $\|  \triangle u^{n+1/2}  \|_{H^2(\D)}$. Next, we have
\begin{align*}
\left| \langle \nabla \triangle u^{n+1/2} \cdot \nabla g^n ,  \mathcal{H}^2  u^{n+1/2}  \rangle \right|  &\le \|  \nabla \triangle u^{n+1/2} \|_{L^4(\D)}
\|  \nabla g^n \|_{L^4(\D)} \|  \mathcal{H}^2  u^{n+1/2}  \| \\
&\le \|  \triangle u^{n+1/2}  \|_{H^2(\D)}  \|  g^n \|_{H^2(\D)} \|  \mathcal{H}^2  u^{n+1/2}  \| .
\end{align*} 
This can be bounded as the previous term, since $\|  g^n \|_{H^2(\D)}  \lesssim \|  \triangle g^n \|_{L^2(\D)} $ for $g^n \in H^1_0(\D) \cap H^2(\D)$. Consequently, $\left| \langle \nabla \triangle u^{n+1/2} \cdot \nabla g^n ,  \mathcal{H}^2  u^{n+1/2}  \rangle \right| \lesssim \| \mathcal{H}^2  u^{n+1/2} \|$. In a similar fashion we can estimate
\begin{align*}
\left| \langle  \triangle^2 u^{n+1/2} \hspace{2pt} g^n ,  \mathcal{H}^2  u^{n+1/2}  \rangle \right|  &\le \|  \triangle u^{n+1/2} \|_{H^2(\D)} \| g^n \|_{L^{\infty}(\D)} \| \mathcal{H}^2  u^{n+1/2} \| \lesssim \| \mathcal{H}^2  u^{n+1/2} \|.
\end{align*} 
Next, we consider
\begin{eqnarray*}
\lefteqn{\left| \langle  \nabla u^{n+1/2} \cdot (6 \Re\left( \nabla u^n \hspace{2pt} \overline{\triangle u^n} \right) + 2 \Re \left( u^n \overline{\nabla \triangle u^n} \right)) ,  \mathcal{H}^2  u^{n+1/2}  \rangle \right|}\\ 
 &\lesssim& \|  \nabla u^{n+1/2} \|_{L^{6}(\D)}  \|  \nabla u^{n} \|_{L^{6}(\D)} 
 \| \triangle u^n \|_{L^6(\D)}  \|  \mathcal{H}^2  u^{n+1/2}\| \\
&\enspace& \quad + \|  \nabla u^{n+1/2} \|_{L^{4}(\D)}  \| u^n \|_{L^{\infty}(\D)} \| \nabla \triangle u^n \|_{L^{4}(\D)}
 \|  \mathcal{H}^2  u^{n+1/2}\| \\
&\lesssim&  \|  \mathcal{H}^2  u^{n+1/2}\| + \| \triangle u^n \|_{H^{2}(\D)}
 \|  \mathcal{H}^2  u^{n+1/2}\| \\
 &\lesssim& 1 +  \|  \mathcal{H}^2  u^{n}\|^2 +  \|  \mathcal{H}^2  u^{n+1}\|^2. 
\end{eqnarray*} 
Note that we  used here that $ \|  \nabla u^{n+1/2} \|_{L^{6}(\D)}  \lesssim  \|  u^{n+1/2} \|_{H^{2}(\D)} \lesssim 1$ by Lemma \ref{lemma-semidiscrete-CN} and that $ \|  \triangle u^{n} \|_{L^{6}(\D)}  \lesssim  \| \triangle u^{n} \|_{H^{1}(\D)} \lesssim 1$ by {\it Step 3}. We can conclude that
\begin{align*}
\left| \langle  \nabla u^{n+1/2} \cdot \nabla \triangle g^n ,  \mathcal{H}^2  u^{n+1/2}  \rangle \right|  &\lesssim  \|  \mathcal{H}^2  u^{n+1/2}\|.
\end{align*} 
It remains to check $ \langle  u^{n+1/2} \triangle^2 g^n ,  \mathcal{H}^2  u^{n+1/2}  \rangle$ where we have
\begin{eqnarray*}
\lefteqn{ \left| \langle  u^{n+1/2} \triangle^2 |u^n|^2 ,  \mathcal{H}^2  u^{n+1/2}  \rangle \right|}\\  
&\lesssim&
\left| \langle  u^{n+1/2}   |\triangle u^n|^2   ,  \mathcal{H}^2  u^{n+1/2}  \rangle \right|  
+ \left| \langle  u^{n+1/2}   \Re\left( \nabla u^n \hspace{2pt} \overline{\nabla \triangle u^n} \right)    ,  \mathcal{H}^2  u^{n+1/2}  \rangle \right|  \\
&\enspace& \quad+ \left| \langle  u^{n+1/2}   \Re \left( u^n \overline{\triangle^2 u^n} \right)     ,  \mathcal{H}^2  u^{n+1/2}  \rangle \right|  \\
&\lesssim& \left( \|  u^{n+1/2} \|_{L^{\infty}(\D)}   \| \triangle u^n \|_{L^4(\D)}^2 
+ \| u^{n+1/2} \|_{L^{\infty}(\D)}  \| \nabla u^n \|_{L^4(\D)} \| \nabla \triangle u^n \|_{L^4(\D)} \right)  \| \mathcal{H}^2  u^{n+1/2} \| \\
&\enspace& \quad +  \|  u^{n+1/2} \|_{L^{\infty}(\D)}  \|  u^{n} \|_{L^{\infty}(\D)}  \|  \triangle^2 u^{n} \| \hspace{3pt}   \| \mathcal{H}^2  u^{n+1/2} \| \\
&\lesssim&  (1 + \|  \mathcal{H}^2  u^{n}\| +  \|  \mathcal{H}^2  u^{n+1}\| ) ( \|  \mathcal{H}^2  u^{n}\| +  \|  \mathcal{H}^2  u^{n+1}\| ).
\end{eqnarray*} 
Again, we used {\it Step 3} when estimating $ \| \triangle u^n \|_{L^4(\D)}^2 \lesssim  \| \triangle u^n \|_{H^1(\D)}^2 \lesssim 1$. We conclude that
$$
\left|  \langle  u^{n+1/2} \triangle^2 g^n ,  \mathcal{H}^2  u^{n+1/2}  \rangle \right| \lesssim  1 + \|  \mathcal{H}^2  u^{n}\|^2 +  \|  \mathcal{H}^2  u^{n+1}\|^2.
$$
Collecting and combining all the estimates allows us to conclude that
\begin{align*}
\beta \left| \Im \langle \triangle \mathcal{H} \left( \tfrac{|u^{n+1}|^2+ |u^n|^2}{2} u^{n+1/2} \right) ,  \mathcal{H}^2  u^{n+1/2} \rangle \right| \lesssim 1 + \|  \mathcal{H}^2  u^{n}\|^2 +  \|  \mathcal{H}^2  u^{n+1}\|^2
\end{align*}
and hence with \eqref{Hsquare-un-id} and \eqref{Hsquare-un-est-1} we have
\begin{eqnarray*}
\| \mathcal{H}^2 u^{n+1}\|^2 - \|  \mathcal{H}^2 u^{n}\|^2
\lesssim \tau \left( 1 + \|  \mathcal{H}^2  u^{n}\|^2 +  \|  \mathcal{H}^2  u^{n+1}\|^2 \right).
\end{eqnarray*}
Gr\"onwall's inequality yields 
$$
\|  \mathcal{H}^2 u^{n}\| \lesssim 1
$$
the regularity estimate $\| \mathcal{H} u^{n}\|_{H^2(\D)}  \lesssim \|  \mathcal{H}^2 u^{n}\|  $ for $ \mathcal{H} u^{n} \in H^1_0(\D) \cap H^2(\D)$ finishes the proof of the last step.
\end{proof}

\begin{lemma}\label{error_pde-lemma}
Assume \ref{A1}-\ref{A6} and denote by $e^n_\CN$ the error of the semi-discrete Crank--Nicolson method \eqref{semi-disc-cnd-problem}, i.e., $e^n_\CN= u^n-u(t_n)$. The error fulfills the identity 
\begin{equation} \label{error_pde}
\ci D_\tau e^n_\CN+\triangle e^{n+1/2}_\CN-Ve^{n+1/2}_\CN-e^n_{\beta,\text{\tiny CN}} = T^n,
\end{equation}
where $e^{n+1/2}_\CN:= \tfrac{e^{n+1}_{\CN}+e^n_{\CN}}{2}$; the consistency error is
\begin{eqnarray}\label{Taylor}
\lefteqn{T^n := \ci \hspace{2pt}( \tfrac{u(t_{n+1}) - u(t_n)}{\tau} - \partial_t u(t_{n+1/2}))+(\triangle - V)(  \tfrac{u(t_{n+1}) + u(t_n)}{2}-u(t_{n+1/2})) } \nonumber \\
&\enspace& -\bigg( \tfrac{|u^{n+1}|^2+|u^n|^2}{2} \tfrac{u(t_{n+1}) + u(t_n)}{2}-|u(t_{n+1/2})|^2)u(t_{n+1/2}) \bigg),\hspace{100pt}
  \end{eqnarray} 
and we define
\begin{eqnarray}\label{enbeta}
 e^n_{\beta,\text{\tiny CN}} &:=&  \frac{|u^{n+1}|^2+|u^n|^2}{2}u^{n+1/2}-\frac{|u(t_{n+1})|^2+|u(t_n)|^2}{2}\bigg(\frac{u(t_{n+1})+u(t_{n})}{2} \bigg) \nonumber \\
&=& \bigg(\frac{|u^{n+1}|^2+|u^n|^2}{2}-\frac{|u(t_{n+1})|^2+|u(t_n)|^2}{2}\bigg)\bigg(\frac{u(t_{n+1})+u(t_{n})}{2} \bigg)+ \\
&& \qquad + \frac{|u^{n+1}|^2+|u^n|^2}{2}\bigg(\frac{u(t_{n+1})+u(t_{n})}{2}-u^{n+1/2} \bigg). \nonumber
\end{eqnarray}
Furthermore in virtue of \ref{A6} there holds:
\begin{eqnarray*}
\sum_{k=0}^n \|T^k\|_{H^2(\D)}^2 \lesssim \tau^3 .
\end{eqnarray*}
\end{lemma}
\begin{proof}
By simply subtracting \eqref{model-problem} from \eqref{semi-disc-cnd-problem} one finds equation \eqref{error_pde}. The consistency error  \eqref{Taylor} is then easily bounded by means of Taylor expansion and assumption \ref{A6}.
\end{proof}
With the previous two lemmas we are now prepared to prove the optimal  $L^\infty(H^2)$-estimates.
\begin{lemma}[Optimal $L^\infty(H^2)$ error estimate of the Crank--Nicolson method]
\label{optimal-H2-est-lemma}
Assume \ref{A1}-\ref{A6}, let $u^{n}\in H^1_0(\D)$ denote the semi-discrete Crank--Nicolson approximation given by \eqref{semi-disc-cnd-problem} and $u$ the exact solution. Then it holds
\begin{align*}
\sup_{0\le n \le N}  \| u(\cdot , t_n ) - u^n \|_{H^2(\D)}
\lesssim  \tau^2.
\end{align*}
Furthermore, there exists a $\tau$-independent constant $C>0$ such that 
$$
\| \triangle (D_\tau u^{n-1/2}) \|_{H^2(\D)} \leq C.
$$
Note that $D_\tau u^{n-1/2} = \tfrac{1}{2\tau}(u^{n+1}- u^{n-1} )$ and that it does \underline{\emph{not}} imply $\| \triangle (D_\tau u^{n}) \|_{H^2(\D)} \leq C$.
\end{lemma}

\begin{proof}
First, we note that $D_\tau e^n_\CN ,Ve^{n+1/2}_\CN,T^n,e^n_{\beta,\text{\tiny CN}} \in H^1_0(\D)$ which allows for integration by parts without boundary terms.
Now, multiplying equation \eqref{error_pde} by $D_\tau \triangle e^n_\CN$ and considering only the real part results in:
\begin{eqnarray}\label{H2-rec}
\lefteqn{\frac{\|\triangle e^{n+1}_\CN\|^2-\|\triangle e^n_\CN\|^2}{2\tau} = \Re\bigg(\langle \triangle(Ve^{n+1/2}_\CN),D_\tau e^n_\CN\rangle +\langle \triangle e^n_{\beta,\text{\tiny CN}} ,D_\tau e^n\rangle + \langle \triangle T^n,D_\tau  e^n_\CN\rangle  \bigg)} \nonumber \\
\nonumber&\leq& |\langle \triangle (Ve^{n+1/2}_\CN),-\triangle e^{n+1/2}_\CN+Ve^{n+1/2}_\CN+e^n_{\beta,\text{\tiny CN}}+T^n\rangle|  \\
&\enspace& \quad + |\langle \triangle e^n_{\beta,\text{\tiny CN}} ,-\triangle e^{n+1/2}_\CN+Ve^{n+1/2}_\CN+e^n_{\beta,\text{\tiny CN}}+T^n\rangle | \nonumber\\
&\enspace& \quad + |\langle \triangle T^n, -\triangle e^{n+1/2}_\CN+Ve^{n+1/2}_\CN+e^n_{\beta,\text{\tiny CN}}+T^n\rangle| \nonumber  \\
&\lesssim& \|\triangle e^{n+1}_\CN\|^2+\|\triangle e^n_\CN\|^2+ \tau^4 +\|T^n\|^2_{H^2(\D)}+\| e^n_{\beta,\text{\tiny CN}}\|^2_{H^2(\D)},\hspace{120pt}
\end{eqnarray}
where elliptic regularity theory guarantees $\| e^n_{\beta,\text{\tiny CN}}\|_{H^2(\D)} \lesssim \| \triangle e^n_{\beta,\text{\tiny CN}}\|$.
In order to use Gr\"onwall's inequality we need to bound $\|\triangle e^n_{\beta,\text{\tiny CN}}\|$ in terms of $\tau^2$,  $\|\triangle e^{n+1}_\CN\|$ and $\|\triangle e^{n}_\CN\|$. From equation \eqref{enbeta} it is clear that this need only be done for two kinds of expression, namely the expressions $\triangle[|u^n|^2(u(t_n)-u^{n})]$ and $\triangle [(|u^n|^2-|u(t_n)|^2)u(t_{n})]$. We expand these two cases using Leibniz's rule. For the first term we use $\triangle |u^n|^2 = 2 |\nabla u^n|^2 + 2 \Re \left( u^n \overline{\triangle u^n} \right)$ to obtain
\begin{eqnarray*}
\lefteqn{\|\triangle[|u^n|^2(u(t_n)-u^n)]\|} \\
&=& \| \triangle |u^n|^2 \hspace{2pt} (u(t_n)-u^n)+2\nabla|u^n|^2\cdot\nabla(u(t_n)-u^n))+|u^n|^2\triangle(u(t_n)-u^n)\| \\
& \lesssim & \| \nabla u^n \|_{L^4(\D)} \| u(t_n)-u^n \|_{L^4(\D)} + \| u^n \|_{L^{\infty}(\D)} \| \triangle u^n \|_{L^4(\D)} \| u(t_n)-u^n \|_{L^4(\D)} \\
&\enspace& \quad + \| u^n \|_{L^{\infty}(\D)} \| \nabla u^n \|_{L^4(\D)} \| \nabla(u(t_n)-u^n)) \|_{L^4(\D)} +\| u^n\|_{L^\infty(\D)}^2 \|\triangle (u(t_n)-u^n)\| \\
& \lesssim & C( \| u^n \|_{H^2(\D)} , \| \triangle u^n \|_{H^1(\D)}  ) \hspace{4pt} \left(\| u(t_n)-u^n \|_{H^1(\D)} + \| u(t_n)-u^n \|_{H^2(\D)}\right),
\end{eqnarray*}
where we used Sobolev embedding estimates. Lemma \ref{lemma-semidiscrete-CN} and Lemma \ref{timedsc-regularity} allow us to bound $\| u^n \|_{H^2(\D)}$ and $\| \triangle u^n \|_{H^1(\D)}$. Together with the regularity estimate $\|v \|_{H^2(\D)} \lesssim \| \triangle v \|$ for $v\in H^1_0(\D)$ we conclude
\begin{align*}
\|\triangle[|u^n|^2(u(t_n)-u^n)]\| \lesssim \tau^2 + \| \triangle e^n_\CN \|.
\end{align*}
For the term $\triangle [(|u^n|^2-|u(t_n)|^2)u(t_{n})]$ we split
\begin{eqnarray*}
\lefteqn{\triangle[(|u^n|^2-|u(t_n)|^2)u(t_n)] }  \\
&=&\underbrace{\triangle(|u^n|^2-|u(t_n)|^2) \hspace{3pt} u(t_n)}_{\mbox{I}} + 2\underbrace{\nabla(|u^n|^2-|u(t_n)|^2)\cdot\nabla u(t_n)}_{\mbox{II}} + \underbrace{(|u^n|^2-|u(t_n)|^2) \hspace{2pt} \triangle u(t_n)}_{\mbox{III}},
\end{eqnarray*}
where we can estimate $I$ using $\triangle |u^n|^2 = 2 |\nabla u^n|^2 + 2 \Re \left( u^n \overline{\triangle u^n} \right)$ by
\begin{align*}
\|\mbox{I}\| &\lesssim \| u(t_n) \|_{L^{\infty}(\D)} \left( \| |\nabla u^n| + |\nabla u(t_n)| \|_{L^4(\D)} \| \nabla e^n_\CN \|_{L^4(\D)} \right.\\ 
&\qquad \left.+ \| e^n_\CN \|_{L^4(\D)} \| \triangle u(t^n) \|_{L^4(\D)} + \| u^n \|_{L^{\infty}(\D)}  \| \triangle e^n_\CN \|\right) \\
&\lesssim \| \nabla e^n_\CN \| + \| \triangle e^n_\CN \| \lesssim \tau^2  + \| \triangle e^n_\CN \|.
\end{align*}
Term II can be bounded as 
\begin{align*}
\|\mbox{II}\| \lesssim \| \nabla u(t_n) \|_{L^{\infty}(\D)} \left( 
\| u^n \|_{L^{\infty}(\D)} \| \nabla e^n_\CN \| + \| \nabla u(t^n) \|_{L^{\infty}(\D)}  \| e^n_\CN \|
\right) \lesssim \tau^2
\end{align*}
and term III easily as
\begin{align*}
\|\mbox{III}\| \lesssim \| \triangle u(t_n) \|_{L^{\infty}(\D)} \left( 
\| u^n \|_{L^{\infty}(\D)} + \| u(t_n) \|_{L^{\infty}(\D)}  \right) \| e^n \|
\lesssim \tau^2.
\end{align*}
Combining the three estimates yields
\begin{eqnarray*}
\|\triangle[(|u^n|^2-|u(t_n)|^2)u(t_n)]\| \lesssim \tau^2 +  \| \triangle e^n_\CN \|.
\end{eqnarray*}
%
%
With this the $H^2$-error recursion \eqref{H2-rec} becomes:
\begin{eqnarray}
\frac{\|\triangle e^{n+1}_\CN\|^2-\|\triangle e^n_\CN\|^2}{2\tau}\lesssim  \|\triangle e^{n+1}_\CN\|^2+\|\triangle e^n_\CN\|^2+ \tau^4 +\|T^n\|^2.
\end{eqnarray}
Gr\"onwall's inequality and Lemma \ref{error_pde-lemma} now yield the optimal estimate,
\[ \|\triangle e^{n+1}_\CN\| \lesssim  \tau^2.\]
This finishes the first part of the proof. 

Next, we prove the bound $\| \triangle (D_\tau u^{n-1/2}) \|_{H^2(\D)} \|\lesssim 1$. For that, we multiply the error recursion \eqref{error_pde} by $-\triangle^2(e^{n+1}_\CN-e^n_\CN)$. Recalling that $\triangle(e^{n+1}_\CN-e^n_\CN) \in H^1_0(\D)$ we can integrate by parts two times to obtain 
\begin{eqnarray*}
\lefteqn{\|\nabla \triangle e^{n+1}_\CN\|^2-\|\nabla\triangle e^n_\CN\|^2}\\
 &=& -\Re\bigg( \langle Ve^{n+1/2}_\CN,\triangle ^2 (e^{n+1}_\CN-e^n_\CN)\rangle +\langle e^n_{\beta,\text{\tiny CN}},\triangle^2(e^{n+1}_\CN-e^n_\CN)\rangle+\langle T^n,\triangle^2(e^{n+1}_\CN-e^n_\CN)\rangle  \bigg) \\
&\leq & |\langle \triangle(Ve^{n+1/2}_\CN),\triangle(e^{n+1}_\CN-e^n_\CN)\rangle| + |\langle \triangle e^n_{\beta,\text{\tiny CN}},\triangle(e^{n+1}_\CN-e^n_\CN)\rangle| +|\langle \triangle T^n,\triangle(e^{n+1}_\CN-e^n_\CN)\rangle| \\
&\lesssim & \tau^4+ \|\triangle T^n\|^2 \lesssim \tau^4.
\end{eqnarray*}
Thus we conclude
\[\|\nabla \triangle e^n_\CN\|\lesssim \tau^{3/2}.\]
Next we apply $\triangle$ to the error recursion \eqref{error_pde}, then multiply by $\triangle^2e^{n+1/2}_\CN$, integrate by parts for the $D_{\tau}$-term and consider the real part to find:
\begin{eqnarray*}
\lefteqn{\|\triangle^2 e^{n+1/2}_\CN\|^2}  \\
& \leq &  \tfrac{1}{\tau} |\langle \nabla \triangle e^{n+1}_\CN,\nabla \triangle e^n_\CN\rangle |+|\langle \triangle(Ve^{n+1/2}_\CN) + \triangle e^n_{\beta,\text{\tiny CN}} + \triangle T^n,\triangle^2 e^{n+1/2}_\CN\rangle |.
\end{eqnarray*}
With the previous estimate $\|\nabla \triangle e^n_\CN\|\lesssim \tau^{3/2}$ and Young's inequality we find:
\begin{eqnarray}\label{InftyH4avg}
\|\triangle^2e^{n+1/2}_\CN\|^2\lesssim \tau^2+\tau^4+\|\triangle T^n\|^2 \lesssim \tau^2.
\end{eqnarray}
It therefore follows $\|\triangle^2 e^{n+1/2}_\CN\|\lesssim \tau$ and $\|D_\tau(\triangle^2e^{n+1/2}_\CN)\|\leq C$. This finishes the argument, because $\triangle (D_\tau u^{n+1/2}_\CN) \in H^1_0(\D)$ and hence
$$
\| \triangle (D_\tau u^{n+1/2}_\CN) \|_{H^2(\D)}
\lesssim \| \triangle^2 (D_\tau u^{n+1/2}_\CN) \|
\lesssim \|D_\tau(\triangle^2e^{n+1/2}_\CN)\| + \| \triangle^{\hspace{-1pt}2} \hspace{-2pt} \left(\tfrac{u(t_{n+1}) - u(t_n)}{2\tau}\right) \| \lesssim 1.
$$
\end{proof}

Collecting all the previous results, we are now able to quantify how well $u^n$ and $D_{\tau} u^n$ are approximated in $V\LOD$.
\begin{conclusion}
\label{conclusion-ALOD-errors}
Assume \ref{A1}-\ref{splitting-potential} and let $u^n$ denote the solution to the semi-discrete method \eqref{semi-disc-cnd-problem}. If $A\LOD : H^1_0(\D) \rightarrow V\LOD$ denotes the $a(\cdot,\cdot)$-orthogonal projection onto the LOD space, i.e.,
$$
a( A\LOD(u) , v ) = a( u , v) \qquad \mbox{for all } v\in V\LOD,
$$
then we have the estimates
\begin{align}
\label{bounds-ALOD-un}
\| u^n - A\LOD(u^n) \| \lesssim H^4 \qquad \mbox{and} 
\qquad \| D_{\tau} u^{n-1/2} - A\LOD(D_{\tau} u^{n-1/2}) \| \lesssim H^4
\end{align}
with $D_\tau u^{n-1/2} = \tfrac{1}{2\tau}(u^{n+1}- u^{n-1} )$, as well as
\begin{align*}
 \| D_{\tau} u^n - A\LOD(D_{\tau} u^n) \| \lesssim H^4 + \tau^2.
\end{align*}
\end{conclusion}

\begin{proof}
Applying the general theory of Section \ref{section-LOD}, the first estimate follows from 
\begin{align*}
\| u^n - A\LOD(u^n) \| \lesssim H^4 \| \triangle u^n + V_1 u^n \|_{H^2(\D)},
\end{align*}
where $\| \triangle u^n + V_1 u^n \|_{H^2(\D)}$ is bounded by Lemma \ref{timedsc-regularity}. In a similar way, using Lemma \ref{optimal-H2-est-lemma} we have
\begin{align*}
\| D_{\tau} u^{n-1/2} - A\LOD(D_{\tau} u^{n-1/2}) \| \lesssim H^4 \| \triangle (D_{\tau} u^{n-1/2}) + V_1 D_{\tau} u^{n-1/2} \|_{H^2(\D)} \lesssim H^4.
\end{align*}
For the last estimate we use Lemma \ref{optimal-H2-est-lemma} which ensures that 
\begin{align}
\label{est-opt-Dtau-en}
\| D_{\tau} e^n \|_{H^2(\D)} = \tfrac{1}{\tau} \| e^n \|_{H^2(\D)} \lesssim \tau.
\end{align}
Consequently, we have
\begin{align*}
\| D_{\tau} u^n - A\LOD(D_{\tau} u^n) \| &\le 
\| D_{\tau} e^n - A\LOD(D_{\tau} e^n) \| + \tfrac{1}{\tau} \| u(t_{n+1}) - u(t_{n})  - A\LOD( u(t_{n+1}) - u(t_{n}))  \| \\
&\lesssim H^2 \| D_{\tau} e^n \|_{H^2(\D)} + H^4 \| \partial_t u \|_{L^{\infty}(t_n,t_{n+1};H^4(\D))} \\
&\overset{\eqref{est-opt-Dtau-en}}{\lesssim} H^2 \tau + H^4 \le \tfrac{3}{2} H^4 +\tfrac{1}{2} \tau^2.
\end{align*}
Note that we know that $\triangle D_\tau u^n \in H^2(\D)\cap H^1_0(\D)$ which allows for the direct estimate $\| D_{\tau} u^n - A\LOD(D_{\tau} u^n) \| \le C H^4 \|(-\triangle + V_1) D_\tau u^n \|_{H^2(\D)}$. However, we are lacking an estimate that guarantees that $\|\triangle D_\tau u^n \|_{H^2(\D)}$ can be bounded independently of $\tau$.
\end{proof}

As a last preparation for the final a priori error estimate, we also require regularity bounds for the $a(\cdot,\cdot)$-projection of a smooth function onto the LOD-space. We stress that the following lemma is only needed in the case $d=3$ to obtain optimal $L^{\infty}(L^2)$-error estimates for our method. In $1d$ and $2d$ the following lemma is not needed.

\begin{lemma}[$H^2$-regularity in the LOD space]\label{H2LOD}
Assume \ref{A1}-\ref{A4} and \ref{splitting-potential} and let $V\LOD$ be the LOD space given by \eqref{def-OD-space} with $a(\cdot,\cdot)$ defined in \eqref{def-a-innerproduct}. Then for any $w \in H^1_0(\D) \cap H^2(\D)$ the LOD approximation $w\LOD \in V\LOD$ with 
\begin{align}
\label{LOD-prob-in-reg-statement}
a(w\LOD , v ) = a(w , v ) \qquad \mbox{for all } v\in V\LOD
\end{align}
fulfills
$$
w\LOD \in H^1_0(\D) \cap H^2(\D) \qquad \mbox{with } \| w\LOD \|_{H^2(\D)} \le C \| w \|_{H^2(\D)},
$$
where $C$ only depends on $a(\cdot,\cdot)$, $\D$ and mesh regularity constants.

Furthermore, for any $v\LOD \in V\LOD$ we have the inverse estimates
\begin{align}\label{inverse-inequalities-LOD-space}
\| v\LOD \|_{H^1(\D)} \lesssim H^{-1} \| v\LOD \|
\qquad
\mbox{and}
\qquad
\| v\LOD \|_{L^{\infty}(\D)} \lesssim H^{-1} \| v\LOD \|_{H^1(\D)}.
\end{align}
\end{lemma}

\begin{proof}
To prove the regularity statement, we start with rewriting \eqref{LOD-prob-in-reg-statement} in a saddle point formulation. For that, we do not introduce the space $W$ explicitly, but we impose constraints through a Lagrange multiplier (cf. \cite{ENGWER2019123} for a corresponding formulation in a fully algebraic setting). The projection $w\LOD$ of $w$ onto the LOD space as given by \eqref{LOD-prob-in-reg-statement} can be equivalently characterized in the following way: find $Q_w \in H^1_0(\D)$ and $\lambda_H \in V_H$ such that 
\begin{align*}
a( Q_w , v ) - \langle \lambda_H , P_H(v) \rangle &= - a( P_H(w) , v ) \qquad \mbox{for all } v\in H^1_0(\D) \\
\langle P_H( Q_w ) , q_H \rangle &= 0 \hspace{81pt}  \mbox{for all } q_H \in V_H.
\end{align*}
It is easily seen that
$$
w\LOD = P_H(w) + Q_w
$$
and that $\lambda_H$ is the $L^2$-Riesz representer of the operator $a ( v\LOD ,\cdot )$ in $V_H$. Hence, $\lambda_H$ should be seen as an approximation of the \quotes{source term} $f= - \triangle w + V_1 w$. Since $P_H$ is the $L^2$-projection, the first equation in the saddle point system simplifies to
\begin{align*}
a( P_H(w) + Q_w , v ) &= \langle \lambda_H , v \rangle & \qquad \mbox{for all } v\in H^1_0(\D).
\end{align*}
Hence we can characterize $w\LOD \in H^1_0(\D)$ as the solution to
\begin{align*}
a( w\LOD  , v )  &= \langle \lambda_H , v \rangle 
 \qquad \mbox{for all } v\in H^1_0(\D).
\end{align*}
Since the coefficients in $a(\cdot,\cdot)$ are smooth and since $\D$ is convex, standard elliptic regularity yields $w\LOD \in H^2(\D)$ (cf. \cite[Theorem 3.2.1.2]{Gri85}) and 
$$
\| w\LOD  \|_{H^2(\D)} \lesssim \| \lambda_H \|.
$$
It remains to bound the $L^2$-norm of $\lambda_H$. Here we have 
\begin{align*}
\| \lambda_H \|^2 &= a ( w\LOD ,\lambda_H )
= a ( w\LOD - w,\lambda_H ) + \langle - \triangle w + V_1 w , \lambda_H \rangle \\ 
&\lesssim \| w\LOD - w \|_{H^1(\D)} \hspace{2pt} \| \lambda_H \|_{H^1(\D)}  + \| w \|_{H^2(\D)} \| \lambda_H \| \\
&\lesssim H \| w \|_{H^2(\D)} \hspace{2pt} \| \lambda_H \|_{H^1(\D)} + \| w \|_{H^2(\D)}\hspace{2pt} \| \lambda_H \|.
\end{align*}
To continue with this estimate, we apply the inverse inequality in the (quasi-uniform) finite element space $V_H$ which yields $ \| \lambda_H \|_{H^1(\D)} \le C H^{-1}  \| \lambda_H \| $ (cf. \cite{BrS08}). Consequently,
\begin{align*}
\| \lambda_H \|^2 &\lesssim H \| w \|_{H^2(\D)} \hspace{2pt} \| \lambda_H \|_{H^1(\D)} + \| w \|_{H^2(\D)}\hspace{2pt} \| \lambda_H \| \lesssim \| w \|_{H^2(\D)} \hspace{2pt} \| \lambda_H \|.
\end{align*}
Dividing by $\| \lambda_H \|$ yields $ \| \lambda_H \| \lesssim \| w \|_{H^2(\D)} $ and we can conclude 
$$
\| w\LOD  \|_{H^2(\D)} \lesssim \| \lambda_H \| \lesssim  \| w \|_{H^2(\D)} .
$$
Next, we prove the two inverse estimates. For that let $v\LOD=v_H + \mathcal{Q}(v_H) \in \VLOD$ be arbitrary. From $a( v_H + \mathcal{Q}(v_H) ,\mathcal{Q}(v_H)) = 0$ we conclude
 $$
 \| \mathcal{Q}(v_H)  \|_{H^1(\D)} \lesssim  \|  v_H \|_{H^1(\D)} 
 \qquad \mbox{and} \qquad 
 \| v_H + \mathcal{Q}(v_H)  \|_{H^1(\D)} \lesssim  \|  v_H \|_{H^1(\D)} .
 $$
 The $H^1$-stability of the $L^2$-projection in $V_H$ on quasi-uniform meshes (cf. \cite{BaY14}) implies 
  $$
 \|  v_H \|_{H^1(\D)} =
  \|  P_H (v_H + \mathcal{Q}(v_H)) \|_{H^1(\D)}  \lesssim  \| v_H + \mathcal{Q}(v_H)  \|_{H^1(\D)}.
 $$ 
  We conclude with the standard inverse estimate in finite element spaces
\begin{align*}
\| v_H + \mathcal{Q}(v_H)  \|_{H^1(\D)}^2 &\lesssim  \|  v_H \|_{H^1(\D)}^2  \lesssim  H^{-2} \|  v_H \|_{L^2(\D)}^2 = H^{-2} (  v_H ,  v_H + \mathcal{Q}(v_H))_{L^2(\D)}\\
&= H^{-2} \| v_H + \mathcal{Q}(v_H)\|_{L^2(\D)}^2 + H^{-2} ( \mathcal{Q}(v_H) ,  v_H + \mathcal{Q}(v_H))_{L^2(\D)}\\
&\lesssim H^{-2} \| v_H + \mathcal{Q}(v_H)\|_{L^2(\D)}^2 + \tfrac{\eps}{H^{2}} \| \mathcal{Q}(v_H) \|_{L^2(\D)}^2 + \tfrac{1}{\eps H^{2}} \| v_H + \mathcal{Q}(v_H) \|_{L^2(\D)}^2 \\
&\lesssim (1+\eps^{-1}) H^{-2} \| v_H + \mathcal{Q}(v_H)\|_{L^2(\D)}^2 + \tfrac{\eps}{H^{2}} H^2 \| \mathcal{Q}(v_H) \|_{H^1(\D)}^2  \\
&\lesssim (1+\eps^{-1}) H^{-2} \| v_H + \mathcal{Q}(v_H)\|_{L^2(\D)}^2 + \eps  \|  v_H + \mathcal{Q}(v_H) \|_{H^1(\D)}^2,
\end{align*}
where $\eps>0$ is a sufficiently small parameter resulting from the application of Young's inequality. Hence, we have the inverse estimate
\begin{align*}
\| v_H + \mathcal{Q}(v_H)  \|_{H^1(\D)} &\lesssim  \tfrac{1+\eps^{-1}}{1-\eps} H^{-1} \| v_H + \mathcal{Q}(v_H)\|_{L^2(\D)}.
\end{align*}
For the $L^{\infty}$-inverse estimate, we note that $v_H + \mathcal{Q}(v_H) \in H^2(\D)$ because if $\lambda_H \in V_H$ is defined by $\langle \lambda_H , q_H \rangle = a( v_H + \mathcal{Q}(v_H) , q_H)$ for all $q_H \in V_H$, then  $v_H + \mathcal{Q}(v_H)  \in H^1_0(\D)$ solves the regular boundary value problem
$$
a( v_H + \mathcal{Q}(v_H)  , v ) = \langle \lambda_H , v \rangle \qquad \mbox{for all } v \in H^1_0(\D).
$$
We conclude with elliptic regularity theory that
\begin{align*}
\| v_H + \mathcal{Q}(v_H)  \|_{L^{\infty}(\D)} \lesssim
\| v_H + \mathcal{Q}(v_H)  \|_{H^2(\D)}
 \lesssim \| \lambda_H \|_{L^2(\D)}
 \lesssim  H^{-1} \| v_H + \mathcal{Q}(v_H)  \|_{H^{1}(\D)}.
\end{align*}
\end{proof}

We are now ready to prove the superconvergence for the $L^\infty(L^2)$-error.
\begin{lemma}(Optimal $L^2$-error estimates)
\label{lemma-op-L2-est}
Assume \ref{A1}-\ref{splitting-potential}.  Then there is a solution $u^{n}\LOD \in\VLOD$ to the modified Crank--Nicolson method \eqref{FullyDiscrete}, with uniform $L^{\infty}$-bounds, i.e., there exists a constant $C>0$ (independent of $\tau$ and $H$) such that
\begin{align}
\label{Linfty-un-LOD}
\max_{0 \le n \le N} \| u^{n}\LOD \|_{L^{\infty}(\D)} \le C
\end{align}
and the $L^\infty(L^2)$-error between $u^{n}\LOD$ and the exact solution $u$ at time $t_n$ converges with
 $$ \sup_{0\leq n\leq N}\|u^{n}\LOD-u(\cdot,t_n)\|\lesssim 
 \tau^2+H^4.
 $$
 \end{lemma}
\begin{proof}	
In the following we denote by $u^n$ the solution to the semi-discrete Crank--Nicolson method \eqref{semi-disc-cnd-problem}. As in the proof of existence we introduce an auxiliary problem with a truncated nonlinearity. The reason for this is that for the truncated problem the necessary $L^\infty$-bounds are available. Once this error estimate is obtained it is possible to show that for sufficiently small $H$ the truncation engenders no change. Given a sufficiently large constant $M> 1 + \sup_{0 \le n \le N} \| u^n \|_{L^{\infty}(\D)}^2 $, the truncated problem reads: find $u^{n+1,(M)}\LOD \in V\LOD$ with\\
\begin{eqnarray}\label{TruncatedLOD}
	\lefteqn{\ci \big\langle D_\tau u^{n,(M)}\LOD,v\big\rangle = \big\langle\nabla \uLOD^{n+1/2,(M)},\nabla v \big\rangle + \big\langle V\uLOD^{n+1/2,(M)},v\big\rangle}  \\
\nonumber	&\enspace& \quad+ \beta\langle\frac{P\LOD\big(\chi_M(|u^{n+1,(M)}\LOD|^2)+ \chi_M(|u^{n,(M)}\LOD|^2)\big)}{2} \frac{\chi_M(u^{n+1,(M)}\LOD) + \chi_M(u^{n,(M)}\LOD)}{2}  ,v\rangle
\end{eqnarray}
for all $v \in V\LOD$, where  $\chi_M : \C \rightarrow \{ z \in \C |\hspace{2pt} |z|\le M \}$ is the Lipschitz-continuous truncation function given by
\begin{align*}
\chi_M(z) := \min\{ \tfrac{M}{|z|} , 1 \} \hspace{2pt} z.
\end{align*}
Note that the Lipschitz constant is $2$, i.e.,
\begin{align}
\label{lipschitz-continuity-chiM}
| \chi_M(z) - \chi_M(y) | \le 2 |x - y| \qquad \mbox{for all } x,y\in \C.
\end{align}
Also observe that $|\chi_M(z)|\le M$ and $\chi_M(z)=z$ for all $z \in \C$ with $|z|\le M$. For real values $x\in \R$ we have $\chi_M(x)=M$ if $x\ge M$.
Existence of truncated solutions $u^{n,(M)}$ follows analogously to the case without truncation. Thanks to previous optimal $L^\infty(L^2)$ estimates of the semi-discrete problem \eqref{semi-disc-cnd-problem} (cf. Lemma \ref{lemma-semidiscrete-CN}), it will suffice to prove an optimal estimate for $\|A\LOD(u^n-u^{n,(M)}\LOD)\|$. This is made clear by splitting the error into:
\begin{align}
\nonumber\|u(t_n)-u^{n,(M)}\LOD\|&\leq \|u(t_n)-u^n\| +\|u^n-u^{n,(M)}\LOD\| \\
\nonumber&\leq C\tau^2+ \| u^n-A\LOD(u^n)\| + \|A\LOD(u^n)-u^{n,(M)}\LOD\|  \\
\label{initial-error-splitting}& \lesssim \tau^2+CH^4+\|A\LOD(u^n)-u^{n,(M)}\LOD\|,
\end{align}
where Conclusion \ref{conclusion-ALOD-errors} was used. We define $e^{n,(M)} := u^n-u^{n,(M)}\LOD$ and its $a$-orthogonal projection onto $V\LOD$ shall be denoted by $e^{n,(M)}\LOD := A\LOD(u^n)-u^{n,(M)}\LOD$. Subtracting \eqref{TruncatedLOD} from \eqref{semi-disc-cnd-problem} yields
\begin{eqnarray}\label{errorpdelod}
\ci \langle D_\tau e^{n,(M)},v \rangle = a(A\LOD(u^{n+1/2})-u^{n+1/2,(M)}\LOD,v)+ \langle V_2 \hspace{2pt}e^{n+1/2,(M)} + \beta \hspace{2pt}e^{n+1/2,(M)}_{\beta},v\rangle
\end{eqnarray}
for all $v\in V\LOD$ where 
\begin{eqnarray*}
\lefteqn{e^{n,(M)}_\beta := \frac{|u^{n+1}|^2+|u^n|^2}{2}u^{n+1/2} }\\
&\enspace& \qquad -P\LOD\bigg(\frac{\chi_M(|u\LOD^{n+1,(M)}|^2)+\chi_M(|u\LOD^{n,(M)}|^2)}{2}\bigg) \frac{\chi_M(u^{n+1,(M)}\LOD) + \chi_M(u^{n,(M)}\LOD)}{2} .
\end{eqnarray*}
Taking $v=e^{n+1/2,(M)}\LOD  = \tfrac{1}{2}(e^{n,(M)}\LOD + e^{n+1,(M)}\LOD)$ in \eqref{errorpdelod} and considering the imaginary part yields a recursion formula for the error: 
\begin{eqnarray} \label{Error-recursion}
\frac{\|e^{n+1,(M)}\LOD\|^2-\|e^{n,(M)}\LOD\|^2}{2\tau}= \Im \bigg( \langle \V \hspace{2pt} e^{n+1/2,(M)},e^{n+1/2,(M)}\LOD\rangle + \beta\langle e^{n,(M)}_\beta,e^{n+1/2,(M)}\LOD \rangle \bigg) \\
- \Re \langle D_\tau(u^n-A\LOD(u^n)),e^{n+1/2,(M)}\LOD\rangle . \nonumber 
\end{eqnarray}
Our first goal will be to bound $|\langle e^{n,(M)}_\beta,e^{n+1/2,(M)}\LOD \rangle|$ in terms of $H^8, \|e^{n,(M)}\LOD\|^2$ and $\|e^{n+1,(M)}\LOD\|^2$. For that we split $e^{n,(M)}_{\beta} = \tfrac{1}{4}e^{n,(M)}_{\beta,1} + \tfrac{1}{4}e^{n,(M)}_{\beta,2} + \tfrac{1}{4}e^{n,(M)}_{\beta,3}$, where
\begin{align*}
e^{n,(M)}_{\beta,1} &:=   P\LOD\left( |u^{n+1}|^2+|u^n|^2  - \chi_M(|u\LOD^{n+1,(M)}|^2) - \chi_M(|u\LOD^{n,(M)}|^2)\right) \\
&\hspace{60pt}\left( \chi_M(u^{n+1,(M)}\LOD) + \chi_M(u^{n,(M)}\LOD)\right); \\
e^{n,(M)}_{\beta,2} &:= (\mbox{\rm Id} -P\LOD)  \left(|u^{n+1}|^2+|u^n|^2 \right) \left( \chi_M(u^{n+1,(M)}\LOD) + \chi_M(u^{n,(M)}\LOD)\right);\\
e^{n,(M)}_{\beta,3} &:= \left(|u^{n+1}|^2+|u^n|^2 \right) 
\left(u^{n+1}+ u^n  - \chi_M(u^{n+1,(M)}\LOD) - \chi_M(u^{n,(M)}\LOD)\right).
\end{align*}
Estimating the various terms, we obtain
\begin{eqnarray*}
\lefteqn{|\langle e^{n,(M)}_{\beta,1} , e^{n+1/2,(M)}\LOD \rangle |} \\ 
&\lesssim&
M \| P\LOD\left( |u^{n+1}|^2+|u^n|^2  - \chi_M(|u\LOD^{n+1,(M)}|^2) - \chi_M(|u\LOD^{n,(M)}|^2)\right) \| \hspace{2pt} \| e^{n+1/2,(M)}\LOD \| \\
&\lesssim& M \| |u^{n+1}|^2+|u^n|^2  - \chi_M(|u\LOD^{n+1,(M)}|^2) - \chi_M(|u\LOD^{n,(M)}|^2) \| \hspace{2pt} \| e^{n+1/2,(M)}\LOD \|  \\
&\lesssim& M^{3/2} \left( \| u^{n}  - u\LOD^{n,(M)} \| +  \| u^{n+1}  - u\LOD^{n+1,(M)} \| \right) \hspace{2pt} \| e^{n+1/2,(M)}\LOD \|,
\end{eqnarray*}
where in the last step we used that pointwise
\begin{eqnarray*}
\lefteqn{\left|  \chi_M(|u\LOD^{n,(M)}|^2) - |u^n|^2 \right|}\\
&\le& \begin{cases}
| u\LOD^{n,(M)} - u^n | \hspace{2pt} (|u^n| + |u\LOD^{n,(M)}|)
\le 2 \sqrt{M} | u\LOD^{n,(M)} - u^n| 
&\mbox{if } |u\LOD^{n,(M)}|^2 \le M;\\
M - |u^n|^2 \le 2 \sqrt{M}\hspace{2pt} (\sqrt{M} - |u^n|)
\le 2 \sqrt{M} |u\LOD^{n,(M)} - u^n |  &\mbox{if } |u\LOD^{n,(M)}|^2 > M.
\end{cases}
\end{eqnarray*}
For the second term we have
\begin{eqnarray*}
\lefteqn{|\langle e^{n,(M)}_{\beta,2} , e^{n+1/2,(M)}\LOD \rangle |} \\ 
&\lesssim&
M \| (\mbox{\rm Id} -P\LOD)  \left(|u^{n+1}|^2+|u^n|^2 \right) \| \hspace{2pt} \| e^{n+1/2,(M)}\LOD \| \\
&\lesssim& M H^4 \| (-\triangle + V_1) \left(|u^{n+1}|^2+|u^n|^2 \right) \|_{H^2(\D)} \hspace{2pt} \| e^{n+1/2,(M)}\LOD \|,
\end{eqnarray*}
where it remains to bound the term $\| \triangle \left(|u^{n+1}|^2+|u^n|^2 \right) \|_{H^2(\D)} \lesssim \| \triangle^2 \left(|u^{n+1}|^2+|u^n|^2 \right) \|$. Using $ \triangle^2 |u^n|^2 = 6 |\triangle u^n|^2 + 8 \Re\left( \nabla u^n \hspace{2pt} \overline{\nabla \triangle u^n} \right) + 2 \Re \left( u^n \overline{\triangle^2 u^n} \right)$ and the estimates
\begin{align*}
	\|\nabla u^n \nabla \triangle u^n\| &\le \|\nabla u^n\|_{L^4(\D)}^2\|\nabla\triangle u^n\|_{L^4(\D)}^2 \lesssim \|u^n\|_{H^2(\D)}^2 \|\triangle u^n\|_{H^2(\D)}^2; \\
\||\triangle u^n|^2\| &\leq \|\triangle u^n\|_{L^\infty(\D)} \|\triangle u^n\| \lesssim 
\|\triangle u^n\|_{H^2(\D)}^2 \qquad \mbox{and}\\
\|u^n\triangle^2 u^n\| &\lesssim \|u^n\|_{L^\infty(\D)}\|\triangle u^n\|_{H^2(\D)}
\end{align*}
we see with Lemma \ref{lemma-semidiscrete-CN} and Lemma \ref{timedsc-regularity} that $\| \triangle \left(|u^{n+1}|^2+|u^n|^2 \right) \|_{H^2(\D)} \lesssim 1$ and hence $|\langle e^{n,(M)}_{\beta,2} , e^{n+1/2,(M)}\LOD \rangle | \lesssim M H^4 \| e^{n+1/2,(M)}\LOD \|$.

It remains to bound $|\langle e^{n,(M)}_{\beta,3} , e^{n+1/2,(M)}\LOD \rangle |$. Here we can readily use the $L^{\infty}$-bounds for $u^n$ (with $\chi_M(u^n)=u^n$ for all $n$) together with the Lipschitz-continuity \eqref{lipschitz-continuity-chiM} to conclude that
\begin{eqnarray*}
|\langle e^{n,(M)}_{\beta,3} , e^{n+1/2,(M)}\LOD \rangle | 
&\lesssim& \left( \| u^n  - u^{n,(M)}\LOD \| + \| u^{n+1}  - u^{n+1,(M)}\LOD \| \right)
 \| e^{n+1/2,(M)}\LOD \|.
\end{eqnarray*}
Combing the estimates for $e^{n,(M)}_{\beta,1}$, $e^{n,(M)}_{\beta,2}$ and $e^{n,(M)}_{\beta,3}$, we have
\begin{eqnarray}
\nonumber|\langle e^{n,(M)}_{\beta} , e^{n+1/2,(M)}\LOD \rangle | 
&\lesssim& M^{3/2} \left( \| e^{n,(M)} \| + \|e^{n+1,(M)} \| \right)
 \| e^{n+1/2,(M)}\LOD \| + M H^4  \| e^{n+1/2,(M)}\LOD \| \\
\label{est-enMbeta} &\le& C(M) \left( H^8 +  \| e^{n,(M)}\LOD \|^2 + \| e^{n+1,(M)}\LOD \|^2  \right)
\end{eqnarray}
for some constant $C(M)=\mathcal{O}(M^{3/2})$. Recalling the initial error recursion formula \eqref{Error-recursion}, we conclude with \eqref{est-enMbeta} that
\begin{eqnarray}
\label{ErrorRecursion2}\lefteqn{\frac{\|e^{n+1,(M)}\LOD\|^2-\|e^{n,(M)}\LOD\|^2}{2\tau}} \\ 
\nonumber&\leq&  \|\V\|_{L^{\infty}(\D)} \|e^{n+1/2,(M)}\| \hspace{2pt} \|e^{n+1/2,(M)}\LOD\|+ C(M) \left( H^8 +  \| e^{n,(M)}\LOD \|^2 + \| e^{n+1,(M)}\LOD \|^2  \right)\\
\nonumber&\enspace&\qquad -\Re \langle D_\tau(u^n-A\LOD(u^n)),e^{n+1/2,(M)}\LOD\rangle \\
\nonumber&\leq& C(M) \left( H^8 +  \| e^{n,(M)}\LOD \|^2 + \| e^{n+1,(M)}\LOD \|^2  \right) -\Re \langle D_\tau(u^n-A\LOD(u^n)),e^{n+1/2,(M)}\LOD\rangle.
\end{eqnarray}
It follows from Lemma \ref{timedsc-regularity} and Conclusion \ref{conclusion-ALOD-errors} that $\triangle D_\tau u^n \in H^2(\D)\cap H^1_0(\D)$ and the estimate $\|D_\tau u^n-A\LOD(D_\tau u^n)\|\lesssim \tau^2 + H^4$. However, in order to avoid unnecessary coupling conditions between the mesh size and the time step size, we cannot afford a $\tau^2$-dependency at this point. Instead we only want to use the estimate $\| D_{\tau} u^{n-1/2} - A\LOD(D_{\tau} u^{n-1/2}) \| \lesssim H^4$ proved in Conclusion \ref{conclusion-ALOD-errors}. 
In order to exploit it, we sum  up recursion \eqref{ErrorRecursion2} to find:
 \begin{equation}\label{LastStep}
 \|e^{n+1,(M)}\LOD\|^2 \le C(M) \left( H^8+ \tau \sum_{k=0}^n\|e^{k,(M)}\LOD\|^2 \right)+ \tau |\sum_{k=0}^n \langle D_\tau(u^n-A\LOD(u^n)),e^{n+1/2,(M)}\LOD\rangle |.
 \end{equation} 
The idea is now to reformulate the expression above in such a way that we can use our optimal bounds for  $\| D_{\tau} u^{n-1/2} - A\LOD(D_{\tau} u^{n-1/2}) \|$ to estimate the last term. To this end we will use the following summation formula:
\begin{align*}
\sum_{k=0}^n D[a^k]b^{k+1/2} &=\frac{1}{2}\big(D[a^n]b^{n+1}+D[a^0]b^0\big)+\sum_{k=1}^{n}D[a^{k-1/2}]b^k.
\end{align*}

When applied to our sum, the formula yields
\begin{eqnarray*}
\lefteqn{\tau |\sum_{k=0}^n \langle D_\tau u^k- A(D_\tau u^k),e\LOD^{k+1/2,(M)}\rangle|}  \\
& \leq& \frac{\tau}{2}|\langle D_\tau u^0-A(D_\tau u^0),e\LOD^{0,(M)}\rangle| +\frac{\tau}{2} |\langle D_\tau u^{n}-A(D_\tau u^{n}),e\LOD^{n+1,(M)}\rangle| \\
&\enspace&\qquad +\tau \left|\sum_{k=1}^n\langle D_\tau u^{k-1/2}-A\LOD(D_\tau u^{k-1/2}),e\LOD^{k,(M)}\rangle\right| \\
&\overset{\eqref{bounds-ALOD-un}}{\lesssim}& H^8+\tau^2 \|e\LOD^{0,(M)}\|^2+\tau^2\|e\LOD^{n+1,(M)}\|^2 + \tau\sum_{k=1}^n H^4 \|e^{k,(M)}\LOD\| \\
&\lesssim& H^8+\tau^2\|e\LOD^{n+1,(M)}\|^2 + \tau\sum_{k=1}^n  \|e^{k,(M)}\LOD\|^2.
\end{eqnarray*}
With $0<(1-\tau^2)^{-1} \lesssim 1$, estimate \eqref{LastStep} thus becomes
\begin{equation}
\|e^{n+1,(M)}\LOD\|^2 \le C(M) \left( H^8+\tau\sum^n_{k=0}\|e^{k,(M)}\LOD\|^2 \right).
\end{equation}
Gr\"onwall's inequality now readily gives us the estimate
\begin{equation}
\label{final-estimate-with-M}
\|e^{n+1,(M)}\LOD\| \le C(M) H^4
\end{equation}
for some new constant $C(M)$ that depends exponentially on $M$.

To conclude the argument, we need to show that $M$ can be selected independent of $H$ and $\tau$, so that $u^{n}\LOD = u^{n,(M)}\LOD$.
For that we can use the inverse inequalities in Lemma \ref{H2LOD} to show with the following calculation that $\|u^{n,(M)}\LOD\|_{L^\infty(\D)}$ and $\|u^{n,(M)}\LOD\|_{L^\infty(\D)}^2$ are bounded by a constant less than $M$ for sufficiently small $H$. We have
\begin{eqnarray*}
\|u^{n,(M)}\LOD\|_{L^\infty(\D)} &\leq& \|u^{n,(M)}\LOD-A\LOD(u^n)\|_{L^\infty(\D)}+\|A\LOD(u^n)\|_{L^\infty(\D)}\\
&\overset{\eqref{inverse-inequalities-LOD-space}}{\leq}& H^{-2}\|u^{n,(M)}\LOD-A\LOD(u^n)\|+\|u^n\|_{H^2} \\
&\leq & C(M) H^{-2}e^{n,(M)}\LOD+ C_0 \\
&\leq & C(M)H^2+C_0.
\end{eqnarray*} 
Hence, if $M$ is selected so that $M \ge (1+C_0)^2$, then for any $H \le C(M)^{-1/2}$ we have $\|u^{n,(M)}\LOD\|_{L^\infty(\D)} \le 1+C_0 < \sqrt{M} < M$ and $\|u^{n,(M)}\LOD\|_{L^\infty(\D)}^2 \le (1+C_0)^2 < M$. Consequently, the truncation in problem \eqref{TruncatedLOD} can be dropped and we have $u^{n,(M)}\LOD = u^{n}\LOD$ for any fixed $M \ge (1+C_0)^2$ and any sufficiently small $H$. The truncated problem coincides with the original problem and we have from \eqref{final-estimate-with-M} that $\|u^{n}\LOD - A\LOD(u^n) \|\lesssim H^4$. 
Together with \eqref{initial-error-splitting}, this finishes the proof.
\end{proof}

With the optimal a priori error estimate available, we can now draw a conclusion on the accuracy of the exact energy.
\begin{corollary}\label{ExactEnergyThm}
Assume \ref{A1}-\ref{splitting-potential} and let $u^{n}\LOD \in\VLOD$ denote Crank--Nicolson approximation with uniform $L^{\infty}$-bounds appearing in Lemma \ref{lemma-op-L2-est}. Then the conserved energy  $E\LOD[u\LOD]$ differs from $E[u\LOD]$ by at most of $\mathcal{O}(H^8)$ and $E[u\LOD]$ itself differs at most of  $\mathcal{O}(H^6)$ from the exact energy. To be precise, we have
\begin{align*}
\left| \hspace{2pt} E\LOD[u\LOD^n] - E[u\LOD^n] \hspace{2pt}  \right|
\lesssim H^8 \qquad \mbox{and}
\qquad
\left| \hspace{2pt} E[u\LOD^n] - E[u(t_n)] \hspace{2pt}  \right|
\lesssim H^6.
\end{align*}
\end{corollary}
\begin{proof}
First, we investigate the difference between the exact energy $E$ of $u\LOD^n$ compared to the preserved modified $E\LOD[u\LOD^n]$ and find that
\begin{eqnarray*}
	\lefteqn{E[u\LOD^n] -E\LOD[u\LOD^n] =  \langle |u^n\LOD|^2,|u^n\LOD|^2\rangle - \langle P\LOD(|u^n\LOD|^2),P\LOD(|u^n\LOD|^2)\rangle} \\
		& =& \langle u^n\LOD|^2, |u^n\LOD|^2-P\LOD(|u^n\LOD|^2)\rangle \\
		& =& \||u^n\LOD|^2-P\LOD(|u^n\LOD|^2)\|^2 \\
		& \leq& \left(\||u^n\LOD|^2-|u^n|^2\|+\||u^n|^2-P\LOD(|u^n|^2)\|+\|P\LOD(|u^n|^2-|u^n\LOD|^2)\| \right)^2 \\
		&\leq& \big( 2\||u^n\LOD|^2-|u^n|^2\|+\||u^n|^2-P\LOD(|u^n|^2)\|\big)^2\\
		&\le& \big( 2\| u^n\LOD - u^n \| \hspace{2pt} \| |u^n\LOD| + |u^n| \|_{L^{\infty}(\D)} + H^4 \| (-\triangle + V_1) |u^n|^2 \|\big)^2\\
		&\overset{\eqref{Linfty-un-LOD},\eqref{final-estimate-with-M}}{\lesssim}& H^8.
\end{eqnarray*}
For the exact energies we only have to estimate the remaining difference $E\LOD[u\LOD^n] - E[u(t_n)]$. Here we have with the conservation properties
\begin{eqnarray*}
\lefteqn{|E\LOD[u\LOD^n]-E[u(t_n)]| = |E\LOD[u\LOD^0]-E[u^0]|} \\
&\leq& |E\LOD[u^0\LOD] -E[u^0\LOD]| +|E[u^0\LOD]-E[u^0]|  \lesssim H^8 + H^6,
\end{eqnarray*}
where we used the energy estimate from Theorem \ref{theorem-superapproximation} in the last step.
\end{proof}
Collecting all the results of this section proves the statements of Theorem \ref{main-theorem-modified-CN}.

\appendix

\section{Existence of solutions to the standard Crank--Nicolson scheme}

In the following we prove the existence result stated in Lemma \ref{lemma-existence-solutions-classiccal-CN}, i.e., under assumptions \ref{A1}-\ref{A3}, there exists at least one solution $u^{\CN n}\LOD \in \VLOD$ to the Crank--Nicolson scheme \eqref{Crank}. 

\begin{proof}[Proof of Lemma \ref{lemma-existence-solutions-classiccal-CN}]
The proof deviates slightly from the existence proof of Lemma \ref{existence-modified-CN-LOD}, mainly because we do not have to take care of a term such as $P\LOD(|u^{n+1}\LOD|^2+|u^{n}\LOD|^2)$ for which we could not guarantee positivity.

Again, let $N_H$ denote the dimension of $V\LOD$ with basis $\{ \phi_{\ell} \hspace{2pt} | 1 \le \ell \le N_H \}$. We note that the proof holds for any finite dimensional space if $P\LOD$ is the $L^2$-projection into that space.\\[0.5em]
For $n \ge 1$ we can express the problem of finding $ u^{\CN n+1}\LOD \in V\LOD$ as
\begin{eqnarray}
\label{CN-LOD-1-algebraic}\nonumber\lefteqn{ 0 = - \tau^{-1} \ci \langle  u^{\CN n+1}\LOD  ,  \phi_{\ell} \rangle 
+ \tau^{-1} \ci \langle  u^{\CN n}\LOD  ,  \phi_{\ell} \rangle
 \hspace{2pt}
+ \langle \nabla u^{\CN n+\tfrac{1}{2}}\LOD , \nabla  \phi_{\ell} \rangle
+ \hspace{2pt} \langle V u^{\CN n+\tfrac{1}{2}}\LOD ,   \phi_{\ell} \rangle
 }\\
&\enspace& \qquad+ \hspace{2pt} \beta \left\langle \frac{|u^{\CN n+1}\LOD|^2+|u^{\CN n}\LOD|^2}{2}u^{\CN n+1/2}\LOD ,  \phi_{\ell}\right\rangle
\hspace{140pt}
\end{eqnarray}
for all $ \phi_{\ell}$. Inductively, we assume that $u^{\CN n}\LOD \in V\LOD $ exists. 
Again, we want to use the variation of the Browder fixed-point theorem \cite[Lemma 4]{Browder1965} to show the existence of $u^{\CN n+1}\LOD\in V\LOD$ (cf. the proof to Lemma \ref{existence-modified-CN-LOD}). For that, we reformulate the problem to a problem on $\mathbb{C}^{N_H}$, by defining a function $g : \C^{N_H} \rightarrow \C^{N_H}$ for $\boldsymbol{\alpha}\in \C^{N_H}$ through
\begin{eqnarray*}
\lefteqn{g_\ell(\boldsymbol{\alpha}) := - \frac{\ci}{\tau} \sum_{m=1}^{N_H} \boldsymbol{\alpha}_m \langle  \phi_m , \phi_\ell \rangle 
 +
\frac{1}{2} \sum_{m=1}^{N_H} \boldsymbol{\alpha}_m \hspace{2pt} 
\langle \nabla \phi_m , \nabla \phi_\ell \rangle
 +
\frac{1}{2} \sum_{m=1}^{N_H} \boldsymbol{\alpha}_m \hspace{2pt} 
\langle V \phi_m  , \phi_\ell \rangle
}\\
&\enspace&
+ \frac{\beta}{4} \langle \left( \left| \sum_{m=1}^{N_H} \boldsymbol{\alpha}_m \phi_m  \right|^2 + |u^{\CN n}\LOD|^2 \right)
\left( \sum_{m=1}^{N_H} \boldsymbol{\alpha}_m \phi_m \hspace{2pt} + u^{\CN n}\LOD
\right)
 ,\phi_\ell \rangle
 +F_\ell,
\end{eqnarray*}
where $F\in \C^{N_H}$ is defined by
\begin{align*}
F_\ell := \frac{1}{2} \langle \nabla u^{\CN n}\LOD ,  \nabla \phi_\ell \rangle
+ \frac{1}{2} \langle V u^{\CN n}\LOD ,  \phi_\ell \rangle
 + \ci \tau^{-1} \langle u^{\CN n}\LOD  , \phi_\ell \rangle.
\end{align*}
To verify existence of $\boldsymbol{\alpha}_0$ with $g(\boldsymbol{\alpha}_0)=0$, we need to show that there exists $K \in \R_{>0}$ such that $\Re( g(\boldsymbol{\alpha}) \cdot \boldsymbol{\alpha} ) > 0$ for all $\boldsymbol{\alpha} \in \C^{N_H}$ with $|\boldsymbol{\alpha}|= K$. 
For brevity, we define again $z_\alpha:=\sum_{m=1}^{N_H} \boldsymbol{\alpha}_m \phi_m$. First, we observe with the Young inequality, $2 |a| \hspace{2pt} |b| \le |a|^2 +|b|^2$ that
\begin{eqnarray*}
\lefteqn{\frac{\beta}{4} \Re \langle (\left|z_\alpha \right|^2 + |u^{\CN n}\LOD|^2 )  \left( z_\alpha + u^{\CN n}\LOD \right) , z_\alpha \rangle}\\
 &\ge&
\frac{\beta}{4} \langle \left|z_\alpha \right|^2 + |u^{\CN n}\LOD|^2   , |z_\alpha|^2  \rangle
- \frac{\beta}{4} \langle \left|z_\alpha \right|^2 + |u^{\CN n}\LOD|^2  , |u^{\CN n}\LOD| \hspace{2pt} |z_\alpha|  \rangle \\
&\ge& \frac{\beta}{4}  \langle |z_\alpha|^2 + |u^{\CN n}\LOD|^2 , |z_\alpha|^2 \rangle
- \frac{\beta}{8} \langle ( \left|z_\alpha \right|^2  + |u^{\CN n}\LOD|^2 ) ,  \left|z_\alpha \right|^2  + |u^{\CN n}\LOD|^2 \rangle\\
&=& 
\frac{\beta}{8}  \langle |z_\alpha|^2 + |u^{\CN n}\LOD|^2 , |z_\alpha|^2 \rangle
- \frac{\beta}{8} \langle ( \left|z_\alpha \right|^2  + |u^{\CN n}\LOD|^2 ) ,   |u^{\CN n}\LOD|^2 \rangle\\
&=& \frac{\beta}{8}  \langle |z_\alpha|^2, |z_\alpha|^2 \rangle  - \frac{\beta}{8} \langle |u^{\CN n}\LOD|^2 , |u^{\CN n}\LOD|^2 \rangle
=  \frac{\beta}{8} \left( \| z_{\alpha} \|_{L^4(\D)}^4 -\| u^{\CN n}\LOD \|_{L^4(\D)}^4  \right).
\end{eqnarray*}
Using this inequality, we get
\begin{eqnarray*}
\lefteqn{\Re( g(\boldsymbol{\alpha}) \cdot \boldsymbol{\alpha} ) =
\frac{1}{2}\| \nabla z_\alpha \|^2
+ \frac{1}{2} \langle V z_\alpha , z_\alpha \rangle
+ \Re( F \cdot \boldsymbol{\alpha} )} \\
&\enspace& \qquad + 
\frac{\beta}{4} \Re \langle (\left|z_\alpha \right|^2 + |u^{\CN n}\LOD|^2 )  \left( z_\alpha + u^{\CN n}\LOD
\right)
, z_\alpha \rangle\\
&\ge& \frac{1}{2}\| \nabla z_\alpha \|^2
- \left( \frac{1}{2} \| V \|_{L^{\infty}(\mathcal{D})} + \tau^{-1} \right)
 \|  u^{\CN n}\LOD \| \hspace{2pt} \| z_\alpha \|
- \frac{1}{2} \| \nabla u^{\CN n}\LOD \| \hspace{2pt} \| \nabla  z_\alpha \|
\\
&\enspace& \qquad 
+ \frac{\beta}{8} \left( \| z_{\alpha} \|_{L^4(\D)}^4 -\| u^{\CN n}\LOD \|_{L^4(\D)}^4  \right)
\\
&\ge& \frac{1}{2}\| \nabla z_\alpha \|^2
- \left( \frac{1}{2} \| V \|_{L^{\infty}(\mathcal{D})} + \tau^{-1} \right)
 \|  u^{\CN n}\LOD \| \hspace{2pt} \| z_\alpha \|
- \frac{1}{2} \| \nabla u^{\CN n}\LOD \| \hspace{2pt} \| \nabla  z_\alpha \|
- \frac{\beta}{8} \| u^{\CN n}\LOD \|_{L^4(\D)}^4
\\
&\ge& \frac{1}{2}\| \nabla z_\alpha \| \hspace{2pt} \left( \| \nabla z_\alpha \|
- \sqrt{2} \hspace{2pt} \mbox{\rm diam}(\mathcal{D}) 
\left( \frac{1}{2} \| V \|_{L^{\infty}(\mathcal{D})} + \tau^{-1} \right)
\| u^{\CN n}\LOD \| \hspace{2pt}
- \| \nabla u^{\CN n}\LOD \| \right)\\
&\enspace& \qquad 
- \frac{\beta}{8} \| u^{\CN n}\LOD \|_{L^4(\D)}^4
\\
&\ge& \frac{1}{2} \| \nabla z_\alpha \| \hspace{2pt} \left( \| \nabla z_\alpha \| - C_1 \right) - C_2,
\end{eqnarray*}
for some $\boldsymbol{\alpha}$-independent positive constants $C_1$ and $C_2$. Exploiting the equivalence of norms in finite dimensional Hilbert spaces we conclude the existence of (new) $\boldsymbol{\alpha}$-independent positive constants such that
$\Re( g(\boldsymbol{\alpha}) \cdot \boldsymbol{\alpha} ) \ge C_3 |\boldsymbol{\alpha}| \left( |\boldsymbol{\alpha}| - C_1 \right) - C_2$. Hence, for all sufficiently large $|\boldsymbol{\alpha}|$ we have $\Re( g(\boldsymbol{\alpha}) \cdot \boldsymbol{\alpha} ) > 0$ and therefore with the Browder fixed point theorem the existence of at least one solution $u^{\CN n+1}\LOD$ to \eqref{Crank}.
\end{proof}

\end{document}